\documentclass[reqno]{amsart}
\usepackage{amsmath, amsthm, amssymb, color}
\usepackage{graphicx}
\usepackage{mathrsfs}

\usepackage{latexsym}
\usepackage{hyperref}
\usepackage{xcolor}

\usepackage{soul}

\newtheorem{theorem}{Theorem}[section]

\newtheorem{lemma}[theorem]{Lemma}
\newtheorem{proposition}[theorem]{Proposition}

\newtheorem{definition}[theorem]{Definition}
\numberwithin{equation}{section}
\newtheorem{Theorem}{Theorem}[section]

\theoremstyle{definition}

\newtheorem{remark}[Theorem]{Remark}

\newcounter{RomanNumber}

\def\dive{\hbox{\rm div}\,}

\def\be{\begin{equation}}
\def\en{\end{equation}}
\def\bs{\begin{split}}
\def\es{\end{split}}

\newcommand{\var}{\varepsilon}

\makeatletter
\@namedef{subjclassname@2020}{\textup{2020} Mathematics Subject Classification}
\makeatother

\allowdisplaybreaks[4]

\title[Non-conservative compressible two-fluid model]{The non-conservative compressible two-fluid system with common pressure: Global existence and sharp time asymptotics}
\author{Ling-Yun Shou}
\address{Ling-Yun Shou \newline School of Mathematical Sciences and Mathematical Institute, Nanjing Normal University, Nanjing 210023, China}
\email{shoulingyun11@gmail.com }
\author{Jiayan Wu}
\address{Jiayan Wu \newline School of Mathematics, South China University of Technology, Guangzhou 510641, China}
\email{wujiayan@scut.edu.cn }
\author{Lei Yao}
\address{Lei Yao \newline  School of Mathematics and Statistics, Northwestern Polytechnical University, Xi'an 710129, P.R. China.}
\email{yaolei1056@hotmail.com}
\author{Yinghui Zhang*}
\address{Yinghui Zhang \newline School of Mathematics and Statistics, Guangxi Normal University, Guilin, Guangxi 541004, P.R.
China} \email{yinghuizhang@mailbox.gxnu.edu.cn}

\subjclass[2020]{76T10;\, 76N10.}
\thanks{Corresponding author: Yinghui Zhang}
\keywords{Non-conservative compressible two-fluid;\, Partial dissipation;\, Global existence;\, Optimal decay;\, Littewood-Paley decomposition.}\bigbreak

\date{\today}
\usepackage{hyperref}

\begin{document}
\begin{abstract}
This paper concerns the global-in-time evolution of a generic compressible two-fluid model in $\mathbb{R}^d$ ($d\geq3$) with the common pressure law. Due to the non-dissipative 
properties for densities and two different particle paths caused by velocities, the system lacks the usual symmetry structure and is partially dissipative in the sense that the 
Shizuta-Kawashima condition is violated, which makes it challenging to study its large-time stability. By developing a pure energy method in the framework of Besov spaces, we succeed in constructing a 
unique global classical solution to the Cauchy problem when the initial data are close to their constant equilibria. Compared to the previous related works, the main novelty lies in that our method is independent
of the spectral analysis and does not rely on the $L^1$ smallness of the initial data. 
Furthermore, if additionally the initial perturbation is bounded in 
$\dot{B}^{\sigma_0}_{2,\infty}$ type spaces with lower regularity, the optimal time convergence rates are also obtained. In particular, the asymptotic convergence of the non-dissipative 
components toward their equilibrium states is first characterized.

\end{abstract}

\maketitle

\section{\texorpdfstring{\leftline {\bf{Introduction.}}}{Introduction}}
\setcounter{equation}{0}
\subsection{Background and motivation}
Multi-fluids are not only prevalent in nature but also play an important role in various industrial applications, such as nuclear power, chemical processing, oil and gas manufacturing. 
Although multi-fluid systems share certain similarities with single-fluid models, they exhibit strong interactions and possess more complex structures, leading to rich and distinct 
phenomena and posing significant challenges to the study of their dynamics.

In this paper, we consider the following two-fluid system modeling compressible
fluids in $\mathbb{R}^{d}$ ($d\geq3$):
\begin{equation}\label{system1}
	\left\{
	\begin{array}{ll}
		\partial_t (\alpha^{+}\rho^{+}) +\dive (\alpha^{+}\rho^{+}u^{+})=0,\\
		\partial_t (\alpha^{+}\rho^{+}u^{+}) +\dive (\alpha^{+}\rho^{+}u^{+}\otimes u^{+})+\alpha^{+}\nabla P^{+}(\rho^{+})=\dive (\alpha^{+}\tau^{+}),\\
        	\partial_t (\alpha^{-}\rho^{-}) +\dive (\alpha^{-}\rho^{-}u^{-})=0,\\
		\partial_t (\alpha^{-}\rho^{-}u^{-}) +\dive (\alpha^{-}\rho^{-}u^{-}\otimes u^{-})+\alpha^{-}\nabla P^{-}(\rho^{-})=\dive (\alpha^{-}\tau^{-}),
	\end{array}
	\right.
\end{equation}
where $\alpha^{+}\in [0,1] $ and $\alpha^{-}\in [0,1] $ are the volume fractions of the fluids $+$ and $-$, respectively, satisfying $\alpha^++\alpha^-=1$; $\rho^{\pm}(x,t)\geq 0$, $u^{\pm}(x,t)$ and $P^{\pm}(\rho^{\pm})$
denote the density, velocity and pressure function of each phase, respectively. 
Moreover, $\tau^{\pm}$ is the viscous stress tensor
\begin{align}
	\tau^{\pm}:=\mu^{\pm}\left(\nabla u^{\pm}+\nabla^t u^{\pm}\right)+\lambda^{\pm}\dive u^{\pm} {\rm Id},\label{system2}
\end{align}
where the constants $\mu^{\pm}$ and $\lambda^{\pm}$ are shear and bulk viscosity coefficients satisfying the physical restrictions
\begin{align}
	\mu^{\pm}>0,\quad \quad 2\mu^{\pm}+d\lambda^{\pm}\geq 0.\label{system2-2}
\end{align}
The pressure function $P^{\pm}$ is assumed to be the general term
\begin{align}
	P^{\pm}(s)\in C^{\infty}(\mathbb{R}^{+}),\quad \partial_{s} P^{\pm}(s)>0\quad\text{for}\quad s>0.\label{system3}
\end{align}
subject to the condition
\begin{align}
	P=P^{+}(\rho^{+})=P^{-}(\rho^{-}).\label{pressure}
\end{align}

We are concerned with the Cauchy problem of the system \eqref{system1} subject to the initial data
\begin{align}
	&(\alpha^{\pm},\rho^{\pm}, u^{\pm})(x,0)=(\alpha^{\pm}_0,\rho^{\pm}_0, u^{\pm}_0)(x)\rightarrow (\bar{\alpha}^{\pm},\bar{\rho}^{\pm},0),\quad |x|\rightarrow \infty,\label{system4}
\end{align}
We also give the initial value of $P$:
\begin{align}
P(x,0)=P_0(x)\quad\text{with}\quad P_0:=P^+(\rho_0^+)=P^-(\rho_0^-).\label{system6}
\end{align}
Here, $\bar{\alpha}^{\pm}$ and $\bar{\rho}^{\pm}$ are positive constants fulfilling
\begin{align}
	\bar{P}=\bar{P}^+=\bar{P}^{-}\quad \quad\text{with}\quad \quad \bar{P}^\pm:=P^{\pm}(\bar{\rho}^{\pm}).\label{system5}
\end{align}

The system \eqref{system1} is known as a
two-fluid flow model with algebraic closure. For more information
about this model, we refer to \cite{Bresch1, Bresch2, Evje9, Ishii1,
Prosperetti} and references therein. Particularly, a nice summary of the equations  \eqref{system1} was given in the introduction of Bresch--Desjardins
--Ghidaglia--Grenier \cite{Bresch1}, where they considered the case of density--dependent viscosities and capillary effects. More specifically, they made the following assumptions:
\begin{itemize}
\item inclusion of density--dependent viscosities of the form \eqref{system2-2} where $\mu^{\pm}(\rho)=\nu^{\pm}\rho^{\pm}$
and $\lambda^{\pm}(\rho^\pm)=0$;

\item inclusion of a third order derivative of $\alpha^{\pm}
\rho^{\pm}$, which are so-called  internal capillary forces
represented by the well-known Korteweg model on each phase.
\end{itemize}
When $P^\pm(\rho^\pm)=A^\pm (\rho^\pm)^{\overline{\gamma}^{\pm}}$ with $ A^\pm >0$ and $1<\overline{\gamma}^{\pm}< 6$, they established global existence of weak solutions in the 3D periodic domain.
Later, Cui--Wang--Yao--Zhu \cite{c1} combined a detailed analysis
of the Green function to the linearized system with delicate
energy estimates to the nonlinear system
to prove time-decay rates of classical solutions for
the three-dimensional Cauchy problem when the initial data are near their
equilibria. Recently, Wu-Yao-Zhang \cite{WYZ1} investigated the stability and instability of the constant equilibrium state for the three-dimensional Cauchy problem.
It should be remarked that the internal
capillary forces played an essential role in the analysis of
\cite{Bresch1, c1, WYZ1}. On the other hand, Evje--Wang--Wen \cite{Evje9} studied the
two--fluid model \eqref{system1} with unequal pressures. More precisely, the following assumptions on pressures are made:
\begin{equation}P^+(\rho^+)-P^-(\rho^-)=(\rho^+)^{\bar{\gamma}^+}-(\rho^-)^{\bar{\gamma}^-}=f(\alpha_-\rho_-),\label{1.3}
\end{equation}
where $f$ is the so--called capillary pressure which belongs to $C^3([0,
\infty))$, and is a strictly decreasing function near the
equilibrium satisfying
\begin{equation}-\frac{s_{-}^{2}(1,1)}{\alpha^{-}(1,1)}<f^{\prime}(1)<\frac{\bar{\eta}-s_{-}^{2}(1,1)}{\alpha^{-}(1,1)}<0,\label{1.4}
\end{equation}
where $\bar{\eta}$ is a positive, small fixed constant, and
$s_{\pm}^{2}:=\frac{\mathrm{d} P^{\pm}}{\mathrm{d}
\rho^{\pm}}\left(\rho^{+}\right)$  represents
sound speed of each phase. Under the assumptions
\eqref{1.3} and \eqref{1.4} on pressures, they obtained global
existence and decay rates of strong solutions when the initial data are close to their
equilibria.  Based on results of \cite{Evje9}, Wang--Wen--Yao \cite{WWY} proved global well--posedness and exponential stability for the initial boundary value problem. However, as indicated by
the authors in \cite{Evje9} and \cite{WWY}, assumptions \eqref{1.3} and
\eqref{1.4} played a key role in their analysis and appeared to
have an essential stabilization effect on the model in question. For the case that the capillary pressure is a
strictly increasing function near equilibrium, namely $f^{\prime}(1)>0$, the authors in \cite{WYZ}
showed that the constant equilibrium state of the two--fluid model \eqref{system1} with unequal pressures
is globally linearly unstable and locally nonlinearly unstable in the sense of
Hadamard.
For the
two--fluid model \eqref{system1} with common pressure (i.e. $f\equiv0$) and without capillary effects,
Bresch--Huang--Li
\cite{Bresch2} obtained global existence of weak solutions and studied the vanishing of vacuum states for general initial data in
the spatial real line when
$\overline{\gamma}^{\pm}>1$ by taking advantage of the fine embedding properties in  the one space dimension. However, it is difficult to study the global large-data weak solutions for the 
multi-dimensional problem due to the lack of integrability of densities.
Recently, by exploiting the dissipation structure of the
model and making full use of several key observations, the authors in \cite{WYZ-MA} first proved
global existence and large-time behavior of classical solutions to
the 3D compressible two--fluid model with common pressure and without capillary effects when
$L^1\cap H^3$-norm of the initial perturbation is sufficiently small.
Very recently, based on this work, the authors in \cite{TWYZ} established the validity
of the vanishing capillary limit.

In conclusion, due to the partially dissipation structure and the non-symmetric property of the system \eqref{system1}, up to now, 
the global well-posedness result of \cite{WYZ-MA} is the only one work on the system \eqref{system1}.
However, it should be mentioned that the analysis of \cite{WYZ-MA} relies heavily on the $L^1$ smallness of the initial perturbation.
Indeed, the $L^1$ smallness assumption
ensures sufficiently fast decay rates of the
velocities, which play a key role in controlling the nonlinear terms involving non-dissipative components, such as $\alpha^{\pm}\rho^\pm$. 
To achieve these decay rates and enclose uniform {\rm a priori} estimates, the authors \cite{WYZ-MA} combine energy methods with spectral analysis. Therefore,  a natural and important
problem is to study what will happen about the global existence and large-time stability of solutions for the system \eqref{system1}
when the initial data belong only to $L^p$ with some $1<p\leq 2$.
The main purpose of this paper is to
give a positive answer to this issue.
More precisely, our main goal is to show the global existence and large-time stability of solutions for the system \eqref{system1} in $\mathbb{R}^d$ ($d\geq3$) supplemented with initial
data close to a constant equilibrium in a “minimal” regularity space based on the Littlewood-Paley theorem. Our method is {\emph{independent of the spectral analysis and does not rely on 
the $L^1$ smallness assumption}}. Particularly, our main results can be summarized as follows:

\begin{itemize}

\item When the initial perturbation is close to the equilibrium in $\dot{B}^{d/2-2}_{2,1}\cap\dot{B}^{d/2+1}_{2,1}$ (in the case $d=3$, this is weaker than $L^p(\mathbb{R}^3)\cap 
    H^s(\mathbb{R}^3)$ with $1<p<3/2$ and $s=5/2$), we construct a unique global solution to the Cauchy problem for \eqref{system1} and derive some qualitative $L^1$ time integrability 
    estimates of the dissipative components $(P-\bar{P},u^+,u^-)$ via a pure energy argument.

\item Under the additional assumption that the initial perturbation is bounded in $\dot{B}^{\sigma_0}_{2,1}$ with $-d/2\leq \sigma_0<d/2-2$, which embeds into $L^p(\mathbb{R}^3)$ ($1\leq 
    p<3/2$) for $d=3$, we establish optimal time convergence rates of the global solution to its equilibrium.

\item When the $\dot{B}^{\sigma_0}_{2,\infty}$-norm of the initial perturbation is sufficiently small, we derive higher-order optimal time convergence rates. In addition, we prove the 
    asymptotic stability of the profile for the non-dissipative component $\alpha^\pm\rho^\pm$, which is totally new as compared to \cite{WYZ-MA}.

\end{itemize}

\subsection{Main results}

We define the energy functional
\begin{equation}\label{energy-functional1}
	\begin{aligned}		
\mathcal{E}(t)=&\|\alpha^{+}\rho^+-\bar{\alpha}^{+}\bar{\rho}^+\|_{\widetilde{L}^{\infty}_t(\dot{B}_{2,1}^{\frac{d}{2}-1}\cap\dot{B}_{2,1}^{\frac{d}{2}+1})}+\|P-\bar{P}\|_{\widetilde{L}^{\infty}_t(\dot{B}_{2,1}^{\frac{d}{2}-2}\cap\dot{B}^{\frac{d}{2}+1}_{2,1})}\\
&\quad\quad\quad\quad\quad\quad\quad\quad\quad\quad\quad\quad\quad\quad\quad\quad+\|(u^{+},u^-)\|_{\widetilde{L}^{\infty}_t(\dot{B}_{2,1}^{\frac{d}{2}-2}\cap\dot{B}^{\frac{d}{2}+1}_{2,1})},
	\end{aligned}
\end{equation}
and the corresponding dissipation functional
\begin{align}\label{energy-functional2}
	\mathcal{D}(t)=&\|P-\bar{P}\|_{L^1_t(\dot{B}_{2,1}^{\frac{d}{2}})}+\|(u^{+},u^-)\|_{L^1_t(\dot{B}^{\frac{d}{2}}_{2,1}\cap \dot{B}^{\frac{d}{2}+1}_{2,1})\cap 
\widetilde{L}^2_t(\dot{B}^{\frac{d}{2}+2}_{2,1})}+\|\dive v\|_{L^1_t(\dot{B}^{\frac{d}{2}+1}_{2,1})},
\end{align}
where $v$ is the two-phase effective velocity
\begin{align}
	v=(2\mu^++\lambda^+)u^+-(2\mu^{-}+\lambda^-)u^-.\label{effective}
\end{align}
Accordingly, the initial energy functional $\mathcal{E}_0$ is given by
\begin{equation}\label{initialenergy}
	\begin{aligned}
\mathcal{E}_0=&\|\alpha_0^{+}\rho_0^+-\bar{\alpha}^{+}\bar{\rho}^+\|_{\dot{B}_{2,1}^{\frac{d}{2}-1}\cap\dot{B}^{\frac{d}{2}+1}_{2,1}}+\|P_0-\bar{P}\|_{\dot{B}_{2,1}^{\frac{d}{2}-2}\cap\dot{B}^{\frac{d}{2}}_{2,1}}+\|(u_0^{+},u_0^-)\|_{\dot{B}_{2,1}^{\frac{d}{2}-2}\cap\dot{B}^{\frac{d}{2}+1}_{2,1}}.
	\end{aligned}
\end{equation}
Since $\rho^\pm$ is determined by the common pressure $P$ and $\alpha^\pm$ depends only on $\alpha^+\rho^+$ and $P$, one can recover the information of $(\alpha^\pm,\rho^\pm)$ from $\mathcal{E}(t)$.
Now, we are in a position to state our main results.

\begin{theorem}\label{Thm1.1}
	For any $d\geq 3$, there exists a constant $\varepsilon_0>0$ such that if the initial data fulfill
	\begin{align}\label{a1}
		\mathcal{E}_0\leq \varepsilon_0,
	\end{align}
	then the Cauchy problem \eqref{system1}--\eqref{system5} admits a unique global strong solution $(\alpha^+,\alpha^-, \rho^+,\rho^-, u^+,u^-)$  satisfying
	\begin{equation}\label{r1}
		\left\{
		\begin{aligned}
			&\alpha^{\pm}-\bar{\alpha}^{\pm}\in \mathcal{C}(\mathbb{R}^{+};\dot{B}^{\frac{d}{2}-1}_{2,1}\cap \dot{B}^{\frac{d}{2}+1}_{2,1}),\\
			& (\rho^\pm-\bar{\rho}^\pm,P-\bar{P})\in \mathcal{C}(\mathbb{R}^{+};\dot{B}^{\frac{d}{2}-2}_{2,1}\cap\dot{B}^{\frac{d}{2}+1}_{2,1})\cap L^1(\mathbb{R}^{+};\dot{B}^{\frac{d}{2}}_{2,1}),\\
			& u^{\pm}\in \mathcal{C}(\mathbb{R}^{+};\dot{B}^{\frac{d}{2}-2}_{2,1}\cap\dot{B}^{\frac{d}{2}+1}_{2,1})\cap 
L^1(\mathbb{R}^{+};\dot{B}^{\frac{d}{2}}_{2,1}\cap\dot{B}^{\frac{d}{2}+1}_{2,1})\cap L^2(\mathbb{R}^{+};\dot{B}^{\frac{d}{2}+2}_{2,1}),\\
			&\dive v\in L^1(\mathbb{R}^{+};\dot{B}^{\frac{d}{2}+1}_{2,1}),
		\end{aligned}
		\right.
	\end{equation}
	and
	\begin{align}
		\mathcal{E}(t)+\int_{0}^{t}D(\tau)d\tau\leq C\mathcal{E}_{0}, \quad t\in\mathbb{R}^{+},\label{e1}
	\end{align}
	for $C>0$ a  constant independent of time.
\end{theorem}

\begin{remark}
	The partially  dissipative structure and the non-symmetric property of the system \eqref{system1} lead to main difficulties in establishing uniform {\emph{a priori}} estimates. In order 
to overcome these difficulties, the authors in \cite{WYZ-MA} used the Green functional 
analysis of the linearized system and the two-phase effective velocity $v$ given in \eqref{effective} by resorting to 
the smallness of the perturbation in $L^1(\mathbb{R}^3)\cap H^3(\mathbb{R}^3)$.
Particularly, as mentioned before, the $L^1$ smallness assumption
 played a key role in the analysis of \cite{WYZ-MA}.
 However, since the $L^1$ smallness assumption is absent in our case, it requires us to develop new thoughts. More precisely, 
in Theorem \ref{Thm1.1}, we prove a new global existence result for \eqref{system1} and establish quantitative regularity estimates of solutions
by developing a pure energy method in the framework of Besov spaces.
More specifically, based on the Littlewood-Paley theorem and the use of $\dot{B}^{\sigma}_{2,1}$ type norms, the $L^1$ time integrability {\rm(}dissipation{\rm)} estimates of $u^\pm$, 
$P-\bar{P}$ and $v$ are obtained. 
Therefore, we succeed in removing the $L^1$ smallness assumption
to close the energy estimates by a pure energy method. In fact, in the case $d=3$, one
observes the following chain of embeddings:
	$$
	L^1(\mathbb{R}^3)\cap H^3(\mathbb{R}^3)\hookrightarrow L^{p}(\mathbb{R}^3)\cap H^s(\mathbb{R}^3)(1<p<\frac{3}{2}, ~s>\frac{5}{2})\hookrightarrow \dot{B}^{-\frac{1}{2}}_{2,1}\cap 
\dot{B}^{\frac{5}{2}}_{2,1}.
	$$ 
\end{remark}

\begin{remark}
	Note that the two-phase effective velocity $v$ processes a  dissipation regularity: $\dive v \in L^1(\mathbb{R}_+;\dot{B}^{d/2+1}_{2,1})$, which is higher than the 
$L^1(\mathbb{R}_+;\dot{B}^{d/2-1}_{2,1}\cap\dot{B}^{d/2}_{2,1})$ regularity of $\dive u^\pm$ and plays a key role in our analysis. In addition, we observe that  a {\emph{regularity-loss}} 
phenomenon appears in the dissipative structure, i.e. the $\dot{B}^{d/2+1}_{2,1}$ regularity of the pressure $P-\bar{P}$ can be propagated globally in time; however, its dissipation 
regularity $L^1(\mathbb{R}_+;\dot{B}^{d/2}_{2,1})$ exhibits a lower order index $d/2$.
\end{remark}

Moreover, we have the optimal decay rates of the global solution to the Cauchy problem \eqref{equation_R_P_u_pm} if the low-frequency part of the initial data is further bounded in $\dot{B}^{\sigma_0}_{2,\infty}$ for $\sigma_0\in[-d/2,d/2-2)$ as follows:
\begin{theorem}\label{Thm1.3}
	For any $d\geq 3$, let the assumptions of Theorem \ref{Thm1.1} hold, and $(\alpha^+,\alpha^-, \rho^+,\rho^-, u^+,u^-)$ be the corresponding global strong solution to the Cauchy problem 
\eqref{equation_R_P_u_pm} given by Theorem \ref{Thm1.1}. If the initial data further satisfy its low frequency part
	\begin{align}
		(P_0-\bar{P}, u_0^{+},u_0^-)^{\ell}\in \dot{B}^{\sigma_0}_{2,\infty}, \quad\mbox{ for }\quad \sigma_0\in\left[-\frac{d}{2},\frac{d}{2}-2\right),\label{a2}
	\end{align}
	then it holds for any $t>1$ that
	\begin{align}
		\|(P-\bar{P},u^{+},u^-)(t)\|_{ \dot{B}_{2,1}^{\sigma }}&\leq C\delta_0 (1+t)^{-\frac{1}{2}(\sigma-\sigma_0)},\quad\quad \sigma\in\left(\sigma_0,\frac{d}{2}-2\right],\label{decay1}\\
		\|(P-\bar{P})(t)\|_{ \dot{B}_{2,1}^{\sigma }}&\leq C\delta_0 (1+t)^{-\frac{1}{2}(\frac{d}{2}-2-\sigma_0)},\quad \sigma\in\left(\frac{d}{2}-2,\frac{d}{2}\right],\label{decay2}\\
		\|(u^{+},u^-)(t)\|_{ \dot{B}_{2,1}^{\sigma }}&\leq C\delta_0 (1+t)^{-\frac{1}{2}(\frac{d}{2}-2-\sigma_0)},\quad \sigma\in\left(\frac{d}{2}-2,\frac{d}{2}+1\right]\label{decay3}
	\end{align}
	with a constant $C>0$ independent of time and $\delta_0:=\|(P_0-\bar{P}, u_0^{+},u_0^-)^{\ell}\|_{\dot{B}^{\sigma_0}_{2,\infty}}+ \mathcal{E}_0$.
	
\end{theorem}

When the smallness of the low-frequency part of the initial data in $\dot{B}^{\sigma_0}_{2,\infty}$ is proposed, we can also get the optimal time-decay rates of the global solution to the Cauchy problem.
\begin{theorem}\label{Thm1.4}
	For any $d\geq 3$, let the assumptions of Theorem \ref{Thm1.1} hold, and $(\alpha^+,\alpha^-, \rho^+,\rho^-, u^+,u^-)$ be the corresponding global strong solution to the Cauchy problem 
\eqref{equation_R_P_u_pm} given by Theorem \ref{Thm1.1}.  There exists a constant $\var_1$ such that if 
	\begin{align}
		\delta_1:=\| \alpha_0^{+}\rho_0^+-\bar{\alpha}^+\bar{\rho}_0^+\|_{\dot{B}^{\sigma_0+1}_{2,\infty}}+ \|(P_0-\bar{P}, u_0^{+},u_0^-)^{\ell}\|_{\dot{B}^{\sigma_0}_{2,\infty}}\leq\var_1, \quad\mbox{for 
}\quad \sigma_0\in\left[-\frac{d}{2},\frac{d}{2}\right),\label{smalla2}
	\end{align}
	then it holds for any $t>1$ and $0<\var<<\frac{d}{2}-2-\sigma_0$ that
	\begin{align}
		\|(P-\bar{P},u^{+},u^-)(t)\|_{ \dot{B}_{2,1}^{\sigma }}&\leq C\delta_1 (1+t)^{-\frac{1}{2}(\sigma-\sigma_0)},\quad \sigma\in\left(\sigma_0,\frac{d}{2}-\var \right],\label{decay4}\\
		\|(P-\bar{P})(t)\|_{ \dot{B}_{2,1}^{\frac{d}{2} }}+\|(u^{+},u^-)\|_{\dot{B}_{2,1}^{\frac{d}{2}} \cap \dot{B}_{2,1}^{\frac{d}{2}+1 }}&\leq C\delta_1 (1+t)^{-\frac{1}{2}(\frac{d}{2} 
-\sigma_0-\var)},\label{decay5}\\
		\|(\alpha^{+}\rho^{+}-R_{\infty}^+,\alpha^{-}\rho^{-}-R_{\infty}^-)(t)\|_{\dot{B}^{\frac{d}{2}}_{2,1}}&\leq C\delta_1(1+t)^{-\frac{1}{2}(\frac{d}{2}-2 -\sigma_0-\var)},\label{decay6}
	\end{align}
	with a constant $C>0$ independent of time and the profile $R_{\infty}^\pm=R_{\infty}^\pm(x)$ is given by
	\begin{align}\label{barrho}
		R_{\infty}^\pm(x):=\alpha_0^{\pm}\rho_0^{\pm}(x)-\int_0^{\infty} \dive (\alpha^{\pm}\rho^{\pm} u^{\pm})(t,x)\,dt.
	\end{align}
\end{theorem}

\begin{remark}
	Due to the embedding $\dot{B}^{\sigma_0}_{2,\infty} \hookrightarrow L^q(\mathbb{R}^d)$ with $\sigma_0\in [-d/2,d/2-2) $ and $q=2d/(d-2\sigma_{0})\in [1,d/2)$ Theorems \ref{Thm1.3}--\ref{Thm1.4} provide optimal time convergence rates for more general initial data. In fact, the assumption $\dot{B}^{\sigma_0}_{2,\infty}$ seems to be sharp {\rm(}see 
\cite{brandolese1,brandolese2}{\rm)}. In particular, in the case $\sigma_0=-d/2$, by \eqref{decay4}--\eqref{decay6} and an interpolation argument, we can recover the classical optimal time 
convergence estimates in Sobolev spaces subject to the $L^1$ assumption:
	\begin{equation*}
		\begin{aligned}
			\|\Lambda^{\sigma}(P-\bar{P},u^+,u^-)(t)\|_{L^p}\lesssim (1+t)^{-\frac{1}{2}(\sigma+d-\frac{d}{p})},\quad\quad -\frac{d}{2}<\sigma+\frac{d}{2}-\frac{d}{p}<\frac{d}{2},\quad 
p\geq2,
		\end{aligned}
	\end{equation*}
	and the almost optimal time convergence rate in the $L^{\infty}$ sense:
	\begin{equation*}
		\begin{aligned}
			\|(P-\bar{P},u^+,u^-)(t)\|_{L^{\infty}}&\lesssim (1+t)^{-\frac{d}{2}+\var},\\
			\|(\alpha^{\pm}\rho^{\pm}-R_{\infty}^\pm)(t)\|_{L^{\infty}}&\lesssim  (1+t)^{-\frac{d}{2}+1+\var}
		\end{aligned}
	\end{equation*}
	for any $0<\var<<1$.
\end{remark}

\begin{remark}
	Although $\alpha^{\pm}\rho^{\pm}$ has no dissipation effect around the constant equilibrium state $1$, inspired by the work of Danchin \cite{danchinpressure} on the large-time behavior of viscous pressureless flows, we establish the convergence rate $(1+t)^{-\frac{1}{2}(d/2-2 -\sigma_0-\var)}$ of $\alpha^{\pm}\rho^{\pm}$ to the profile $R_{\infty}^\pm$.
	The shape and asymptotic stability of the profiles $R^{\pm}_{\infty}$ are new in the study of the two-fluid model \eqref{system1}.
\end{remark}

\subsection{Reformulation}

Owing to the relation between the pressures of $(\ref{system1})_3,$ one has
\begin{align}
	{\rm d}P=s_+^2 {\rm d}\rho^+=s_-^2{\rm d} \rho^-,\label{1:1}
\end{align}
where $s_{+}^2=\frac{dP^{+}}{d\rho^{+}}(\rho^{+})>0$ and $s_{-}^2=\frac{dP^{-}}{d\rho^{-}}(\rho^{-})>0$, respectively, represent the sound speed of the phases $+$ and $-$. Following \cite{Bresch1}, we use the notation of fraction density $R^\pm$, which is given by
\begin{align}
R^{\pm}:=\alpha^{\pm}\rho^{\pm}.\label{1:2}
\end{align}
Together with $\alpha^++\alpha^-=1$, this implies that
\begin{align}
{\rm d} \rho^+=\frac{1}{\alpha^+} ({\rm d} R^+-\rho^+ {\rm d} \alpha^+),\quad  {\rm d} \rho^-=\frac{1}{\alpha^-} ({\rm d} R^-+\rho^- {\rm d} \alpha^-).\label{1:3}
\end{align}
In view of \eqref{1:1} and \eqref{1:2}, one also has
\begin{align}
{\rm d} \alpha^+=\frac{\alpha^- s_+^2}{\alpha^-\rho^+ s_+^2+\alpha^+\rho^- s_-^2} {\rm d} R^+-\frac{\alpha^+ s_-^2}{\alpha^- \rho^- s_+^2+\alpha^+ \rho^- s_-^2} {\rm d} R^-.\label{1:4}
\end{align}
Consequently, substituting \eqref{1:4} into \eqref{1:1} leads to
\begin{align}
{\rm d} \rho^\pm=\frac{\rho^+\rho^- s_{\pm}^2}{R^-(\rho^+)^2 s_+^2+ R^+(\rho^-)^2 s_-^2} ( \rho^-{\rm d} R^++ \rho^+ {\rm d} R^-).\label{1:5}
\end{align}
In addition, owing to \eqref{1:1} and \eqref{1:5}, one can recover the common pressure differential ${\rm d} P$ as
\begin{align}
	{\rm d} P=\mathcal{C}(\rho^- {\rm d}\,R^++\rho^+ {\rm d}\, R^-),\label{1:6}
\end{align}
where 
\begin{equation}
	\begin{aligned}
		\mathcal{C}:=\frac{s_-^2s_+^2}{\alpha^-\rho^+s_+^2+\alpha^+\rho^-s_-^2}.\label{1:7}
	\end{aligned}
\end{equation}
Since $\alpha^{+}+\alpha^{-}=1$, we know that
\begin{align}
	\frac{R^+}{\rho^+}+\frac{R^-}{\rho^-}=1,\label{1:8}
\end{align}
which implies
\begin{align}
	R^{-}=\frac{\rho^-(\rho^+-R^+)}{\rho^+}=\frac{\pi^-(P)(\pi^+(P)-R^+)}{\pi^+(P)}.\label{1:9}
\end{align}
Due to the inverse function theorem, $\pi^{\pm}(P)$ exists for all $P>0$ such that $\rho^{\pm}=\pi^{\pm}(P)$. Using \eqref{pressure} and \eqref{1:2}, we can represent $\alpha^+$ as 
\begin{equation}\label{equation_R_P_u_pm00}
	\begin{array}{ll}
		\alpha^+=\alpha^+(R^+,P)=\frac{R^+}{\pi^{+}(P)},\quad \alpha^-=\alpha^-(R^+,P)= 1-\frac{R^+}{\pi^{+}(P)}.
	\end{array}
\end{equation}
Therefore, one can recover the desired bounds for $\alpha^\pm$, $\rho^\pm$ and $R^-$ once we establish the estimates of $R^+$ and $P$. In particular, we know $\mathcal{C}=\mathcal{C}(R^+,P)$.

On the other hand, one also can represent  $\alpha^\pm$, $\rho^\pm$ and $P$ in terms of $(R^+,R^-)$. Indeed, note that \eqref{pressure} and \eqref{1:9} ensure that
\begin{align}
\Psi(\rho^+, R^+,R^-):=P^+(\rho^+)-P^-\Big(\frac{R_-\rho^-}{\rho^+-R^+}\Big)=0.\label{1.1.7}
\end{align}
Differentiating the above equation with respect to $\rho^+$ gives
\begin{align*}
\partial_{\rho_{+}}\Psi(\rho^+, R^+,R^-):=s_+^2+s_-^2 \frac{R_-R_+}{(\rho^+-R^+})^2>0.
\end{align*}
As a consequence of the implicit function theorem, $\rho^+=\rho^+(R^+,R^-)$ is the unique solution of the ODE \eqref{1.1.7}. Thus, it holds that
\begin{equation}\label{given}
\left\{
\begin{aligned}
&\rho^+=\rho^+(R^+,R^-),\\
&\rho^-=\rho^-(R^+,R^-)=\frac{R^-\rho^+(R^+,R^-)}{\rho^+(R^+,R^-)},\\
&P=P(R^+,R^-)=P^+(\rho^+(R^+,R^-)),\\
&\alpha^+=\alpha^+(R^+,R^-)=\frac{R^+}{\rho^+(R^+,R^-)},\\
&\alpha^-=\alpha^-(R^+,R^-)=1-\alpha^-(R^+,R^-).
\end{aligned}
\right.
\end{equation}
Thus, we can also recover the bounds for $\alpha^\pm$, $\rho^\pm$ and $P$ once the estimates of $(R^+,\theta)$ have been obtained.

Next, we give some reformulations of the original system \eqref{system1}. For brevity, in what follows we set the background doping profile
$$
\bar{\alpha}^{\pm} \bar{\rho}^{\pm}=1.
$$
Denote
\begin{align}
	n^{\pm}=R^{\pm}-1, \quad  \theta=P-\bar{P},
\end{align}
and its initial values
\begin{align}
n_0^{\pm}=\alpha_0^\pm\rho_0^\pm-1 , \quad  \theta_0=P_0-\bar{P}.
\end{align}
Then, the system \eqref{system1} can be written as 
\begin{equation}\label{equation-for-new-term}
	\left\{
	\begin{array}{ll}
		\partial_t n^++ u^+\cdot \nabla n^+ +(1+n^+)\dive u^+=0,\\
		\partial_t n^- +u^-\cdot \nabla n^- +(1+n^-)\dive u^-=0,\\
		\partial_t u^+ + u^+\cdot\nabla u^++\tfrac{M_1}{\rho^+}\nabla n^++\tfrac{M_2}{\rho^+}\nabla n^- - \nu_1^+\Delta u^+-\nu_2^+ \nabla\dive u^+=H^+,\\
		\partial_t u^- + u^-\cdot\nabla u^-+\tfrac{M_1}{\rho^-}\nabla n^++\tfrac{M_2}{\rho^-}\nabla n^-- \nu_1^-\Delta u^--\nu_2^-\nabla\dive u^-=H^-,\\
		( n^+,n^-, u^+,u^-)|_{t=0}=(n_0^+,n^-,u_0^+,u_0^-),
	\end{array}
	\right.
\end{equation}
where we used the notations
\begin{align}
\nu_1^\pm=\frac{\mu^\pm}{\bar{\rho}^\pm},\quad\quad \nu_2^\pm=\frac{\mu^\pm+\lambda^\pm}{\bar{\rho}^\pm},
\end{align}
and
\begin{equation}\label{1.43}
	\left\{
	\begin{aligned}
    &M_1=\mathcal{C}\rho^-,\quad\quad\quad \quad\quad\quad\quad  M_2=\mathcal{C}\rho^+,\\
	&	H^\pm=  g_1^\pm(\theta)\Delta u^\pm+g_2^{\pm}(\theta)\dive u^\pm+ \frac{\mu^\pm(\nabla u^\pm+\nabla^t u^\pm) \cdot\nabla \alpha^\pm}{R^+}+\frac{\lambda^\pm\dive u^\pm\nabla 
\alpha^\pm}{R^\pm},\\
    & g_1^\pm(\theta)= \frac{\mu^\pm}{\rho^\pm}-\nu_1^\pm,\quad\quad\quad\quad g_2^\pm(\theta)= \frac{\mu^\pm+\lambda^\pm}{\rho^\pm}-\nu_2^
   \pm.
	\end{aligned}
	\right.
\end{equation}

Note that both $n^+$ and $n^-$ have no dissipation; however, $\theta=P-\bar{P}$ can be viewed as a dissipative variable. To capture the dissipation of $(\theta,u^+,u^-)$, we use \eqref{1:6} and rewrite the system \eqref{system1} as follows:
\begin{equation}\label{equation_R_P_u_pm0}
	\left\{
	\begin{array}{ll}
		\partial_t R^+ +\dive (R^+u^+)=0,\\
		\partial_t P +\mathcal{C}\rho^{-}\dive (R^+u^+)+\mathcal{C}\rho^{+}\dive (R^-u^-)=0,\\
		\partial_t (R^+u^+) +\dive (R^+u^+\otimes u^+)+\alpha^+\nabla P\\
		\quad\quad\quad\quad\quad\quad= \dive \Big(\alpha^+\big(\mu^+(\nabla u^++\nabla^{\top} u^+)+\lambda^+\dive u^+ {\rm Id}\big) \Big),\\
		\partial_t (R^-u^-) +\dive (R^-u^-\otimes u^-)+\alpha^-\nabla P\\
		\quad\quad\quad\quad\quad\quad= \dive \Big( \alpha^-\big(\mu^-(\nabla u^-+\nabla^{\top} u^-)+\lambda^-\dive u^- {\rm Id} \big) \Big).\ &
	\end{array}
	\right.
\end{equation}
The linearization of \eqref{equation_R_P_u_pm0} reads
\begin{equation}\label{equation_R_P_u_pm}
	\left\{
	\begin{array}{ll}
		\partial_t n^+ + \dive u^+ =-\dive (n^+u^+),\\
		\partial_t \theta +\beta_1^+\dive u^++\beta_1^-\dive u^-=F_1,\\
		\partial_t u^++\beta_2^+ \nabla \theta- \nu_1^{+}\Delta u^+-\nu_2^+\nabla\dive u^+=F_2^+,\\
		\partial_t u^-+\beta_2^- \nabla \theta- \nu_1^{-}\Delta u^--\nu_2^-\nabla\dive u^-=F_2^-,\\
		( n^+,\theta, u^+, u^-)|_{t=0}=(n_0^+, n_0^-,u_0^+,u_0^-),
	\end{array}
	\right.
\end{equation}
where the terms $F_1$ and $F_2^\pm$ are given by
\begin{equation}\label{1.46}
	\left\{
	\begin{aligned}
		F_{1}&=-g_3^+(n^+,\theta)\dive u^+-g_3^-(n^+,\theta)\dive u^-\\
        &\quad \quad -\mathcal{C}(n^++1,\theta+\bar{P})\rho^- u^+\cdot\nabla n^+-\mathcal{C}(n^++1,\theta+\bar{P})\rho^+ u^-\cdot\nabla n^-,\\
		F_2^\pm&=-u^{\pm}\cdot\nabla u^{\pm}-g_4^{\pm}(\theta)\nabla \theta+g_1^{\pm}(\theta)\Delta u^{\pm}+g_2^{\pm}(\theta)\nabla\dive  u^{\pm} , \\
		&\quad\quad+ \frac{\mu^\pm(\nabla u^\pm+\nabla^{\top} u^\pm) \cdot\nabla \alpha^\pm}{R^\pm}+\frac{\lambda^\pm\dive u^\pm\nabla \alpha^\pm}{R^\pm},
	\end{aligned}
	\right.
\end{equation}
where we defined the constants
\begin{equation}
	\begin{aligned}
    &\beta_1^\pm=\mathcal{C}(1,\bar{P})\bar{\rho}^\pm ,\quad    \beta_2^\pm= \frac{1}{\bar{\rho^\pm}}, 
	\end{aligned}
\end{equation}
and the nonlinear functions $g_i^{\pm}(1\leq i\leq 4)$ 
\begin{equation}
	\left\{
	\begin{aligned}
		& g_3^+(n^+,\theta)=\mathcal{C}(n^++1,\theta+\bar{P})\rho^- R^+-\beta_1^+, \\
		&g_3^-(n^+,\theta)=\mathcal{C}(n^++1,\theta+\bar{P})\rho^+ R^--\beta_1^-, \\
		& g_4^\pm(\theta)=\frac{1}{\rho^\pm}-\beta_{2}^\pm.
	\end{aligned}
	\right.
\end{equation}


\subsection{Difficulties and strategies}

To investigate the global stability of solutions to the two-fluid system \eqref{system1}, one employs techniques derived from partially dissipative hyperbolic-parabolic systems
\begin{align}
	\partial_{t}W+\sum_{i=1}^{d}A^{i}(W)\partial_{x_{i}}W-\sum_{i,j=1}^d\partial_{x_j}\big(B_{i,j}(W)\partial_{x_i}W\big)=Q(W),
	\label{entropeq}
\end{align}
where $W\in \mathbb{R}^n$ ($n\geq2$) is the unknown. These models describe the dynamics of physical systems out of thermodynamic equilibrium, a common scenario in gas dynamics, where the 
dissipative effect from the parabolic term does not affect all components of the unknown  $W\in \mathbb{R}^n$. To promise its global stability, Shizuta and Kawashima \cite{shi1} introduced
a celebrated stability condition, called the Shizuta-Kawashima (SK) condition, to characterize the interactions from the coupling of the matrices $A^i$ and $B_{i,j}$. Yong \cite{yong1} 
formulated the notion of entropy, which provides a symmetrization of $A^i$ and extends these results to more general quasilinear systems. 
Qu and Wang \cite{qu1} established a
global existence result for partially dissipative  symmetric systems in Sobolev spaces, where exactly one eigen-family is allowed to violate the (SK) condition.

However, for the reformulated system \eqref{equation_R_P_u_pm}, the key difficulty lies in the nonlinear terms in the equation $\eqref{equation_R_P_u_pm}_2$ for $\theta$:
\begin{align}
	-\mathcal{C}\rho^- u^+\cdot\nabla n^+\quad\text{and}\quad-\mathcal{C}\rho^+ u^-\cdot\nabla n^-.\label{1.38}
\end{align}
In fact, such nonlinear terms break the symmetry of \eqref{equation_R_P_u_pm} and lead to a one-order regularity loss in high frequencies which cannot be handled by usual commutator 
estimates. Moreover, finding an entropy variable appears to be particularly challenging. Therefore, even though the linear part of the reformulated system \eqref{equation_R_P_u_pm} seems 
to have a similar structure to that in \cite{qu1}, it remains difficult to close the {\emph{a priori}} estimates and establish global existence.

To solve the nonlinear problem globally in time, as emphasized in many classical works \cite{bahouri1,danchinhand}, a key challenge arises from the essential \emph{Lipschitz} requirement 
of the $L^1$ time integrability for the velocity $u^\pm$, i.e.
\begin{align}
	\int_0^{\infty}\|\nabla u^\pm\|_{L^{\infty}}\,dt<\infty.\label{1.39}
\end{align}
This condition represents the minimal regularity control of the convective terms in the equation for $u^\pm$. On the other hand, to derive the upper bound of $\theta$, which is necessary 
for typical nonlinear terms, one has to address the $L^1(\mathbb{R}^{+};L^{\infty})$-bound of the challenging source terms \eqref{1.38}. Consequently, the following key conditions must 
also be satisfied:
\begin{align}
	\int_0^{\infty}\|u^\pm\|_{L^{\infty}}\,dt<\infty,\quad\quad \| R^\pm\|_{L^{\infty}(\mathbb{R}^{+};L^{\infty})}+\| \nabla 
R^\pm\|_{L^{\infty}(\mathbb{R}^{+};L^{\infty})}<\infty.\label{1.40}
\end{align}

\vspace{2mm}

We now outline the main strategies for proving Theorems \ref{Thm1.1}--\ref{Thm1.4} concerning the global existence and large-time asymptotic behavior of strong solutions to 
\eqref{system1}--\eqref{system5} within the Besov regularity framework. In the spirit of hypocoercivity (see \cite{danchinhand}), we construct a localized Lyapunov functional 
$\mathcal{L}_j(t)\sim \|\dot{\Delta}_j(\theta,\nabla\theta,u^{+},u^{-})\|_{L^2}^2$ for $j\in\mathbb{Z}$ such that
\begin{equation}\label{Lj1}
	\begin{aligned}
		\frac{d}{dt}\mathcal{L}_j(t)+\min\{1,2^{2j}\}\mathcal{L}_j(t)+2^{2j}\|\dot{\Delta}_j(u^+,u^-)\|_{L^2}^2\lesssim {\rm nonlinear\, terms}\,\sqrt{\mathcal{L}_j(t)}.
	\end{aligned}
\end{equation}
This reveals the sharp behaviors in different frequency regimes:
\begin{itemize}
	
	\item In the low frequencies $j\leq 0$, $(\theta,\nabla\theta,u^{+},u^{-})$ exhibit the effect of heat diffusion, as expected in spectral analysis (see \cite{WYZ-MA}). To derive the 
bound of $u^\pm$ in \eqref{1.40}, we shall establish the $L^{1}(\mathbb{R}^+;\dot{B}^{d/2}_{2,1})$-regularity of $(\theta,u^+,u^-)$. Consequently, the maximal regularity estimates of the 
heat equation indicate that the initial data have to belong to  $\dot{B}^{d/2-2}_{2,1}$ (see Lemma \ref{lemma2}).
	
	\item In the high frequencies $j\geq -1$,  $(\theta,\nabla\theta,u^{+},u^{-})$ behaves like a damping, while $u^\pm$ obeys a heat diffusion. Thus, the dissipative structure exhibited in 
\eqref{Lj1} enables us to obtain the $ L^{1}(\mathbb{R}^+;\dot{B}^{d/2}_{2,1})$-estimate for $\theta$ and $ L^{1}(\mathbb{R}^+;\dot{B}^{d/2+1}_{2,1})$-estimate of $u^\pm$ required in 
\eqref{1.38} -\eqref{1.39}  (see Lemma \ref{lemma2}).

\end{itemize}

Note that the linear term $\dive u^\pm$ in the equation of $n^\pm$ leads to a zero eigenvalue for the linear system associated with $\eqref{equation_R_P_u_pm}$, which implies that the classical (SK) condition does not hold. The main observation here is that the equation of $u^+$ can be viewed as a transport equation with the source $\dive u^\pm$, so at the $\dot{B}^{d/2-1}_{2,1}\cap \dot{B}^{d/2}_{2,1}$ 
regularity level for $n^\pm$, the above linear term $\dive u^\pm$ can be controlled based on the $L^1(\mathbb{R}^{+};\dot{B}^{d/2}_{2,1}\cap\dot{B}^{d/2+1}_{2,1}))$ bound of $u^\pm$ 
obtained in the above low-frequency and high-frequency analysis. This allows us to establish low-order estimates without using the method of Green's function as in  \cite{WYZ-MA} (see 
Lemma \ref{lemman+}).

Yet, the aforementioned low-frequency and high-frequency estimates are insufficient to derive the Lipschitz bounds of $R^\pm$ in \eqref{1.40}, or stronger $\dot{B}^{d/2+1}_{2,1}$-bound. To 
overcome this difficulty, we perform the nonlinear energy estimates for the system \eqref{equation-for-new-term} (see Lemma \ref{lemma4}). More precisely, we apply $\dot{\Delta}_j$ to 
\eqref{equation-for-new-term} and rewrite the system as a quasilinear system \eqref{equation-for-new-term-local} with some commutator terms, and take advantage of $L^2$ energy estimates 
for $\dot{\Delta}_{j}(n^+,n^-,u^+,u^-)$ with suitably nonlinear weights inspired by techniques from quasilinear symmetric hyperbolic systems (see \cite{burtea1, CB1,CB3, XuKawashima1}). 
However, due to the nonsymmetric part of \eqref{equation-for-new-term}, there are some challenging cross terms
\begin{align}
	\int_{\mathbb{R}^d} \tfrac{M_1}{\rho^-}\dive \dot{\Delta}_ju^-\dot{\Delta}_jn^+\,dx\quad\text{and}\quad \int_{\mathbb{R}^d} \tfrac{M_2}{\rho^+}\dive 
\dot{\Delta}_ju^+\dot{\Delta}_jn^-\,dx,\label{crossd}
\end{align}
arising from the heterogenization of two velocities. To overcome this difficulty, we rewrite the equations $\eqref{equation-for-new-term}_1$--$\eqref{equation-for-new-term}_2$ for $n^+$ and 
$n^-$ in terms of the two-phase effective velocity $v$ defined in \eqref{effective} as follows:
\begin{equation}\label{n+n-ddd}
	\left\{
	\begin{array}{ll}
		\partial_t n^++ u^+\cdot \nabla n^+ +\tfrac{1+n^+}{\nu_1^++\nu_2^+}\dive v+ \tfrac{\nu_1^-+\nu_2^-}{\nu_1^++\nu_2^+}(1+n^+)\dive u^-=0,\\
		\partial_t n^- +u^-\cdot \nabla n^- -\frac{1+n^-}{\nu_1^-+\nu_2^-}\dive v+ \tfrac{\nu_1^++\nu_2^+}{\nu_1^-+\nu_2^-}(1+n^-)\dive u^-=0.
	\end{array}
	\right.
\end{equation}
The above term involving $n^+\dive u^-$ and $n^-\dive u^+$ can be used to cancel the challenging terms \eqref{crossd}. Consequently, we must control the additional terms involving $\dive v$ 
in \eqref{n+n-ddd}, which require higher-order $L^1(\mathbb{R}_+;\dot{B}^{d/2+2}_{2,1})$ estimates. To achieve it, the key observation here is that $v$ satisfies a fine elliptic structure:
\begin{align}
	 -\Delta \dive v=-\dive(\bar{\rho}^+ \partial_t u^++\bar{\rho}^-\partial_t u^-)+{\text{good nonlinear terms}}.
\end{align}
Thus, it suffices to establish the $L^1(\mathbb{R}_+;\dot{B}^{d/2}_{2,1})$-estimate for the time derivatives $u^\pm_t$. To this end, we introduce the effective variables
$$
w^{\pm}=\beta_2^\pm\nabla \theta-\nu_1^{\pm}\Delta u^{\pm}-\nu_2^{\pm}\nabla\dive u^{\pm},
$$
which allows us to express $\partial_t u^\pm= w^\pm$ up to some nonlinear terms. Compared with \cite{WYZ-MA}, this approach allows us to avoid the direct analysis of the initial data for 
$u_t^{\pm}$ which may require compatibility conditions. Since 
 $w^\pm$ satisfy two Lem\'e systems \eqref{equationforw} with sources in $L^1(\mathbb{R}^{+};\dot{B}^{d/2-2}_{2,1})$, we can apply the maximal regularity estimates to derive the 
 $L^1(\mathbb{R}^{+};\dot{B}^{d/2}_{2,1})$-bound for $w^\pm$, which in particular implies the desired estimate for $u_t^{\pm}$ (see Lemma \ref{lemma3}). With these key observations in hand, we can 
 establish the global {\emph{a priori}} estimates and prove the global existence of solutions to the Cauchy problem \eqref{system1}. Based on the global existence result and the localized 
 Lyapunov functional inequality \eqref{Lj1}, we establish some new time-weighted energy estimates (see \eqref{priorestimate-forEtheta} and \eqref{zbound}), which allow us to derive decay 
 rates in Theorems \ref{Thm1.3}--\ref{Thm1.4} provided that the initial perturbation additionally belongs to $\dot{B}^{\sigma_0}_{2,\infty}$ type spaces with lower regularity.

\vspace{1mm}

\vspace{2mm}

The rest of the paper unfolds as follows. In Section \ref{section2}, we explain the notations and briefly recall the Littlewood-Paley decomposition, Besov spaces and Chemin-Lerner spaces. 
Section \ref{section3} is devoted to the global existence and {\emph{a priori}} estimates of solutions for \eqref{system1}--\eqref{system5}. In Sections \ref{sectiondecay1}, we establish
the large-time behavior of global solutions under the additional condition \eqref{a2}, that is, Theorem \ref{Thm1.3}, while the time convergence rates are improved in Section \ref{sectiondecay2} under the stronger condition \eqref{smalla2}. Appendix A collects some useful lemmas for non-standard product laws and the composition of functions that will be used throughout the text.

\section{Preliminaries}\label{section2}

We explain the notations used throughout this paper. $C>0$ and $c>0$ denote two constants independent of time. $A\lesssim B(A\gtrsim B)$ means $A\leq CB$ $(A\geq CB)$, and $A\sim B$ stands for $A\lesssim B$ and $A\gtrsim B$. For any Banach space $X$ and the functions $g,h\in X$, let $\|(g,h)\|_{X}:=\|g\|_{X}+\|h\|_{X}$. For any $T>0$ and $1\leq \varrho\leq\infty$, we denote
by $L^{\varrho}(0,T;X)$ the set of measurable functions $g:[0,T]\rightarrow X$ such that $t\mapsto \|g(t)\|_{X}$ is in $L^{\varrho}(0,T)$ and write $\|\cdot\|_{L^{\varrho}(0,T;X)}:=\|\cdot\|_{L^{\varrho}_{T}(X)}$.

We recall the Littlewood-Paley decomposition, Besov spaces and related analysis tools. The reader can refer to Chapters 2-3 in \cite{bahouri1} for the details. Choose a non-increasing smooth radial function $\chi(\xi)$ compactly supported in $B(0,4/3)$ and satisfying $\chi(\xi)=1$ in $B(0,3/4)$. Then $\varphi(\xi):=\chi(\frac{\xi}{2})-\chi(\xi)$ satisfies
$$
\sum_{j\in \mathbb{Z}}\varphi(2^{-j}\cdot)=1,\quad \text{{\rm{Supp}}}~ \varphi\subset \{\xi\in \mathbb{R}^{d}~|~\frac{3}{4}\leq |\xi|\leq \frac{8}{3}\}.
$$
For any $j\in \mathbb{Z}$, define the homogeneous dyadic blocks $\dot{\Delta}_{j}$ by
$$
\dot{\Delta}_{j}u:=\mathcal{F}^{-1}\big{(} \varphi(2^{-j}\cdot )\mathcal{F}(u) \big{)}=2^{jd}h(2^{j}\cdot)\star u,\quad\quad h:=\mathcal{F}^{-1}\varphi,
$$
where $\mathcal{F}$ and $\mathcal{F}^{-1}$ are the Fourier transform and its inverse. Let $\mathcal{P}$ be the class of all polynomials in $\mathbb{R}^{d}$ and $\mathcal{S}_{h}'=\mathcal{S}'/\mathcal{P}$ stand for the tempered distributions in $\mathbb{R}^{d}$ modulo polynomials. One can get
\begin{equation}\nonumber
	\begin{aligned}
		&u=\sum_{j\in \mathbb{Z}}\dot{\Delta}_{j}u\quad\text{in}~\mathcal{S}_h',\quad \quad \quad \dot{\Delta}_{j}\dot{\Delta}_{l}u=0,\quad\text{if}\quad|j-l|\geq2.
	\end{aligned}
\end{equation}

With the help of those dyadic blocks, we give the definition of homogeneous Besov spaces as follows.

\begin{definition}\label{defnbesov}
	For $d\geq1$, $s\in \mathbb{R}$ and $1\leq p,r\leq \infty$, the  homogeneous Besov norm $\|\cdot\|_{\dot{B}^{s}_{p,r}}$ on $\mathbb{R}^d$ is defined by
	$$
	\|u\|_{\dot{B}^{s}_{p,r}}:=\|\{2^{js}\|\dot{\Delta}_{j}u\|_{L^{p}}\}_{j\in\mathbb{Z}}\|_{l^{r}} .
	$$
	\begin{itemize}
		
		\item $(i)$ If $s<\frac{d}{p}$ {\rm(}or $s
		=\frac{d}{p}, r=1${\rm)}, we denote
		$$
		\dot{B}_{p,r}^{s}:{=}\Big\{u\in\mathcal{S}_{h}^{\prime}(\mathbb{R}^{3}):\|u\|_{\dot{B}_{p,r}^{s}}<\infty\Big\},
		$$
		which is a Banach space.
		
		\item $(ii)$ If $K\in \mathbb{N}$ and if $\frac{d}{p}+k\leq\frac{d}{p}+k+1$ {\rm(}or $s=\frac{p}{3}+k+1, r=1${\rm)}, we denote $\dot{B}_{p,r}^{s}$ as the subset
of $u\in\mathcal{S}_{h}^{\prime}$ such that $\partial^{\beta} u$ belongs to $\dot{B}_{p,r}^{s-k}$ while $|\beta|=k$.
		
		\item $(iii)$ In particular, the Besov space $\dot{B}_{2,2}^{s}$ is consistent with the homogeneous Sobolev space $\dot{H}^{s}$.
		
	\end{itemize}
	
\end{definition}

Next, we state a class of mixed space-time Besov spaces introduced by Chemin-Lerner \cite{chemin1}.
\begin{definition}\label{defntimespace}
	For $T>0$, $s\in\mathbb{R}$ and $1\leq \varrho,r, q \leq \infty$, the space $\widetilde{L}^{\varrho}(0,T;\dot{B}^{s}_{p,r})$ is defined as
	$$
	\widetilde{L}^{\varrho}(0,T;\dot{B}^{s}_{p,r}):= \big{\{} u\in L^{\varrho}(0,T;\mathcal{S}'_{h})~|~ 
\|u\|_{\widetilde{L}^{\varrho}_{T}(\dot{B}^{s}_{p,r})}:=\|\{2^{js}\|\dot{\Delta}_{j}u\|_{L^{\varrho}_{T}(L^{p})}\}_{j\in\mathbb{Z}}\|_{l^{r}}<\infty \big{\}}.
	$$
	 Moreover, we denote
	\begin{equation}\nonumber
		\begin{aligned}
			&\mathcal{C}_{b}(\mathbb{R}^{+};\dot{B}^{s}_{p,r}):=\big{\{} 
u\in\mathcal{C}(\mathbb{R}^{+};\dot{B}^{s}_{p,r})~|~\|f\|_{\widetilde{L}^{\infty}(\mathbb{R}^{+};\dot{B}^{s}_{p,r})}<\infty \big{\}}.
		\end{aligned}
	\end{equation}
\end{definition}

Note that due to Minkowski's inequality, one can compare Chemin-Lerner norms with usual Lebesgue-Besov norms $\|\cdot\|_{L^{\varrho}_{T}(\dot{B}^{s}_{p,r})}$:
	\begin{equation}\nonumber
		\begin{aligned}
			&\|u\|_{L^{\varrho}_{T}(\dot{B}^{s}_{p,r})}\leq \|u\|_{\widetilde{L}^{\varrho}_{T}(\dot{B}^{s}_{p,r})},\quad \|u\|_{L^{2}_{T}(\dot{B}^{s}_{p,r})}\leq \|u\|_{\widetilde{L}^{2}_{T}(\dot{B}^{s}_{p,r})},\quad \|u\|_{L^{1}_{T}(\dot{B}^{s}_{p,r})}= \|u\|_{\widetilde{L}^{1}_{T}(\dot{B}^{s}_{p,r})}.
		\end{aligned}
	\end{equation}
    Furthermore, in order to restrict Besov norms to the low frequency part and the high-frequency part associated with different dissipative structures, we often use the following notations for any $s\in\mathbb{R}$ and $p\in[1,\infty]$:
\begin{equation}\nonumber
	\left\{
	\begin{aligned}
		&\|u\|_{\dot{B}^{s}_{p,r}}^{\ell}:=\|\{2^{js}\|\dot{\Delta}_{j}u\|_{L^{p}}\}_{j\leq 0}\|_{l^r},\quad 
\quad\quad\quad~\|u\|_{\dot{B}^{s}_{p,r}}^{h}:=\|\{2^{js}\|\dot{\Delta}_{j}u\|_{L^{p}}\}_{j\geq-1}\|_{l^r},\\
		&\|u\|_{\widetilde{L}^{\varrho}_{T}(\dot{B}^{s}_{p,r})}^{\ell}:=\|\{2^{js}\|\dot{\Delta}_{j}u\|_{L^{\varrho}_{T}(L^{p})}\}_{j\leq 0}\|_{l^r},\quad 
\|u\|_{\widetilde{L}^{\varrho}_{T}(\dot{B}^{s}_{p,r})}^{h}:=\|\{2^{js}\|\dot{\Delta}_{j}u\|_{L_{T}^{\varrho}(L^{p})}\}_{j\geq-1}\|_{l^r}.
	\end{aligned}
	\right.
\end{equation}
Define
$$
u^{\ell}:=\sum_{j\leq -1}\dot{\Delta}_{j}u,\quad\quad u^{h}:=u-u^{\ell}=\sum_{j\geq0}\dot{\Delta}_{j}u.
$$
It is easy to check for any $s'>0$ that
\begin{equation}\label{lh}
	\left\{
	\begin{aligned}
		&\|u^{\ell}\|_{\dot{B}^{s}_{p,r}}\lesssim \|u\|_{\dot{B}^{s}_{p,r}}^{\ell}\lesssim \|u\|_{\dot{B}^{s-s'}_{p,r}}^{\ell},\quad\quad\quad\quad\quad\quad 
\|u^{h}\|_{\dot{B}^{s}_{p,1}}\lesssim \|u\|_{\dot{B}^{s}_{p,r}}^{h}\lesssim \|u\|_{\dot{B}^{s+s'}_{p,r}}^{h},\\
		&\|u^{\ell}\|_{\widetilde{L}^{\varrho}_{T}(\dot{B}^{s}_{p,r})}\lesssim \|u\|_{\widetilde{L}^{\varrho}_{T}(\dot{B}^{s}_{p,r})}^{\ell}\lesssim 
\|u\|_{\widetilde{L}^{\varrho}_{T}(\dot{B}^{s-s'}_{p,r})}^{\ell},\quad\|u^{h}\|_{\widetilde{L}^{\varrho}_{T}(\dot{B}^{s}_{p,r})}\lesssim 
\|u\|_{\widetilde{L}^{\varrho}_{T}(\dot{B}^{s}_{p,r})}^{h}\lesssim \|u\|_{\widetilde{L}^{\varrho}_{T}(\dot{B}^{s+s'}_{p,r})}^{h}.
	\end{aligned}
	\right.
\end{equation}
Furthermore, one has
\begin{equation}\label{lhl}
	\left\{
	\begin{aligned}
		&\|u\|_{\dot{B}^{s}_{p,r}\cap\dot{B}^{s'}_{p,r}}=\|u\|_{\dot{B}^{s}_{p,r}}^{\ell}+\|u\|_{\dot{B}^{s'}_{p,r}}^{h},\quad s<s',\\
		&\|u\|_{\dot{B}^{s}_{p,r}+\dot{B}^{s'}_{p,r}}=\|u\|_{\dot{B}^{s}_{p,r}}^{\ell}+\|u\|_{\dot{B}^{s'}_{p,r}}^{h},\quad s>s'.
	\end{aligned}
	\right.
\end{equation}

\section{Global existence: Proof of Theorem \ref{Thm1.1}}\label{section3}
In this section, we prove Theorem \ref{Thm1.1} pertaining to the global existence of the solution to the Cauchy problem \eqref{system1}--\eqref{system5}. We will first work on the 
reformulated system \eqref{equation_R_P_u_pm}. Define the energy functional

\begin{equation}\label{Et}
	\begin{aligned}
		\mathcal{X}(t)=&\|n^{+}\|_{\widetilde{L}^{\infty}_t(\dot{B}_{2,1}^{\frac{d}{2}-1 }\cap\dot{B}_{2,1}^{\frac{d}{2}+1 } 
)}+\|\theta\|_{\widetilde{L}^{\infty}_t(\dot{B}_{2,1}^{\frac{d}{2}-2}\cap\dot{B}_{2,1}^{\frac{d}{2}})}+\|(u^{+},u^-)\|_{\widetilde{L}^{\infty}_t(\dot{B}_{2,1}^{\frac{d}{2}-2}\cap\dot{B}_{2,1}^{\frac{d}{2}+1 
} )}\\
		&+\|\theta\|_{L^{1}_t(\dot{B}_{2,1}^{\frac{d}{2}}\cap\dot{B}_{2,1}^{\frac{d}{2}})}+\|(u^{+},u^-)\|_{L^{1}_t(\dot{B}_{2,1}^{\frac{d}{2}}\cap\dot{B}_{2,1}^{\frac{d}{2}+1})\cap \widetilde{L}^{2}_t(\dot{B}_{2,1}^{\frac{d}{2}+2})}+\| \dive v\|_{L^{1}_t(\dot{B}_{2,1}^{\frac{d}{2}+1})}, \\
	\end{aligned}
\end{equation}
and the corresponding initial energy functional
\begin{equation}\label{E0}
	\begin{aligned}
		\mathcal{X}_0=&\|n_0^{+}\|_{ \dot{B}_{2,1}^{\frac{d}{2}-1 }\cap\dot{B}_{2,1}^{\frac{d}{2}+1 } } +\|\theta_0\|_{\dot{B}_{2,1}^{\frac{d}{2}-2}\cap\dot{B}_{2,1}^{\frac{d}{2}} 
}+\|(u_0^{+},u_0^-)\|_{\dot{B}_{2,1}^{\frac{d}{2}-2}\cap\dot{B}_{2,1}^{\frac{d}{2}+1 } }.
	\end{aligned}
\end{equation}
Note that \eqref{Et} and \eqref{E0} are, respectively, equivalent to $\mathcal{E}(t)+D(t)$ and $\mathcal{E}_0$ defined in \eqref{energy-functional1}--\eqref{initialenergy}.

We establish the uniform-in-time {\emph{a priori}} estimates as follows.

\begin{proposition}\label{propapriori}
	For the given time $T>0$, suppose that  the strong solution $(n^+, \theta,u^+,u^-)$  to the Cauchy problem \eqref{equation_R_P_u_pm} defined in $[0,T)$ satisfies
	\begin{align}
		\|(n^+,\theta)\|_{L^{\infty}_{t}(L^{\infty})}\leq \delta_0, \quad t\in(0,T),\label{priori1}
	\end{align}
	for some uniform constant $\delta_0>0$. Then, it holds that
	\begin{align}
		\mathcal{X}(t)\leq C_0\Big(\mathcal{X}_{0}(t)+(1+\mathcal{X}(t)^2)\mathcal{X}(t)^2 \Big), \quad t\in(0,T),
	\end{align}
	where $C_0>0$ is a constant independent of the time $T>0$.
\end{proposition}

The proof of Proposition \ref{propapriori} consists of Lemmas \ref{lemma1}--\ref{lemma3} which will be proved in Subsections \ref{sub1}--\ref{sub3.5}.

\subsection{Low-frequency estimates for the dissipative components}\label{sub1}
We have the following uniform estimates of $(\theta,u^{+},u^{-})$ at low frequencies. As emphasized before, we shall perform $d/2-2$ order energy estimates to derive the $L^1_t(B^{d/2}_{2,1})$ 
bound of $u^{\pm}$.

\begin{lemma}\label{lemma1}
	Let $T>0$ be any given time, and $(n^+, \theta,u^+,u^-)$ be any strong solution to the Cauchy problem \eqref{equation_R_P_u_pm} defined on $t\in(0,T)$ satisfying \eqref{priori1}. Then, 
we have
	\begin{align}
		\|(\theta, u^{+},u^{-})\|^{\ell}_{\widetilde{L}^{\infty}_t(\dot{B}_{2,1}^{\frac{d}{2}-2})}+\|(\theta, u^{+},u^{-})\|^{\ell}_{L^{1}_t (\dot{B}_{2,1}^{\frac{d}{2}})}\lesssim  \|(\theta_0, 
u_0^{+},u_0^{-})\|^{\ell}_{\dot{B}_{2,1}^{\frac{d}{2}-2}}+ (1+\mathcal{E}(t)^2)\mathcal{E}(t)^2,\label{low}
	\end{align}
	where $\mathcal{E}(t)$ is defined through \eqref{Et}.	
\end{lemma}

\begin{proof}
	First, we construct a suitable localized Lyapunov functional to capture the coupled dissipation structure of $\eqref{equation_R_P_u_pm}_2$--$\eqref{equation_R_P_u_pm}_4$.  To this end, 
applying the operator $\dot{\Delta}_j$ to $(\ref{equation_R_P_u_pm})_2$--$(\ref{equation_R_P_u_pm})_4$ yields
	\begin{equation}\label{Localization equation}
		\left\{
		\begin{aligned}
			& \partial_t \dot{\Delta}_j\theta +\beta_1^+\dive \dot{\Delta}_ju^++\beta_1^-\dive \dot{\Delta}_ju^-=\dot{\Delta}_jF_1,\\
			&  \partial_t \dot{\Delta}_ju^+ +\beta_2^+\nabla\dot{\Delta}_j\theta- \nu_1^{+}\Delta \dot{\Delta}_ju^+-\nu_2^+\nabla\dive \dot{\Delta}_ju^+=\dot{\Delta}_jF_2^+,\\
			&  \partial_t\dot{\Delta}_j u^- +\beta_2^-\nabla\dot{\Delta}_j\theta- \nu_1^{-}\Delta \dot{\Delta}_ju^--\nu_2^-\nabla\dive \dot{\Delta}_ju^-=\dot{\Delta}_jF_2^-.
		\end{aligned}
		\right.
	\end{equation}
	Taking the $L^2$ inner product of  $(\ref{Localization equation})_{2}$, $(\ref{Localization equation})_{3}$ and $(\ref{Localization equation})_{4}$ with $\dot{\Delta}_j \theta$, 
$\frac{\beta_1^+}{\beta_2^+}\dot{\Delta}_j u^+ $ and $\frac{\beta_1^-}{\beta_2^-}\dot{\Delta}_j u^-$,  respectively, we have
	\begin{equation}\label{basicj}
		\begin{aligned}			
\frac12&\frac{d}{dt}\left(\|\dot{\Delta}_j\theta\|^2_{L^2}+\frac{\beta_1^+}{\beta_2^+}\|\dot{\Delta}_ju^{+}\|^2_{L^2}+\frac{\beta_1^-}{\beta_2^-}\|\dot{\Delta}_ju^{-}\|^2_{L^2}\right)\\
			&\quad+\frac{\beta_1^+}{\beta_2^+}\left(\nu_1^{+}\|\nabla\dot{\Delta}_ju^{+}\|^2_{L^2}+\nu_2^{+}\|\dive \dot{\Delta}_ju^{+}\|^2_{L^2}\right)\\
			&\quad+\frac{\beta_1^-}{\beta_2^-}\left(\nu_1^{-}\|\nabla\dot{\Delta}_ju^{-}\|^2_{L^2}+\nu_2^{-}\|\dive \dot{\Delta}_ju^{-}\|^2_{L^2}\right)\\
			&=(\dot{\Delta}_jF_1|\dot{\Delta}_j \theta)_{L^2}+ \frac{\beta_1^+}{\beta_2^+}(\dot{\Delta}_jF_2^+|\dot{\Delta}_j 
u^+)_{L^2}+\frac{\beta_1^-}{\beta_2^-}(\dot{\Delta}_jF_2^-|\dot{\Delta}_j u^-)_{L^2}.
		\end{aligned}
	\end{equation}
	To obtain the dissipation of $\theta$, we deduce from $(\ref{Localization equation})$ that
	\begin{equation}\label{crossj}
		\begin{aligned}
			&\frac{d}{dt}\left(( \dot{\Delta}_ju^+|\nabla\dot{\Delta}_j\theta )_{L^2}+( \dot{\Delta}_ju^-|\nabla\dot{\Delta}_j\theta )_{L^2}\right)\\
			&\quad+(\beta_2^++\beta_2^-)\|\nabla\dot{\Delta}_j\theta\|^2_{L^2}-\beta_1^+\|\dive \dot{\Delta}_ju^{+}\|^2_{L^2}-\beta_1^-\|\dive \dot{\Delta}_ju^{-}\|^2_{L^2}\\
		&\quad	-\beta_1^+(\dive \dot{\Delta}_ju^{+}|\dive \dot{\Delta}_ju^{-})_{L^2}-\beta_1^-(\dive \dot{\Delta}_ju^{-}|\dive \dot{\Delta}_ju^{+})_{L^2}\\
            &\quad-\nu_1^{+}(\Delta \dot{\Delta}_ju^+|\nabla\dot{\Delta}_j\theta)_{L^2}-\nu_2^+(\nabla\dive \dot{\Delta}_ju^+|\nabla\dot{\Delta}_j\theta)_{L^2}\\
			&\quad-\nu_1^{-}(\Delta \dot{\Delta}_ju^-|\nabla\dot{\Delta}_j\theta)_{L^2}-\nu_2^-(\nabla\dive \dot{\Delta}_ju^-|\nabla\dot{\Delta}_j\theta)_{L^2}\\
			&=(\nabla\dot{\Delta}_jF_1|\dot{\Delta}_ju^{+} )_{L^2}+(\nabla\dot{\Delta}_jF_1|\dot{\Delta}_ju^{+} )_{L^2}\\
			&\quad+(\dot{\Delta}_jF_2^+|\nabla\dot{\Delta}_j\theta)_{L^2}+(\dot{\Delta}_jF_2^-|\nabla\dot{\Delta}_j\theta)_{L^2}.
		\end{aligned}
	\end{equation}
It follows from \eqref{basicj} and \eqref{crossj} that
	\begin{equation}\label{Lyapunovlow}
		\begin{aligned}
			& \frac{d}{dt} \mathcal{E}_{1,j}(t)+ \mathcal{D}_{1,j}(t)\lesssim \|\dot{\Delta}_j( F_1, F_2^+, F_2^-)\|_{L^2}\sqrt{\mathcal{E}_{1,j}(t)},\quad j\leq 0,
		\end{aligned}
	\end{equation}
	where for $\eta_1\in(0,1)$, we defined the Lyapunov functional
	\begin{equation}\label{E1l}
		\begin{aligned}			
\mathcal{E}_{1,j}(t)&=\frac{1}{2}\Big(\|\dot{\Delta}_j\theta\|^2_{L^2}+\frac{\beta_1^+}{\beta_2^+}\|\dot{\Delta}_ju^{+}\|^2_{L^2}+\frac{\beta_1^-}{\beta_2^-}\|\dot{\Delta}_ju^{-}\|^2_{L^2} 
\Big)\\
			&\quad+\eta_1 \Big( ( \dot{\Delta}_ju^+|\nabla\dot{\Delta}_j\theta )_{L^2}+( \dot{\Delta}_ju^-|\nabla\dot{\Delta}_j\theta )_{L^2} \Big),
		\end{aligned}
	\end{equation}
	and the dissipation term
	\begin{equation}\label{D1l}
		\begin{aligned}
			\mathcal{D}_{1,j}(t)&=\frac{\beta_1^+}{\beta_2^+}\left(\nu_1^{+}\|\nabla\dot{\Delta}_ju^{+}\|^2_{L^2}+\nu_2^{+}\|\dive \dot{\Delta}_ju^{+}\|^2_{L^2}\right)\\
			&\quad+\frac{\beta_1^-}{\beta_2^-}\Big(\nu_1^{-}\|\nabla\dot{\Delta}_ju^{-}\|^2_{L^2}+\nu_2^{-}\|\dive \dot{\Delta}_ju^{-}\|^2_{L^2}\\
			&\quad+\eta_1(\beta_2^++\beta_2^-)\|\nabla\dot{\Delta}_j\theta\|^2_{L^2}-\beta_1^+\|\dive \dot{\Delta}_ju^{+}\|^2_{L^2}-\beta_1^-\|\dive \dot{\Delta}_ju^{-}\|^2_{L^2}\\
			&\quad-\beta_1^+(\dive \dot{\Delta}_ju^{+}|\dive \dot{\Delta}_ju^{-})_{L^2}-\beta_1^-(\dive \dot{\Delta}_ju^{-}|\dive \dot{\Delta}_ju^{+})_{L^2}\\
			&\quad-\nu_1^{+}(\Delta \dot{\Delta}_ju^+|\nabla\dot{\Delta}_j\theta)_{L^2}-\nu_2^+(\nabla\dive \dot{\Delta}_ju^+|\nabla\dot{\Delta}_j\theta)_{L^2}\nonumber\\
			&\quad-\nu_1^{-}(\Delta \dot{\Delta}_ju^-|\nabla\dot{\Delta}_j\theta)_{L^2}-\nu_2^-(\nabla\dive \dot{\Delta}_ju^-|\nabla\dot{\Delta}_j\theta)_{L^2} \Big).
		\end{aligned}
	\end{equation}
    Note that the condition \eqref{system2-2} ensures
    \begin{align}\label{nablaDeltau}
    &\nu_1^{\pm}\|\nabla\dot{\Delta}_ju^{\pm}\|^2_{L^2}+\nu_2^{\pm}\|\dive \dot{\Delta}_ju^{\pm}\|^2_{L^2}\geq \frac{1}{\bar{\rho}^\pm} \min\{\mu^+,2\mu^++\lambda^+\} 
    \|\nabla\dot{\Delta}_ju^{\pm}\|_{L^2}^2.
    \end{align}
	Choosing $\eta_1$ small enough and taking advantage of Bernstein's lemma and the low-frequency cutoff property, we can easily deduce by \eqref{E1l} and \eqref{nablaDeltau} that
	\begin{align}
		\mathcal{E}_{1,j}(t)\sim \|\dot{\Delta}_j(\theta, u^{+},u^{-})\|^2_{L^2},\quad j\leq 0,\label{E1sim}
	\end{align}
	and
	\begin{align}
		\mathcal{D}_{1,j}(t)\gtrsim 2^{2j}\|\dot{\Delta}_j(\theta, u^{+},u^{-}) \|^2_{L^2}\quad j\leq 0.\label{D1sim}
	\end{align}
	Thence, in view of \eqref{E1sim} and \eqref{D1sim}, we are able to divide \eqref{Lyapunovlow} by $(\mathcal{E}_{1,j}+\var^2_*)^{\frac12}$ for $ \var_*>0$, integrate the resulting 
inequality over $[0,t],$ and then take the limit as  $ \var_*\to 0$ to obtain
	\begin{equation}\label{low-frequency1}
		\begin{aligned}
			\|\dot{\Delta}_j(\theta, u^{+},u^{-})\|_{L^2}&+2^{2j}\int_0^t\|\dot{\Delta}_j(\theta, u^{+},u^{-})\|_{L^2}d\tau\\
			&\lesssim \|\dot{\Delta}_j(\theta_0, u_0^{+},u_0^{-})\|_{L^2}+\int_0^t \|\dot{\Delta}_j( F_1, F_2^+, F_2^-)\|_{L^2}d\tau,
		\end{aligned}
	\end{equation}
	which implies
	\begin{equation}\label{low:es}
		\begin{aligned}
			& \|(\theta, u^{+},u^{-})\|^{\ell}_{\widetilde{L}^{\infty}_t(\dot{B}_{2,1}^{\frac{d}{2}-2})}+\|(\theta, u^{+},u^{-})\|^{\ell}_{L^{1}_t (\dot{B}_{2,1}^{\frac{d}{2}})}\\
			&\quad \lesssim \|(\theta_0, u_0^{+},u_0^{-})\|^{\ell}_{\dot{B}_{2,1}^{\frac{d}{2}-2}}+\|( F_1, F_2^+, F_2^-)\|^{\ell}_{L^{1}_t(\dot{B}_{2,1}^{\frac{d}{2}-2})}.
		\end{aligned}
	\end{equation}
	
	Next, we analyze the nonlinear terms associated with $F_1$, $F_2$ and $F_3$ as follows. Note that $n^\pm$ do not exhibit dissipation. As a result, when estimating the nonlinear terms, we need to transfer the time-integrability burden to $\theta$ or $u^\pm$. Due to $d/2-2>-d/2$ with $d\geq3$, we know that the product law $\dot{B}^{d/2-2}_{2,1}\times 
\dot{B}^{d/2}_{2,1}\hookrightarrow\dot{B}^{d/2-2}_{2,1} $ holds. This, together with the estimate \eqref{F1} for the smooth composite function $g_3^\pm(n^{+},\theta)$ with $g_3^\pm(0,0)=0$ 
and the condition \eqref{priori1}, implies
	\begin{equation}\nonumber
		\begin{aligned}
			\|g_3^\pm(n^{+},\theta)\dive u^{\pm}\|^{\ell}_{L^1_t (\dot{B}_{2,1}^{\frac{d}{2}-2})}&\lesssim 
\|g_3^\pm(n^{+},\theta)\|_{\widetilde{L}^{\infty}_t(\dot{B}_{2,1}^{\frac{d}{2}-1})} \|\dive u^{\pm}\|_{L^1_t (\dot{B}_{2,1}^{\frac{d}{2}-1})}\\
			&\lesssim \|(n^{+},\theta)\|_{\widetilde{L}^{\infty}_t(\dot{B}_{2,1}^{\frac{d}{2}-1})}\| u^{\pm}\|_{L_t^{1}(\dot{B}^{\frac{d}{2}}_{2,1})}.
		\end{aligned}
	\end{equation}
	As $\rho^{\pm}=\pi^{\pm}(\theta+\bar{P})$ with $\pi^{\pm}(\bar{P})=\bar{\rho}$, one has
	\begin{equation}\nonumber
		\begin{aligned}
			&\|\rho^- u^+\cdot\nabla n^+\|^{\ell}_{L^1_t (\dot{B}_{2,1}^{\frac{d}{2}-2})}\\&\lesssim (1+\|\rho^--\bar{\rho}\|_{\widetilde{L}^{\infty}_{t}(\dot{B}^{\frac{d}{2}}_{2,1})})(\| 
u^+\cdot\nabla n^+\|_{L^1_t (\dot{B}_{2,1}^{\frac{d}{2}-2})}+\| u^+\cdot\nabla n^+\|_{L^1_t (\dot{B}_{2,1}^{\frac{d}{2}-2})})\\
			&\lesssim (1+\|\theta\|_{\widetilde{L}^{\infty}_{t}(\dot{B}^{\frac{d}{2}}_{2,1})})\|u^+\|_{L^1_t (\dot{B}_{2,1}^{\frac{d}{2}})}\| n^+\|_{\widetilde{L}^{\infty}_t 
(\dot{B}_{2,1}^{\frac{d}{2}-1})},
		\end{aligned}
	\end{equation}
	and
	\begin{equation*}
		\begin{aligned}
			&\|\rho^+ u^-\cdot\nabla n^-\|^{\ell}_{L^1_t (\dot{B}_{2,1}^{\frac{d}{2}-2})}\\
			&\lesssim (1+\|\theta\|_{\widetilde{L}^{\infty}_{t}(\dot{B}^{\frac{d}{2}}_{2,1})})\|u^-\|_{L^1_t (\dot{B}_{2,1}^{\frac{d}{2}})}\| 
n^-\|_{\widetilde{L}^{\infty}_t(\dot{B}_{2,1}^{\frac{d}{2}-1})}.
		\end{aligned}
	\end{equation*}
	Since $n^-=n^-(n^++1,\theta+\bar{P})$, similar computations also give
	\begin{equation*}
		\begin{aligned}
			&\|(\mathcal{C}(n^++1,\theta+\bar{P})-\mathcal{C}( 1, \bar{P}))\rho^- u^+\cdot\nabla n^+\|^{\ell}_{L^1_t (\dot{B}_{2,1}^{\frac{d}{2}-2})}\\
			&\lesssim \|\mathcal{C}(n^++1,\theta+\bar{P})-\mathcal{C}( 1, \bar{P})\|_{\widetilde{L}^{\infty}_t(\dot{B}_{2,1}^{\frac{d}{2}})} 
(1+\|\rho^--1\|_{\widetilde{L}^{\infty}_{t}(\dot{B}^{\frac{d}{2}}_{2,1})})\| u^+\cdot\nabla n^+\|_{L^1_t (\dot{B}_{2,1}^{\frac{d}{2}-2})}\\
			&\lesssim \|(n^{+},\theta)\|_{\widetilde{L}^{\infty}_t(\dot{B}_{2,1}^{\frac{d}{2}})}  (1+\|\theta\|_{\widetilde{L}^{\infty}_{t}(\dot{B}^{\frac{d}{2}}_{2,1})})\|u^+\|_{L^1_t 
(\dot{B}_{2,1}^{\frac{d}{2}})}\| n^+\|_{\widetilde{L}^{\infty}_t(\dot{B}_{2,1}^{\frac{d}{2}-1})},
		\end{aligned}
	\end{equation*}
	and
	\begin{equation*}
		\begin{aligned}
			&\|(\mathcal{C}(n^++1,\theta+\bar{P})-\mathcal{C}( 1, \bar{P}))\rho^+ u^-\cdot\nabla n^-\|^{\ell}_{L^1_t (\dot{B}_{2,1}^{\frac{d}{2}-2})}\\&\lesssim 
\|(n^{+},\theta)\|_{\widetilde{L}^{\infty}_t(\dot{B}_{2,1}^{\frac{d}{2}})}  (1+\|\theta\|_{\widetilde{L}^{\infty}_{t}(\dot{B}^{\frac{d}{2}}_{2,1})})\|u^-\|_{L^1_t 
(\dot{B}_{2,1}^{\frac{d}{2}})}\|( n^+,\theta)\|_{\widetilde{L}^{\infty}_t(\dot{B}_{2,1}^{\frac{d}{2}-1})}.
		\end{aligned}
	\end{equation*}
	Combining the above estimates of $F_1$, we arrive at
	\begin{align}
		\|F_1\|^{\ell}_{L^1_t( \dot{B}_{2,1}^{\frac{d}{2}-1})}\lesssim(1+\mathcal{X}(t)^2)\mathcal{X}(t)^2.\label{F1ell}
	\end{align}
	With regard to the term $F_2$, the product estimate \eqref{uv2} gives 
	\begin{equation}
		\begin{aligned}
			\|u^{\pm}\cdot\nabla u^{\pm}\|^{\ell}_{L^1_t (\dot{B}_{2,1}^{\frac{d}{2}-2})}&\lesssim\|u^{\pm}\|_{\widetilde{L}^{\infty}_t(\dot{B}_{2,1}^{\frac{d}{2}-2})}\|\nabla 
u^{\pm}\|_{L^1_t (\dot{B}_{2,1}^{\frac{d}{2}})}\lesssim\|u^{\pm}\|_{\widetilde{L}^{\infty}_t(\dot{B}_{2,1}^{\frac{d}{2}-2})}\| u^{\pm}\|_{L^1_t (\dot{B}_{2,1}^{\frac{d}{2}+1})}.
		\end{aligned}
	\end{equation}
	Employing the estimate \eqref{F0} for the composition functions $g_i^{\pm}(\theta)$ ($i=1,2,4$), we have
	\begin{equation}\nonumber
		\begin{aligned}
			\|g_4^{\pm}(\theta)\nabla \theta\|^{\ell}_{L^1_t (\dot{B}_{2,1}^{\frac{d}{2}-2})}&\lesssim \|g_4^{\pm}(\theta)\|_{L^1_t 
(\dot{B}_{2,1}^{\frac{d}{2}})}\|\nabla\theta\|_{\widetilde{L}^{\infty}_t (\dot{B}_{2,1}^{\frac{d}{2}-2})}&\lesssim\| \theta\|_{L^1_t (\dot{B}_{2,1}^{\frac{d}{2}})} \| \theta\|_{\widetilde{L}^{\infty}_t(\dot{B}_{2,1}^{\frac{d}{2}-1})},\nonumber
		\end{aligned}
	\end{equation}
	and
	\begin{align}
		&\|g_1^{\pm}(\theta)\Delta u^{\pm}\|^{\ell}_{L^1_t (\dot{B}_{2,1}^{\frac{d}{2}-2})}+\|g_2^{\pm}(\theta)\nabla\dive  u^{\pm}\|^{\ell}_{L^1_t (\dot{B}_{2,1}^{\frac{d}{2}-2})} \lesssim\| \theta\|_{\widetilde{L}^{\infty}_t(\dot{B}_{2,1}^{\frac{d}{2}-2})}\| u^{\pm}\|_{L^1_t (\dot{B}_{2,1}^{\frac{d}{2}})}\nonumber.
	\end{align}
	The remainder is to analyze the composite term
	$$
	\frac{\mu^\pm(\nabla u^\pm+\nabla^{\top} u^\pm) \cdot\nabla \alpha^\pm}{R^\pm}+\frac{\lambda^\pm\dive u^\pm\nabla \alpha^\pm}{R^\pm}.
	$$
	We only deal with $ \frac{\nabla u^\pm \cdot \nabla \alpha^\pm}{R^\pm}$ since the other is essentially the same. To this end, we decompose
	\begin{align}
		\frac{\nabla u^\pm \cdot \nabla \alpha^\pm}{R^\pm}=\nabla u^\pm\cdot \nabla h^\pm (n^+,\theta)\tfrac{n^\pm}{n^\pm+1}+\nabla u^\pm\cdot\nabla h^\pm(n^+,\theta),\label{decom++}
	\end{align}
	where $h^\pm(n^+,\theta)=\alpha^\pm(n^++1,\theta+\bar{P})-\alpha^\pm(1,\bar{P}).$
	Using $n^-=n^-(n^++1,\theta+\bar{P})$, \eqref{uv2}, \eqref{F0} and \eqref{F1} , one has
	\begin{align}
		&\quad\Big\| \frac{\nabla u^\pm \cdot \nabla \alpha^\pm}{R^\pm}\Big\|^{\ell}_{L^1_t (\dot{B}_{2,1}^{\frac{d}{2}-2})}\nonumber\\
		&\leq\|\nabla u^\pm\cdot \nabla h^\pm(n^+,\theta)\tfrac{n^\pm}{n^++1}\|^{\ell}_{L^1_t (\dot{B}_{2,1}^{\frac{d}{2}-2})}+\|\nabla u^\pm\cdot\nabla h^\pm (n^+,\theta)\|^{\ell}_{L^1_t 
(\dot{B}_{2,1}^{\frac{d}{2}-2})}\nonumber\\
		&\lesssim \Big(\|\nabla h^\pm (n^+,\theta)\tfrac{n^\pm}{n^\pm+1}\|_{\widetilde{L}^{\infty}_t(\dot{B}_{2,1}^{\frac{d}{2}-1})}+\|\nabla 
h^\pm (n^+,\theta)\|_{\widetilde{L}^{\infty}_t(\dot{B}_{2,1}^{\frac{d}{2}-1})}\Big)\|\nabla u^\pm\|_{L^1_t (\dot{B}_{2,1}^{\frac{d}{2}})}\nonumber\\
		&\lesssim \Big(\|(n^+,\theta)\|_{\widetilde{L}^{\infty}_t(\dot{B}_{2,1}^{\frac{d}{2}})}\|n^+\|_{\widetilde{L}^{\infty}_t(\dot{B}_{2,1}^{\frac{d}{2}})}+\|(n^+,\theta)\|_{\widetilde{L}^{\infty}_t(\dot{B}_{2,1}^{\frac{d}{2}})}\Big)\|(u^+,u^{-})\|_{L^1_t (\dot{B}_{2,1}^{\frac{d}{2}})}.\nonumber
	\end{align}
	Collecting the above estimates concerning $F_2$, we end up with
	\begin{align}
		\|F_2^+\|^{\ell}_{L^1_t (\dot{B}_{2,1}^{\frac{d}{2}-2})}\lesssim \mathcal{X}(t)^2+\mathcal{X}(t)^3.\label{F2ell}
	\end{align}
	Similarly, we get
	\begin{align}
		\|F_2^-\|^{\ell}_{L^1_t (\dot{B}_{2,1}^{\frac{d}{2}-2})}\lesssim (1+\mathcal{X}(t) )\mathcal{X}(t)^2.\label{F3ell}
	\end{align}
	The details are omitted. Substituting \eqref{F1ell}, \eqref{F2ell} and \eqref{F3ell} into \eqref{low:es}, we obtain \eqref{low} and complete the proof of Lemma \ref{lemma1}.

\end{proof}

\subsection{High-frequency estimates for for the dissipative components}\label{sub2}
In this subsection, we establish the corresponding a priori estimates in the high-frequency regime.
\begin{lemma}\label{lemma2}
	Let $T>0$ be any given time, and $(n^+, \theta,u^+,u^-)$ be any strong solution to the Cauchy problem \eqref{equation_R_P_u_pm} for $t\in(0,T)$. Then it holds that
	\begin{equation}\label{2.23}
		\begin{aligned}
			&\|\theta\|^h_{\widetilde{L}^{\infty}_t(\dot{B}_{2,1}^{\frac{d}{2}})}+\|(u^{+},u^{-})\|^h_{\widetilde{L}^{\infty}_t(\dot{B}_{2,1}^{\frac{d}{2}-1})}+\|\theta \|^h_{L^{1}_t 
(\dot{B}_{2,1}^{\frac{d}{2}})}+\|(u^{+},u^{-})\|^h_{L^{1}_t (\dot{B}_{2,1}^{\frac{d}{2}+1})}\\
			&\quad \lesssim \|\theta_0 \|^h_{\dot{B}_{2,1}^{\frac{d}{2}}}+\|(u_0^{+},u_0^-)\|^h_{_{ \dot{B}_{2,1}^{\frac{d}{2}-1}}}+(1+\mathcal{X}(t)^2)\mathcal{X}(t)^2,
		\end{aligned}
	\end{equation}
	where $\mathcal{X}(t)$ is defined in \eqref{Et}.	
\end{lemma}

\begin{proof}
	The proof is based on the construction of the Lyapunov functional at high frequencies similarly as in Lemma \ref{lemma1}. 
	We rewrite \eqref{Localization equation} by
	\begin{equation}\label{Localization equation1}
		\left\{
		\begin{aligned}
			& \partial_t \dot{\Delta}_j\theta +\beta_1^+\dive \dot{\Delta}_ju^++\beta_1^-\dive \dot{\Delta}_ju^-=\dot{\Delta}_jF_1,\\
			&  \partial_t \dot{\Delta}_ju^+ +\beta_2^+\nabla\dot{\Delta}_j\theta- \nu_1^{+}\Delta \dot{\Delta}_ju^+-\nu_2^+\nabla\dive \dot{\Delta}_ju^+=\dot{\Delta}_jF_2^+,\\
			&  \partial_t\dot{\Delta}_j u^- +\beta_2^-\nabla\dot{\Delta}_j\theta- \nu_1^{-}\Delta \dot{\Delta}_ju^--\nu_2^-\nabla\dive \dot{\Delta}_ju^-=\dot{\Delta}_jF_2^-.
		\end{aligned}
		\right.
	\end{equation}
First, we take the $L^2$ inner product of    $\nabla(\ref{Localization equation1})_{2} $ with $\nabla\dot{\Delta}_j \theta$ to get
	\begin{equation}\label{high1}
		\begin{aligned}
			&\frac12\frac{d}{dt}\|\nabla\dot{\Delta}_j\theta\|^2_{L^2}+\beta_1^+(\nabla\dive u^+|\nabla\dot{\Delta}_j\theta)_{L^2}+\beta_1^-(\nabla\dive 
u^-|\nabla\dot{\Delta}_j\theta)_{L^2}=(\nabla\dot{\Delta}_j F_1|\nabla\dot{\Delta}_j\theta)_{L^2}.
		\end{aligned}
	\end{equation}
	Next, by taking the $L^2$ inner product of    $\nabla(\ref{Localization equation1})_{1} $, $ (\ref{Localization equation1})_{2} $, and $ (\ref{Localization equation1})_{3} $  with 
$\nabla\dot{\Delta}_j \theta$, $\dot{\Delta}_j u^+$, and $\dot{\Delta}_j u^-$, respectively, we get
		\begin{flalign}\label{high2}
	&\ \frac{d}{dt}(\nabla\dot{\Delta}_j\theta|\dot{\Delta}_ju^{+})_{L^2}-\beta_1^+\|\dive \dot{\Delta}_ju^{+}\|^2_{L^2}-\beta_1^-(\dive \dot{\Delta}_ju^{-}|\dive 
\dot{\Delta}_ju^{+})_{L^2}&\nonumber\\
			&\ \quad\quad +\beta_{2}^+\|\nabla\dot{\Delta}_j\theta\|^2_{L^2}-\nu_1^+(\Delta \dot{\Delta}_ju^+|\nabla\dot{\Delta}_j\theta)_{L^2}-\nu_2^+(\nabla\dive 
\dot{\Delta}_ju^+|\nabla\dot{\Delta}_j\theta)_{L^2}& \nonumber\\
			&\ \quad =(\nabla\dot{\Delta}_j F_1|\dot{\Delta}_ju^{+})_{L^2}+(\dot{\Delta}_j F_2^+|\nabla\dot{\Delta}_j\theta)_{L^2}, 
		\end{flalign}
	and
			\begin{flalign}\label{high3}
			&\ \frac{d}{dt}(\nabla\dot{\Delta}_j\theta|\dot{\Delta}_ju^{-})_{L^2}-\beta_1^+(\dive \dot{\Delta}_ju^{+}|\dive \dot{\Delta}_ju^{-})_{L^2}-\beta_1^-\|\dive 
\dot{\Delta}_ju^{+}\|^2_{L^2}& \nonumber\\
			&\ \quad \quad +\beta_{2}^-\|\nabla\dot{\Delta}_j\theta\|^2_{L^2}-\nu_1^-(\Delta \dot{\Delta}_ju^-|\nabla\dot{\Delta}_j\theta)_{L^2}-\nu_2^-(\nabla\dive 
\dot{\Delta}_ju^-|\nabla\dot{\Delta}_j\theta)_{L^2} & \nonumber\\
			&\ \quad  =(\nabla\dot{\Delta}_j F_1|\dot{\Delta}_ju^{-})_{L^2}+(\dot{\Delta}_j F_2^-|\nabla\dot{\Delta}_j\theta)_{L^2}. &
		\end{flalign}
Due to $(\nu_1^\pm+\nu_2^\pm)(\nabla\dive \dot{\Delta}_j u^\pm | \nabla\dot{\Delta}_j\theta)_{L^2}=\nu_1^\pm (\Delta \dot{\Delta}_j u^\pm | \nabla\dot{\Delta}_j \theta)_{L^2}+\nu_2^\pm ( \nabla\dive \dot{\Delta}_j u^\pm| \nabla\dot{\Delta}_j\theta)_{L^2}$, there are some cancellations of high order cross terms when we add \eqref{high1}, \eqref{high2} and \eqref{high3} with suitably coefficients. Indeed, we have
	\begin{flalign}
			&\  
\frac{d}{dt}\Big(\frac12\|\nabla\dot{\Delta}_j\theta\|^2_{L^2}+\frac{\beta_1^+}{\nu_1^++\nu_2^+}(\nabla\dot{\Delta}_j\theta|\dot{\Delta}_ju^{+})_{L^2}+\frac{\beta_1^-}{\nu_1^-+\nu_2^-}(\nabla\dot{\Delta}_j\theta|\dot{\Delta}_ju^{-})_{L^2}\Big)\\
&\ \quad\quad+\Big(\frac{\beta_1^+\beta_2^+}{\nu_1^++\nu_2^+} +\frac{\beta_1^-\beta_2^-}{\nu_1^-+\nu_2^-}\Big)\|\nabla\dot{\Delta}_j\theta\|^2_{L^2} 
& \nonumber\\
			&\ \quad\quad-\frac{({\beta_1^+})^2}{\nu_1^++\nu_2^+}\|\dive \dot{\Delta}_ju^{+}\|^2_{L^2}-\frac{\beta_1^+\beta_1^-}{\nu_1^++\nu_2^+}(\dive \dot{\Delta}_ju^{-}|\dive \dot{\Delta}_ju^{+})_{L^2} &\nonumber\\
			&\ \quad\quad-\frac{\beta_1^-\beta_1^+}{\nu_1^-+\nu_2^-}(\dive \dot{\Delta}_ju^{+}|\dive \dot{\Delta}_ju^{-})_{L^2}-\frac{(\beta_1^-)^2}{\nu_1^-+\nu_2^-}\|\dive \dot{\Delta}_ju^{+}\|^2_{L^2} & \nonumber\\
			&\ \quad =(\nabla\dot{\Delta}_j F_1|\nabla\dot{\Delta}_j\theta)_{L^2}+\frac{\beta_1^+}{\nu_1^++\nu_2^+}(\nabla\dot{\Delta}_j F_1|\dot{\Delta}_ju^{+})_{L^2}+\frac{\beta_1^+}{\nu_1^++\nu_2^+}(\dot{\Delta}_j 
F_2^+|\nabla\dot{\Delta}_j\theta)_{L^2}\nonumber\\
			&\ \quad\quad+\frac{\beta_1^-}{\nu_1^-+\nu_2^-}(\nabla\dot{\Delta}_j F_1|\dot{\Delta}_ju^{-})_{L^2}+\frac{\beta_1^-}{\nu_1^-+\nu_2^-}(\dot{\Delta}_j F_2^-|\nabla\dot{\Delta}_j\theta)_{L^2} .\label{high4}
		\end{flalign}
	Let
	\begin{align}
		\mathcal{E}_{2,j}(t)&=\frac{1}{2}\Big(\|\dot{\Delta}_j\theta\|^2_{L^2}+\frac{\beta_1^+}{\beta_2^+}\|\dot{\Delta}_ju^{+}\|^2_{L^2}+\frac{\beta_1^-}{\beta_2^-}\|\dot{\Delta}_ju^{-}\|^2_{L^2}\Big)\nonumber\\	
&+\eta_2\Big(\frac12\|\nabla\dot{\Delta}_j\theta\|^2_{L^2}+\frac{\beta_1^+}{\nu_1^++\nu_2^+}(\nabla\dot{\Delta}_j\theta|\dot{\Delta}_ju^{+})_{L^2}+\frac{\beta_1^-}{\nu_1^-+\nu_2^-}(\nabla\dot{\Delta}_j\theta|\dot{\Delta}_ju^{-})_{L^2}\Big),\label{E2}
	\end{align}
	and
	\begin{align}
		\mathcal{D}_{2,j}(t)&=\frac{\beta_1^+}{\beta_2^+}\Big(\nu_1^{+}\|\nabla\dot{\Delta}_ju^{+}\|^2_{L^2}+\nu_2^{+}\|\dive 
\dot{\Delta}_ju^{+}\|^2_{L^2}\Big)+\frac{\beta_1^-}{\beta_2^-}\Big( \nu_1^{-}\|\nabla\dot{\Delta}_ju^{-}\|^2_{L^2}+\nu_2^{-}\|\dive \dot{\Delta}_ju^{-}\|^2_{L^2}\Big)\nonumber\\
		&\ \quad+\eta_2\bigg( \Big(\frac{\beta_1^+\beta_2^+}{\nu_1^++\nu_2^+} +\frac{\beta_1^-\beta_2^-}{\nu_1^-+\nu_2^-}\Big)\|\nabla\dot{\Delta}_j\theta\|^2_{L^2} \nonumber\\
		&\ \quad\quad- \frac{({\beta_1^+})^2}{\nu_1^++\nu_2^+}\|\dive \dot{\Delta}_ju^{+}\|^2_{L^2}-\frac{\beta_1^+\beta_1^-}{\nu_1^++\nu_2^+}(\dive \dot{\Delta}_ju^{-}|\dive \dot{\Delta}_ju^{+})_{L^2} \nonumber\\
		&\ \quad\quad-\frac{({\beta_1^+})^2}{\nu_1^++\nu_2^+}\|\dive \dot{\Delta}_ju^{+}\|^2_{L^2}-\frac{\beta_1^+\beta_1^-}{\nu_1^++\nu_2^+}(\dive \dot{\Delta}_ju^{-}|\dive \dot{\Delta}_ju^{+})_{L^2} &\nonumber\\
		&\left.\ \quad\quad-\frac{\beta_1^-\beta_1^+}{\nu_1^-+\nu_2^-}(\dive \dot{\Delta}_ju^{+}|\dive \dot{\Delta}_ju^{-})_{L^2}-\frac{(\beta_1^-)^2}{\nu_1^-+\nu_2^-}\|\dive \dot{\Delta}_ju^{+}\|^2_{L^2}\right).\label{D2}
	\end{align}
	In light of \eqref{nablaDeltau}, \eqref{E2} and \eqref{D2}, we choose $\eta_2$ small enough and use Bernstein's lemma to deduce that, for any $j\geq -1$,
	\begin{align}
		\mathcal{E}_{2,j}(t)\sim \|\dot{\Delta}_j(\nabla\theta, u^{+},u^{-})\|^2_{L^2},\label{3.30}
	\end{align}
	and
	\begin{align}
		\mathcal{D}_{2,j}(t)\gtrsim \|\dot{\Delta}_j(\nabla\theta, u^{+},u^{-})\|^2_{L^2}.\label{3.31}
	\end{align}
Therefore, combining \eqref{basicj} and \eqref{high1}--\eqref{3.31} we have
	\begin{align}\label{Lyapunovhigh}
		\frac{d}{dt} \mathcal{E}_{2,j}(t)+ \mathcal{E}_{2,j}(t)\lesssim \|\dot{\Delta}_j(\nabla F_1, F_2^+, F_2^-)\|_{L^2}\sqrt{\mathcal{E}_{2,j}(t)},\quad j\geq -1.
	\end{align}
By similar argument  as in \eqref{low-frequency1}, the inequality \eqref{Lyapunovhigh} becomes
	\begin{equation}\label{high-frequency1}
		\begin{aligned}
			\|\dot{\Delta}_j(\nabla\theta, u^{+},u^{-})\|_{L^2}&+\int_0^t\|\dot{\Delta}_j(\nabla\theta, u^{+},u^{-})\|^2_{L^2}d\tau \\
			&\lesssim \|\dot{\Delta}_j(\nabla\theta_0, u_0^{+},u_0^{-})\|_{L^2}+\int_0^t \|\dot{\Delta}_j(\nabla F_1, F_2^+, F_2^-)\|_{L^2}d\tau.
		\end{aligned}
	\end{equation}
	Multiplying \eqref{high-frequency1} by $2^{j(\frac{d}{2}-1)}$,  taking the supremum on $[0,t]$, and then summing over $j\geq-1, $ we have
	\begin{equation}
		\begin{aligned}
			& \|(\nabla\theta, u^{\pm}\|^h_{\widetilde{L}^{\infty}_t(\dot{B}_{2,1}^{\frac{d}{2}-1})}+\|(\nabla\theta, u^{+},u^{-})\|^h_{L^{1}_t(\dot{B}_{2,1}^{\frac{d}{2}-1})}\\
			&\quad\lesssim \|(\nabla\theta_0, u_0^{+},u_0^{-})\|^h_{\dot{B}_{2,1}^{\frac{d}{2}-1}}+\|(\nabla F_1, F_2^+, F_2^-)\|^h_{L^{1}_t(\dot{B}_{2,1}^{\frac{d}{2}-1})}.\label{high1111}
		\end{aligned}
	\end{equation}
	We estimate the nonlinear terms as follows. First, as $\dot{B}^{\frac{d}{2}}_{2,1}$ is a algebra, one has
	\begin{equation*}
		\begin{aligned}
			\|g_3^+(n^{+},\theta)\dive u^{\pm}\|^h_{L^1_t (\dot{B}_{2,1}^{\frac{d}{2}})}&\lesssim \|g_3^+(n^{+},\theta)\|_{\widetilde{L}^{\infty}_t(\dot{B}_{2,1}^{\frac{d}{2}})} 
\|\dive u^{\pm}\|_{L^1_t (\dot{B}_{2,1}^{\frac{d}{2}})}\\
			&\lesssim \|(n^{+},\theta)\|_{\widetilde{L}^{\infty}_t(\dot{B}_{2,1}^{\frac{d}{2}})}\| u^{\pm}\|_{L^1_t(\dot{B}_{2,1}^{\frac{d}{2}+1 })}\lesssim\mathcal{X}^2(t).
		\end{aligned}
	\end{equation*}
	And similarly, it holds that
	\begin{equation*}
		\begin{aligned}
			&\|\rho^- u^+\cdot\nabla n^+\|^h_{L^1_t (\dot{B}_{2,1}^{\frac{d}{2}})}+\|\rho^+ u^-\cdot\nabla n^-\|^h_{L^1_t (\dot{B}_{2,1}^{\frac{d}{2}})}\\
			&\quad \lesssim (1+\|(\rho^+-\bar{\rho}^+,\rho^--\bar{\rho}^-)\|_{\widetilde{L}^{\infty}_{t}(\dot{B}^{\frac{d}{2}}_{2,1})})(\| u^+\cdot\nabla n^+\|_{L^1_t (\dot{B}_{2,1}^{\frac{d}{2}})}+\| u^-\cdot\nabla n^-\|_{L^1_t (\dot{B}_{2,1}^{\frac{d}{2}})})\\
			&\quad \lesssim (1+\|\theta\|_{\widetilde{L}^{\infty}_{t}(\dot{B}^{\frac{d}{2}}_{2,1})})\|(u^+,u^{-})\|_{L^1_t(\dot{B}_{2,1}^{\frac{d}{2}})}\| 
(n^+,\theta)\|_{\widetilde{L}^{\infty}_t(\dot{B}_{2,1}^{\frac{d}{2}+1})}\lesssim (1+\mathcal{X}(t))\mathcal{X}(t)^2,
		\end{aligned}
	\end{equation*}
	and
	\begin{equation*}
		\begin{aligned}
			&\|(\mathcal{C}(n^++1,\theta+\bar{P})-\mathcal{C}( 1, \bar{P}))\rho^- u^+\cdot\nabla n^+\|^h_{L_t^{1}(\dot{B}^{\frac{d}{2}}_{2,1})}\\
			&\quad+\|(\mathcal{C}(n^++1,\theta+\bar{P})-\mathcal{C}( 1, \bar{P}))\rho^+ u^-\cdot\nabla n^-\|^h_{L_t^{1}(\dot{B}^{\frac{d}{2}}_{2,1})}\\
			&\lesssim \|\mathcal{C}(n^++1,\theta+\bar{P})-\mathcal{C}( 1, \bar{P})\|_{\widetilde{L}^{\infty}_t(\dot{B}_{2,1}^{\frac{d}{2}})} 
(1+\|\rho^--\bar{\rho}^-\|_{\widetilde{L}^{\infty}_{t}(\dot{B}^{\frac{d}{2}}_{2,1})})\| u^+\cdot\nabla n^+\|_{L^1_t (\dot{B}_{2,1}^{\frac{d}{2}})}\\
			&\quad+\|\mathcal{C}(n^++1,\theta+\bar{P})-\mathcal{C}( 1, \bar{P})\|_{\widetilde{L}^{\infty}_t(\dot{B}_{2,1}^{\frac{d}{2}})} 
(1+\|\rho^+-\bar{\rho}^+\|_{\widetilde{L}^{\infty}_{t}(\dot{B}^{\frac{d}{2}}_{2,1})})\| u^-\cdot\nabla n^-\|_{L_t^{1}(\dot{B}^{\frac{d}{2}}_{2,1})}\\
			&\lesssim \|(n^{+},\theta)\|_{\widetilde{L}^{\infty}_t(\dot{B}_{2,1}^{\frac{d}{2}})}  (1+\|\theta\|_{\widetilde{L}^{\infty}_{t}(\dot{B}^{\frac{d}{2}}_{2,1})})\|(u^+,u^-)\|_{L^1_t 
(\dot{B}_{2,1}^{\frac{d}{2}})}\| (n^+,\theta)\|_{\widetilde{L}^{\infty}_t(\dot{B}_{2,1}^{\frac{d}{2}+1})}\\
			&\lesssim (1+\mathcal{X}(t))\mathcal{X}(t)^3.
		\end{aligned}
	\end{equation*}
	Note that the $\dot{B}^{\frac{d}{2}+1}_{2,1}$-norm of $n^+$ plays a key role in the above analysis. Thus, we have
	\begin{align}
		\|F_1\|^h_{L_t^{1}(\dot{B}^{\frac{d}{2}}_{2,1}))}\lesssim(1+\mathcal{X}(t)^2)\mathcal{X}(t)^2.\label{F1high}
	\end{align}
	For the term $F_2$, we first have
	\begin{align*}
		\|u^{\pm}\cdot\nabla u^{\pm}\|^h_{L^1_t( \dot{B}_{2,1}^{\frac{d}{2}-1})}\lesssim\|u^{\pm}\|_{\widetilde{L}^{\infty}_t(\dot{B}_{2,1}^{\frac{d}{2}-1})}\|u^{\pm}\|_{L^1_t 
(\dot{B}_{2,1}^{\frac{d}{2}+1})},
	\end{align*}
	and
	\begin{align*}
		\|g_4^{\pm}(\theta)\nabla \theta\|^h_{L^1_t( \dot{B}_{2,1}^{\frac{d}{2}-1})}&\lesssim\|g_4^{\pm}(\theta)\|_{\widetilde{L}^{\infty}_t(\dot{B}_{2,1}^{\frac{d}{2}})}\|\nabla \theta\|_{L^1_t( 
\dot{B}_{2,1}^{\frac{d}{2}-1})}\lesssim\| \theta\|_{\widetilde{L}^{\infty}_t(\dot{B}_{2,1}^{\frac{d}{2}})}\| \theta\|_{L^1_t (\dot{B}_{2,1}^{\frac{d}{2}})}.
	\end{align*}
	A similar argument yields
	\begin{align*}
		\|g_1^{\pm}(\theta)\Delta u^{\pm}\|^h_{L^1_t( \dot{B}_{2,1}^{\frac{d}{2}-1})}&\lesssim\|g_1^{\pm}(\theta)\|_{\widetilde{L}^{\infty}_t(\dot{B}_{2,1}^{\frac{d}{2}})}\|\Delta 
u^{\pm}\|_{L^1_t( \dot{B}_{2,1}^{\frac{d}{2}-1})}\lesssim\| \theta\|_{\widetilde{L}^{\infty}_t(\dot{B}_{2,1}^{\frac{d}{2}})}\| u^{\pm}\|_{L_t^{1}(\dot{B}^{\frac{d}{2}+1}_{2,1})},
	\end{align*}
	and
	\begin{align*}
		\|g_2^{\pm}(\theta)\nabla\dive  u^{\pm}\|^h_{L^1_t( \dot{B}_{2,1}^{\frac{d}{2}-1})}&\lesssim\|g_2^{\pm}(\theta)\|_{\widetilde{L}^{\infty}_t(\dot{B}_{2,1}^{\frac{d}{2}})}\|\nabla\dive  
u^{\pm}\|_{L^1_t( \dot{B}_{2,1}^{\frac{d}{2}-1})}\lesssim\| \theta\|_{\widetilde{L}^{\infty}_t(\dot{B}_{2,1}^{\frac{d}{2}})}\|u^{\pm}\|_{L_t^{1}(\dot{B}^{\frac{d}{2}+1}_{2,1})}.
	\end{align*}
	It suffices to analyze
	$$
	\frac{\mu^\pm(\nabla u^\pm+\nabla^{\top} u^\pm) \cdot\nabla \alpha^\pm}{R^\pm}+\frac{\lambda^\pm\dive u^\pm\nabla \alpha^\pm}{R^\pm}.
	$$
	We only deal with $ \frac{\nabla u^\pm \cdot \nabla \alpha^\pm}{R^\pm}$ since the other is essential the same. It holds by \eqref{decom++}, \eqref{uv2} and \eqref{F1} that
	\begin{align*}
		& \quad \Big\| \frac{\nabla u^\pm \cdot \nabla \alpha^\pm }{R^\pm}\Big\|^h_{L^1_t( \dot{B}_{2,1}^{\frac{d}{2}-1})}\\
		&\leq\|\nabla u^\pm \cdot \nabla h^\pm(n^+,\theta)\tfrac{n^\pm }{n^\pm +1}\|^h_{L^1_t( \dot{B}_{2,1}^{\frac{d}{2}-1})}+\|\nabla u^\pm \cdot\nabla h^\pm(n^+,\theta)\|^h_{L^1_t( 
\dot{B}_{2,1}^{\frac{d}{2}-1})}\nonumber\\
		&\lesssim \Big(\|\nabla h^\pm(n^+,\theta)\tfrac{n^\pm}{n^\pm+1}\|_{\widetilde{L}^{\infty}_t(\dot{B}_{2,1}^{\frac{d}{2}-1})}+\|\nabla 
h^\pm(n^+,\theta)\|_{\widetilde{L}^{\infty}_t(\dot{B}_{2,1}^{\frac{d}{2}-1})}\Big)\|\nabla u^\pm\|_{L^1_t (\dot{B}_{2,1}^{\frac{d}{2}})}\nonumber\\
		&\lesssim 
\Big(\|\nabla h^\pm(n^+,\theta)\|_{\widetilde{L}^{\infty}_t(\dot{B}_{2,1}^{\frac{d}{2}-1})}\|n^\pm\|_{\widetilde{L}^{\infty}_t(\dot{B}_{2,1}^{\frac{d}{2}})}+\|\nabla h^\pm(n^+,\theta)\|_{\widetilde{L}^{\infty}_t(\dot{B}_{2,1}^{\frac{d}{2}-1})}\Big)\|\nabla 
u^\pm\|_{L^1_t(\dot{B}_{2,1}^{\frac{d}{2}})}\nonumber\\
		&\lesssim 
\Big(\|(n^+,\theta)\|_{\widetilde{L}^{\infty}_t(\dot{B}_{2,1}^{\frac{d}{2}})}\|(n^+,\theta)\|_{\widetilde{L}^{\infty}_t(\dot{B}_{2,1}^{\frac{d}{2}})}+\|(n^+,\theta)\|_{\widetilde{L}^{\infty}_t( 
\dot{B}_{2,1}^{\frac{d}{2}})}\Big)\| u^\pm\|_{L_t^{1}(\dot{B}^{\frac{d}{2}+1}_{2,1})}.
	\end{align*}
	Therefore, we arrive at
	\begin{align}
		\|F_2^+\|^h_{L^1_t( \dot{B}_{2,1}^{\frac{d}{2}-1})}\lesssim \mathcal{X}(t)^2+\mathcal{X}(t)^3.\label{F2high}
	\end{align}
	Similarly, it follows that
	\begin{align}
		\|F_2^-\|^h_{L^1_t( \dot{B}_{2,1}^{\frac{d}{2}-1})}\lesssim \mathcal{X}(t)^2+\mathcal{X}(t)^3.\label{F3high}
	\end{align}
	Substituting \eqref{F1high}, \eqref{F2high} and \eqref{F3high} into \eqref{high1111}, we prove
	\begin{align}
		& \|(\nabla\theta, u^{\pm}\|^h_{\widetilde{L}^{\infty}_t(\dot{B}_{2,1}^{\frac{d}{2}-1})}+\|(\nabla\theta, u^{+},u^{-})\|^h_{L^{1}_t(\dot{B}_{2,1}^{\frac{d}{2}-1})}\lesssim \|(\nabla\theta_0, u_0^{+},u_0^{-})\|^h_{\dot{B}_{2,1}^{\frac{d}{2}-1}}+\mathcal{X}(t)^2+\mathcal{X}(t)^3.\label{high11111}
	\end{align}

	Finally, we establish the higher-order dissipation estimate of $u^{\pm}$. In fact, $\eqref{equation_R_P_u_pm}_3$--$\eqref{equation_R_P_u_pm}_4$ can be written by
	\begin{equation}\label{3.39}
		\begin{array}{ll}
			\partial_t u^\pm - \nu_1^\pm\Delta u^\pm-\nu_2^\pm\nabla\dive u^\pm=-\beta_2^\pm\nabla \theta+ F_2^\pm.
		\end{array}
	\end{equation}
	The maximal regularity estimate (see Lemma \ref{lame-system-estimate}) applied to \eqref{3.39}  gives rise to
	\begin{align}   &\|u^{\pm}\|^h_{_{\widetilde{L}^{\infty}_t(\dot{B}_{2,1}^{\frac{d}{2}-1})}}+ \|u^{\pm}\|^h_{_{L^{1}_t (\dot{B}_{2,1}^{\frac{d}{2}+1})}}\nonumber\\
		&\quad \lesssim \|(u_0^{+},u_0^-)\|^h_{_{ \dot{B}_{2,1}^{\frac{d}{2}-1}}}+\|\theta\|^h_{L^1_t (\dot{B}_{2,1}^{\frac{d}{2}})}+\|(F_2^+,F_2^-)\|^h_{L^1_t( \dot{B}_{2,1}^{\frac{d}{2}-1})}\nonumber\\
		&\quad \lesssim \|(\nabla\theta_0, u_0^{+},u_0^{-})\|^h_{\dot{B}_{2,1}^{\frac{d}{2}-1}}+\mathcal{X}(t)^2+\mathcal{X}(t)^3,
	\end{align}
	where one has used \eqref{F2high}, \eqref{F3high} and \eqref{high11111}.
	This completes the proof of \eqref{2.23}.
\end{proof}

\subsection{The non-dissipative component:  Low-order estimates}

The uniform $\widetilde{L}^{\infty}_t(\dot{B}_{2,1}^{d/2-1}\cap\dot{B}^{d/2}_{2,1})$ bounds of the non-dissipative component $n^{+}$ are derived based on the dissipation estimates of 
$u^+$ exhibited in \eqref{low} and \eqref{2.23}. However, it is not direct to derive the $\widetilde{L}^{\infty}_t(\dot{B}^{d/2+1}_{2,1})$ estimate for $n^{+}$ required in Lemma
\ref{lemma2} using the equations of $n^\pm$. This will be done in the next section based on a (almost) symmetric argument.

\begin{lemma}\label{lemman+}
	Let $T>0$ be any given time, and $(n^+, \theta,u^+,u^-)$ be any strong solution to the Cauchy problem \eqref{equation_R_P_u_pm} defined on $t\in(0,T)$ satisfying \eqref{priori1}. Then, 
we have
	\begin{align}
		\|n^{+}\|_{\widetilde{L}^{\infty}_t(\dot{B}_{2,1}^{\frac{d}{2}-1}\cap\dot{B}^{\frac{d}{2}}_{2,1})}\lesssim \|n_0^+\|_{\dot{B}^{\frac{d}{2}-1}_{2,1}\cap\dot{B}^{\frac{d}{2}}_{2,1}}+ (1+\mathcal{X}(t)^2)\mathcal{X}(t)^2,\label{3.43}
	\end{align}
	where $\mathcal{X}(t)$ and $\mathcal{X}_0$ are defined through \eqref{Et} and \eqref{E0},  respectively.
\end{lemma}

\begin{proof}
	Applying $\dot{\Delta}_j$ to the equation of $n^{+}$ yields
	\begin{align}
	\partial_t \dot{\Delta}_j n^{+} +\dive \dot{\Delta}_j u^{+}=-\dive \dot{\Delta}_j(n^{+} u^{+}).\label{npm:e}
	\end{align}
	Taking the $L^2$ inner product of  $(\ref{npm:e})$ with $\dot{\Delta}_j n^+$  to get
	\begin{align}
		\frac{1}{2}\frac{d}{dt}\|\dot{\Delta}_j n^{+}\|^2_{L^2}=-(\dive \dot{\Delta}_j u^{+}|\dot{\Delta}_j n^{+})_{L^2}-(\dive \dot{\Delta}_j(n^{+} u^{+})|\dot{\Delta}_j 
n^{+})_{L^2}.
	\end{align}
	By using H\"older's inequality and integrating over the time interval $[0,t]$, we have
	\begin{align}\label{inequ-for-n2+}
		\|\dot{\Delta}_j n^{+}\|_{L^2}\lesssim\|\dot{\Delta}_j n_0^+\|^2_{L^2}+\int_0^t (2^{j}\|\dot{\Delta}_j u^{+}\|_{L^2}+\| \dive \dot{\Delta}_j(n^{+} u^{+})\|_{L^2})d\tau.
	\end{align}
	Then, multiplying \eqref{inequ-for-n2+} by $2^{j\frac{d}{2}}$, taking the supremum on $[0,t]$, and then summing over $j\geq-1,$ we have
	\begin{align*}
		\|n^{+}\|^h_{\widetilde{L}^{\infty}_t(\dot{B}_{2,1}^{\frac{d}{2}})}&\lesssim 
\|n_0^{+}\|^h_{\dot{B}_{2,1}^{\frac{d}{2}}}+\|u^{+}\|_{L^1_t(\dot{B}_{2,1}^{\frac{d}{2}+1})}^h+\|\dive  (n^{+}  u^{+})\|_{L^1_t(\dot{B}_{2,1}^{\frac{d}{2}})}.\nonumber
	\end{align*}
	Note that the linear term involving $u^+$  has been addressed in Lemma \ref{lemma2}. Due to the law \eqref{uv1} and the embedding 
$\dot{B}^{\frac{d}{2}}_{2,1}\hookrightarrow L^{\infty}$, one handles the nonlinear term by
	\begin{equation*}
		\begin{aligned}
			\|\dive  (n^{+}  u^{+})\|_{L^1_t(\dot{B}_{2,1}^{\frac{d}{2}})}&\lesssim \|n^{+}\|_{L^{\infty}_t(L^{\infty})}\| 
u^{+}\|_{L^1_t(\dot{B}_{2,1}^{\frac{d}{2}+1})}+\|u^{+}\|_{L^{\infty}_t(L^{\infty})}\| n^{+}\|_{L^1_t(\dot{B}_{2,1}^{\frac{d}{2}+1})}\\
			&\lesssim \|n^\pm\|_{L^{\infty}_t(\dot{B}^{\frac{d}{2}}_{2,1}\cap \dot{B}_{2,1}^{\frac{d}{2}+1})}\|u^\pm\|_{L^{1}_t(\dot{B}^{\frac{d}{2}}_{2,1}\cap 
\dot{B}_{2,1}^{\frac{d}{2}+1})}\\
			&\lesssim \mathcal{X}(t)^2.
		\end{aligned}
	\end{equation*}
	This, together with \eqref{2.23}, shows that
	\begin{align}
		&\|n^{+}\|^h_{\widetilde{L}^{\infty}_t(\dot{B}_{2,1}^{\frac{d}{2}})}\lesssim\mathcal{X}_0+ (1+\mathcal{X}(t)^2)\mathcal{X}(t)^2.\label{nd2}
	\end{align}
	Similarly, we multiply \eqref{inequ-for-n2+} by $2^{j(\frac{d}{2}-1)}$ and then sum over $j\leq0$ to obtain
	\begin{align}
		\|n^{+}\|^{\ell}_{\widetilde{L}^{\infty}_t(\dot{B}_{2,1}^{\frac{d}{2}-1})}&\lesssim 
\|n_0^{+}\|^{\ell}_{\dot{B}_{2,1}^{\frac{d}{2}-1}}+\|u^{+}\|^{\ell}_{L^1_t(\dot{B}_{2,1}^{\frac{d}{2}})}+\| \dive  (n^{+} 
u^{+})\|^{\ell}_{L^1_t(\dot{B}_{2,1}^{\frac{d}{2}-1})}\nonumber\\
		&\lesssim \|n_0^{+}\|^{\ell}_{\dot{B}_{2,1}^{\frac{d}{2}-1}}+\|u^{+}\|^{\ell}_{L^1_t(\dot{B}_{2,1}^{\frac{d}{2}})}+\|n^{+} \|_{L^\infty_t(\dot{B}_{2,1}^{\frac{d}{2}})}\| 
u^{+}\|_{L^1_t(\dot{B}_{2,1}^{\frac{d}{2}})}\nonumber\\
		&\lesssim  \|n_0^{+}\|^{\ell}_{\dot{B}_{2,1}^{\frac{d}{2}-1}}+(1+\mathcal{X}(t)^2)\mathcal{X}(t)^2,\label{nd21}
	\end{align}
	where \eqref{low} and \eqref{uv2} have been used. Combining \eqref{lhl}, \eqref{nd2} and \eqref{nd21}, we obtain \eqref{3.43}.
\end{proof}

\subsection{The non-dissipative components: Higher-order estimates}

We are now in a position to derive the estimate of the norm $\|(n^{+},\theta)\|_{\widetilde{L}^\infty_t(\dot{B}_{2,1}^{d/2+1})}$ required in Subsection \ref{sub2}. In order to overcome the 
difficulty caused by the lack of symmetry, we use the two-phase effective velocity
\begin{align*}
	v=(\nu_1^++\nu_2^+)u^+-(\nu_1^-+\nu_2^-)u^-
\end{align*}
to transfer a loss of derivative in \eqref{1.38}.

\begin{lemma}\label{lemma4}
	Let $T>0$ be any given time, and $(n^+, n^-,u^+,u^-)$ be any strong solution to the Cauchy problem \eqref{equation_R_P_u_pm} for $t\in(0,T)$. Then it holds that
	\begin{equation}\label{3.47}
		\begin{aligned}
			&\|(n^{+},\theta,u^{+},u^{-})\|_{\widetilde{L}_t^{\infty}(\dot{B}^{\frac{d}{2}+1}_{2,1})} +\|(u^{+},u^{-})\|_{\widetilde{L}_t^{2}(\dot{B}^{\frac{d}{2}+2}_{2,1})}\\
            &\quad\lesssim \mathcal{X}_0+ 
\mathcal{X}(t)^{\frac{3}{2}}+ \mathcal{X}(t)^2+\mathcal{X}(t)^4+\|\dive v\|_{L^1_t(\dot{B}^{\frac{d}{2}+1}_{2,1})},
		\end{aligned}
	\end{equation}
	where $\mathcal{X}(t)$ and $v$ are defined in \eqref{Et} and \eqref{effective},  respectively.
\end{lemma}
\begin{proof}
	Initially, we apply  the operator $\dot{\Delta}_j$ to \eqref{equation-for-new-term} and write the resulting equations in terms of a localized quasi-linear system with some commutator 
terms:
	\begin{equation}\label{equation-for-new-term-local}
		\left\{
		\begin{array}{ll}
			\partial_t \dot{\Delta}_jn^++ u^+\cdot \nabla \dot{\Delta}_jn^+  +(1+n^+)\dive \dot{\Delta}_ju^+= T^1_j,\\
			\partial_t \dot{\Delta}_j n^- +u^-\cdot \nabla \dot{\Delta}_j n^- +(1+n^-)\dive \dot{\Delta}_j u^-= T_j^2,\\
			\partial_t \dot{\Delta}_j u^+ + u^+\cdot\nabla \dot{\Delta}_j u^++\tfrac{M_1^1}{\rho^+}\nabla \dot{\Delta}_j n^++\tfrac{M_2^1}{\rho^+}\nabla \dot{\Delta}_j n^- \\
			\quad\quad\quad\quad\quad\quad - \nu_1^+\Delta \dot{\Delta}_j u^+-\nu_2^+\nabla\dive \dot{\Delta}_j u^+=\dot{\Delta}_j  H^++T^3_j,\\
			\partial_t \dot{\Delta}_j u^- + u^-\cdot\nabla \dot{\Delta}_j u^-+\frac{M_1^1}{\rho^-}\nabla \dot{\Delta}_j n^++\tfrac{M_2^1}{\rho^-}\nabla \dot{\Delta}_j n^-\\
			\quad\quad\quad\quad\quad\quad- \nu_1^-\Delta \dot{\Delta}_j u^--\nu_2^-\nabla\dot{\Delta}_j\dive u^-=\dot{\Delta}_j H^-+T_j^4,
		\end{array}
		\right.
	\end{equation}
	where
	\begin{equation}
		\left\{
		\begin{aligned}
			T^1_j&= [u^+,\dot{\Delta}_j]\nabla n^++[n^+,\dot{\Delta}_j]\dive u^+ ,\\
			T^2_j&= [u^-,\dot{\Delta}_j]\nabla n^-+[n^-,\dot{\Delta}_j]\dive u^- ,\\
			T^3_j&= [u^+,\dot{\Delta}_j]\nabla u^+ +\Big[\tfrac{M_1^1}{\rho^+},\dot{\Delta}_j\Big]\nabla n^+ +\Big[\tfrac{M_2^1}{\rho^+},\dot{\Delta}_j\Big]\nabla n^-,\\
			T^4_j &= [u^-,\dot{\Delta}_j]\nabla u^- +\Big[\tfrac{M_1^1}{\rho^-},\dot{\Delta}_j\Big]\nabla n^+ +\Big[\tfrac{M_2^1}{\rho^-},\dot{\Delta}_j\Big]\nabla n^-.
		\end{aligned}
		\right.
	\end{equation}
	Taking the $L^2$ inner product of  $(\ref{equation-for-new-term-local})_{1}$, $(\ref{equation-for-new-term-local})_{2}$, $(\ref{equation-for-new-term-local})_{3}$ and 
$(\ref{equation-for-new-term-local})_{4}$ with $\tfrac{M_1}{(1+n^+)\rho^+}\dot{\Delta}_j n^+$, $\tfrac{M_2}{(1+n^-)\rho^+}\dot{\Delta}_j n^- $, $\dot{\Delta}_j u^+$ and $\dot{\Delta}_ju^-$,  
respectively, and integrating by parts, we have
	\begin{align}\label{inequ-foru+u-n+n-}
		\frac{1}{2}&\frac{d}{dt}\int_{\mathbb{R}^d} 
\Big(|\dot{\Delta}_ju^+|^2+|\dot{\Delta}_ju^-|^2+\tfrac{M_1}{(1+n^+)\rho^+}|\dot{\Delta}_jn^+|^2+\tfrac{M_2}{(1+n^-)\rho^-}|\dot{\Delta}_jn^-|^2\Big)dx\nonumber\\
		& \quad+\nu_1^{+}\int_{\mathbb{R}^d}|\nabla \dot{\Delta}_ju^{+}|^2 dx +\nu_2^{+}\int_{\mathbb{R}^d}|\dive  \dot{\Delta}_ju^+|^2 dx\nonumber\\
        &\quad+\nu_1^{-}\int_{\mathbb{R}^d}|\nabla \dot{\Delta}_ju^{-}|^2 dx +\nu_2^{-}\int_{\mathbb{R}^d}|\dive  \dot{\Delta}_ju^-|^2 dx\nonumber\\
        &\quad+\int_{\mathbb{R}^d}\Big(\tfrac{M_2 }{\rho^+}\nabla\dot{\Delta}_jn^-\cdot \dot{\Delta}_ju^++\tfrac{M_1 }{\rho^-}\nabla\dot{\Delta}_jn^+\cdot\dot{\Delta}_ju^-\Big)dx\nonumber\\
		&\lesssim \Big\|\partial_t\Big(\tfrac{M_1}{(1+n^+)\rho^+}\Big)\Big\|_{L^{\infty}} \|\dot{\Delta}_j n^+\|_{L^2}^2 +\Big\|\partial_t \Big(\tfrac{M_2}{(1+n^-)\rho^-} \Big)
\Big\|_{L^{\infty}}\|\dot{\Delta}_j n^-\|_{L^2}^2\nonumber\\
		&\quad+\Big\|\dive \Big( \tfrac{M_1u^+}{(1+n^+)\rho^+}\Big)\Big\|_{L^{\infty}}\|\dot{\Delta}_j n^+\|_{L^2}^2+\Big\|\dive 
\Big(\tfrac{M_2u^-}{(1+n^-)\rho^-}\Big)\Big\|_{L^{\infty}}\|\dot{\Delta}_j n^-\|_{L^2}^2\nonumber\\
		&\quad+\Big\|\tfrac{M_1}{(1+n^+)\rho^+}\Big\|_{L^{\infty}}\|T_j^1\|_{L^2}\|\dot{\Delta}_j n^+\|_{L^2}+ \Big\|\tfrac{M_2}{(1+n^-)\rho^-}\Big\|_{L^{\infty}}\|T_j^2\|_{L^2}\|\dot{\Delta}_j 
n^-\|_{L^2}\nonumber\\
		&\quad+\|\dot{\Delta}_jH^+\|_{L^2} \|\dot{\Delta}_j u^+\|_{L^2}+\|\dot{\Delta}_jH^-\|_{L^2} \|\dot{\Delta}_j u^- \|_{L^2}\nonumber\\
		&\quad+\|T_j^3\|_{L^2}\|\dot{\Delta}_j u^+\|_{L^2} +\|T_j^4\|_{L^2}\|\dot{\Delta}_j u^-\|_{L^2}\nonumber\\
		&\quad+\Big\|\nabla\Big(\tfrac{M_1}{\rho^+}\Big)\Big\|_{L^{\infty}}\|\dot{\Delta}_j 
n^+\|_{L^{2}}\|\dot{\Delta}_ju^+\|_{L^{2}}+\Big\|\nabla\Big(\tfrac{M_2}{\rho^-}\Big)\Big\|_{L^{\infty}}\|\dot{\Delta}_j n^-\|_{L^{2}}\|\dot{\Delta}_ju^-\|_{L^{2}}.
	\end{align}
	The challenge arises in overcoming the higher-order regularity in the cross terms
	\begin{align}
		\int_{\mathbb{R}^d} \tfrac{M_1 }{\rho^-}\nabla\dot{\Delta}_jn^+\cdot \dot{\Delta}_ju^-+\tfrac{M_2 }{\rho^+}\nabla\dot{\Delta}_jn^-\cdot \dot{\Delta}_ju^+ dx,\label{cross}
	\end{align}
	which corresponds to the non-symmetric part of the system. To eliminate the loss of derivatives of $\dot{\Delta}_jn^\pm$ in \eqref{cross}, we reformulate $(\ref{equation-for-new-term})_1$ and 
$(\ref{equation-for-new-term})_2$ by
	\begin{equation}\label{equation-for-v}
		\left\{
		\begin{array}{ll}
			\partial_t n^++ u^+\cdot \nabla n^+ +\tfrac{1+n^+}{\nu_1^++\nu_2^+}\dive v+ \tfrac{\nu_1^-+\nu_2^-}{\nu_1^++\nu_2^+}(1+n^+)\dive u^-=0,\\
			\partial_t n^- +u^-\cdot \nabla n^- -\frac{1+n^-}{\nu_1^-+\nu_2^-}\dive v+ \tfrac{\nu_1^++\nu_2^+}{\nu_1^-+\nu_2^-}(1+n^-)\dive u^-=0.
		\end{array}
		\right.
	\end{equation}
	Then  applying  the operator $\dot{\Delta}_j$ to \eqref{equation-for-v} to get
	\begin{equation}\label{equation-for-v-local}
		\left\{
		\begin{array}{ll}
			\partial_t \dot{\Delta}_jn^++ u^+\cdot \nabla\dot{\Delta}_j n^+ +\tfrac{1+n^+}{\nu_1^++\nu_2^+}\dive \dot{\Delta}_jv+ \tfrac{\nu_1^-+\nu_2^-}{\nu_1^++\nu_2^+}(1+n^+)\dive 
\dot{\Delta}_j u^-=T_j^5,\\
			\partial_t \dot{\Delta}_j n^- +u^-\cdot \nabla \dot{\Delta}_jn^- -\frac{1+n^-}{\nu_1^-+\nu_2^-}\dive \dot{\Delta}_jv+ \tfrac{\nu_1^++\nu_2^+}{\nu_1^-+\nu_2^-}(1+n^-)\dive 
\dot{\Delta}_j u^+=T_j^6,
		\end{array}
		\right.
	\end{equation}
	where
	\begin{equation}
		\left\{
		\begin{aligned}
			T^5_j&= [u^+,\dot{\Delta}_j]\nabla n^++\tfrac{1}{\nu_1^-+\nu_2^-}[n^+,\dot{\Delta}_j]\dive v  + \tfrac{\nu_1^-+\nu_2^-}{\nu_1^++\nu_2^+}[n^+,\dot{\Delta}_j]\dive u^+ ,\\
			T^6_j&= [u^-,\dot{\Delta}_j]\nabla n^--\tfrac{1}{\nu_1^-+\nu_2^-}[n^-,\dot{\Delta}_j]\dive v + \tfrac{\nu_1^++\nu_2^+}{\nu_1^-+\nu_2^-}[n^-,\dot{\Delta}_j]\dive u^-.
		\end{aligned}
		\right.
	\end{equation}
	Taking the $L^2$ inner product of  $(\ref{equation-for-v-local})_{1}$ and  $(\ref{equation-for-v-local})_{2}$  with
	$$
	\tfrac{\nu_1^++\nu_2^+}{\nu_1^-+\nu_2^-}\tfrac{M_1}{(1+n^+)\rho^-}\dot{\Delta}_j n^+\quad \text{and}\quad 
\tfrac{\nu_1^-+\nu_2^-}{\nu_1^++\nu_2^+}\tfrac{M_2}{(1+n^-)\rho^+}\dot{\Delta}_j n^- ,
	$$
	respectively, we have
	\begin{align}\label{inequforv}
		\frac{1}{2}\frac{d}{dt}&\int_{\mathbb{R}^d} 
\tfrac{\nu_1^++\nu_2^+}{\nu_1^-+\nu_2^-}\tfrac{M_1}{(1+n^+)\rho^-}|\dot{\Delta}_jn^+|^2+\tfrac{\nu_1^-+\nu_2^-}{\nu_1^++\nu_2^+}\tfrac{M_2}{(1+n^-)\rho^+}|\dot{\Delta}_jn^-|^2dx\nonumber\\
		&\quad+\int_{\mathbb{R}^d} \Big( \tfrac{M_1}{\rho^-}\dive \dot{\Delta}_ju^-\dot{\Delta}_jn^++\tfrac{M_2}{\rho^+}\dive \dot{\Delta}_ju^+\dot{\Delta}_jn^- \Big) dx\nonumber\\
		&\lesssim \Big\|\partial_t\Big(\tfrac{M_1}{(1+n^+)\rho^-}\Big)\Big\|_{L^{\infty}} \|\dot{\Delta}_j n^+\|_{L^2}^2 +\Big\|\partial_t \Big(\tfrac{M_2}{(1+n^-)\rho^+} 
\Big)\Big\|_{L^{\infty}}\|\dot{\Delta}_j n^-\|_{L^2}^2\nonumber\\
		&\quad+\Big\|\dive \Big( \tfrac{M_1u^+}{(1+n^+)\rho^-}\Big)\Big\|_{L^{\infty}}\|\dot{\Delta}_j n^+\|_{L^2}^2 +\Big\|\dive 
\Big(\tfrac{M_2u^-}{(1+n^-)\rho^+}\Big)\Big\|_{L^{\infty}}\|\dot{\Delta}_j n^-\|_{L^2}^2 \nonumber\\
		&\quad+\Big\|\tfrac{M_1}{(1+n^+)\rho^-}\Big\|_{L^{\infty}}\|\dive \dot{\Delta}_j v\|_{L^2}\|\dot{\Delta}_j n^+\|_{L^2}+\Big\| \tfrac{M_2}{(1+n^-)\rho^+}\Big\|_{L^{\infty}}\|\dive 
\dot{\Delta}_j v\|_{L^2}\|\dot{\Delta}_j n^-\|_{L^2}\nonumber\\
		&\quad+\Big\|  \tfrac{M_1}{(1+n^+)\rho^-}\Big\|_{L^{\infty}}\|T_j^5\|_{L^2}\|\dot{\Delta}_jn^+\|_{L^2} 
+\Big\|\tfrac{M_2}{(1+n^-)\rho^+}\Big\|_{L^{\infty}}\|T_j^6\|_{L^2}\|\dot{\Delta}_j n^-\|_{L^2}.
	\end{align}
	Note that $M_1, M_2$, $\rho^\pm$ and $\theta$ depend only on $n^\pm$ due to \eqref{given}, and while $n^-$ can be represented by $n^+$ and $n^-$. By the condition \eqref{priori1} with a suitably constant $\var_0$, one knows that $\|(n^+,n^-)\|_{L^{\infty}}$ is sufficiently small. Consequently, there exist two positive constants $c, C$ such that
	\begin{align}
		0<c\leq \tfrac{M_1}{(1+n^+)\rho^+}, \tfrac{M_2}{(1+n^-)\rho^-}\leq C<\infty.\label{positive}
	\end{align}
	By combining \eqref{inequ-foru+u-n+n-} with \eqref{inequforv} and \eqref{positive}, integrating   over $[0,t]$ and then taking the square root of both sides, we have
	\begin{align}\label{ineqfor-n+andu+1}
		&\|\dot{\Delta}_j(n^+,n^-,u^+,u^-)\|_{L_t^{\infty}(L^2)}+\|\nabla\dot{\Delta}_j (u^+,u^-)\|_{L_t^{2}(L^2)}\nonumber\\		
&\lesssim\|\dot{\Delta}_j(n_0^+,n_0^-,u_0^+,u_0^-)\|_{L^2}+\Big\|\partial_t\Big(\tfrac{M_1}{(1+n^+)\rho^-},\tfrac{M_2}{(1+n^-)\rho^+}\Big)\Big\|^{\frac12}_{L^1_t(L^{\infty})}\|\dot{\Delta}_j(n^+,n^-)\|_{L^{\infty}(L^2)}\nonumber\\
		&+\Big\|\dive \Big( \tfrac{M_1u^+}{(1+n^+)\rho^-},\tfrac{M_2u^-}{(1+n^-)\rho^+}\Big)\Big\|^{\frac12}_{L^1_t(L^{\infty})}\|\dot{\Delta}_j (n^+,n^-)\|_{L^{\infty}_{t}(L^2)} 
\nonumber\\
&+\Big\|\nabla\Big(\tfrac{M_1}{\rho^+},\tfrac{M_2}{\rho^-}\Big)\Big\|^{\frac12}_{L_t^{\infty}(L^{\infty})}\|\dot{\Delta}_j (n^+,n^-)\|^{\frac12}_{L_t^{\infty}(L^2)}\|\dot{\Delta}_j(u^+,u^-)\|^{\frac12}_{L_t^{1}(L^{2})}\nonumber\\
		&+\|\dot{\Delta}_j(H^+,H^-)\|^{\frac12}_{L_t^{2}(L^2)}\|\dot{\Delta}_j(u^+,u^-)\|^{\frac12}_{L_t^{2}(L^2)}+\|\dive \dot{\Delta}_j 
v\|^{\frac12}_{L_t^{1}(L^2)}\|\dot{\Delta}_j(n^+,n^-)\|^{\frac12}_{L_t^{\infty}(L^2)}\nonumber\\
		&+ \|(T_j^1,T_j^2)\|^{\frac12}_{L_t^{1}(L^{2})}\|\dot{\Delta}_j(n^+,n^-)\|^{\frac{1}{2}}_{L_t^{\infty}(L^2)}+\|(T_j^3,T_j^4)\|^{\frac12}_{L_t^{\infty}(L^2)}\|\dot{\Delta}_j(u^+,u^-)\|^{\frac{1}{2}}_{L_t^{1}(L^2)}\nonumber\\
		&+\|(T_j^5,T_j^6)\|^{\frac12}_{L_t^{1}(L^{2})}\|\dot{\Delta}_j(n^+,n^-)\|^{\frac{1}{2}}_{L_t^{\infty}(L^2)}.
	\end{align}
	Multiplying \eqref{ineqfor-n+andu+1} by $2^{j(\frac{d}{2}+1)}$,  taking the supremum on $[0,t]$, and then summing over $j\in\mathbb{Z}$ we have
	\begin{align}
		&\|(n^{+},n^{-},u^+,u^-)\|_{\widetilde{L}_t^{\infty}(\dot{B}^{\frac{d}{2}+1}_{2,1})}+\|(u^+,u^-)\|_{\widetilde{L}_t^{2}(\dot{B}^{\frac{d}{2}+2}_{2,1})}\nonumber\\
        &\quad\leq 
\|(n_0^{+},n_0^{-},u_0^{+},u_0^{-})\|_{\dot{B}^{\frac{d}{2}+1}_{2,1}} 
+\|\partial_t(n^{+},n^{-})\|^{\frac12}_{L^1_t(L^{\infty})}\|(n^{+},n^{-})\|_{\widetilde{L}^{\infty}(\dot{B}^{\frac{d}{2}+1}_{2,1})}\nonumber\\
		&\quad\quad+\|\nabla 
(n^{+},n^{-})\|^{\frac12}_{L_t^{\infty}(L^{\infty})}\|(n^{+},n^{-})\|^{\frac12}_{\widetilde{L}_t^{\infty}(\dot{B}^{\frac{d}{2}+1}_{2,1})}\|(u^+,u^-)\|^{\frac12}_{L_t^{1}(\dot{B}^{\frac{d}{2}+1}_{2,1})}\nonumber\\
&\quad\quad+\|(u^+,u^-)\|_{L^1_t(W^{1,\infty})}^{\frac{1}{2}}\|(n^+,n^-)\|_{\widetilde{L}^{\infty}(\dot{B}^{\frac{d}{2}+1}_{2,1})}\nonumber\\
		&\quad\quad+\| (H^+,H^-)\|^{\frac12}_{\widetilde{L}_t^{2}(\dot{B}^{\frac{d}{2}}_{2,1})}\|(u^+,u^-)\|^{\frac12}_{\widetilde{L}_t^{2}(\dot{B}^{\frac{d}{2}+2}_{2,1})}+\|\dive  v\|^{\frac12}_{L^1_t 
(\dot{B}^{\frac{d}{2}+1}_{2,1})}\| n^{\pm}\|^{\frac12}_{\widetilde{L}_t^{\infty}(\dot{B}^{\frac{d}{2}+1}_{2,1})}\nonumber\\
		&\quad\quad+\sum_{j\in\mathbb{Z}}\left(2^{j(\frac{d}{2}+1)}\|(T_j^1,T_j^2,T_j^5,T_j^6)\|_{L_t^{1}(L^{2})}\right)^{\frac12}\| (n^{+},n^{-})  
\|^{\frac{1}{2}}_{\widetilde{L}_t^{\infty}(\dot{B}^{\frac{d}{2}+1}_{2,1})}\nonumber\\
		&\quad\quad+ \sum_{j\in\mathbb{Z}}\left(2^{j(\frac{d}{2}+1)}\|(T_j^3,T_j^4)\|_{L_t^{1}(L^{2})}\right)^{\frac12}\|(u^+,u^-)\|^{\frac{1}{2}}_{L_t^{1}(\dot{B}^{\frac{d}{2}+1}_{2,1})}.\label{n10}
	\end{align}
	First, as $n^-=n^-(n^+,\theta)$ and $\theta=\theta(n^+,n^-)$, one has
    \begin{align}
    &\|n^-_0\|_{\dot{B}^{\frac{d}{2}+1}_{2,1}}\lesssim \|(n_0^+,\theta)\|_{\dot{B}^{\frac{d}{2}+1}_{2,1}}\lesssim \mathcal{X}_0,\label{n11}\\
   & \|n^-\|_{\widetilde{L}^{\infty}_{t}(\dot{B}^{\frac{d}{2}+1}_{2,1})}\lesssim \|(n^+,\theta)\|_{\widetilde{L}^{\infty}_{t}(\dot{B}^{\frac{d}{2}+1}_{2,1})}\lesssim \mathcal{X}(t),\label{n12}\\
   &\|\theta\|_{\widetilde{L}^{\infty}_{t}(\dot{B}^{\frac{d}{2}+1}_{2,1})}\lesssim \|(R^+,R^-)\|_{\widetilde{L}^{\infty}_{t}(\dot{B}^{\frac{d}{2}+1}_{2,1})}.\label{n13}
    \end{align}
    Due to  $\eqref{equation-for-new-term}_1$,  $\eqref{equation-for-new-term}_2$ and \eqref{n12}, we obtain
	\begin{align}
		\|\partial_tn^{\pm}\|_{L^1_t(L^{\infty})}&\lesssim\|\dive u^{\pm}\|_{L^1_t(L^{\infty})}+\|u^{\pm}\|_{L^1_t(L^{\infty})}\|\nabla n^{\pm}\|_{L^{\infty}_t(L^{\infty})}+\|\nabla 
u^{\pm}\|_{L^1_t(L^{\infty})}\|n^{\pm}\|_{L^{\infty}_t(L^{\infty})}\nonumber\\
		&\lesssim \| u^{\pm}\|_{L^1_t(\dot{B}^{\frac{d}{2}+1}_{2,1})}+\|u^{\pm}\|_{L^1_t( \dot{B}^{\frac{d}{2}+1}_{2,1})}\|  n^{\pm}\|_{\widetilde{L}^{\infty}_t( \dot{B}^{\frac{d}{2}+1}_{2,1})}+\|  
u^{\pm}\|_{L^1_t (\dot{B}^{\frac{d}{2}+1}_{2,1})}\|n^{\pm}\|_{\widetilde{L}^{\infty}_t( \dot{B}^{\frac{d}{2}+1}_{2,1})}\nonumber\\
		&\lesssim \mathcal{X}_0+ (1+\mathcal{X}(t)^2)\mathcal{X}(t)^2.
	\end{align}
	With regard to the nonlinear terms for $H^\pm$, we deduce that
	\begin{align}
		\| H^\pm\|_{\widetilde{L}_t^{2}(\dot{B}^{\frac{d}{2}}_{2,1})}&\lesssim\|\theta\|_{\widetilde{L}_t^{\infty}(\dot{B}^{\frac{d}{2} 
}_{2,1})}\|u^{\pm}\|_{\widetilde{L}_t^{2}(\dot{B}^{\frac{d}{2}+2}_{2,1})} 
+\|u^{\pm}\|_{\widetilde{L}_t^{2}(\dot{B}^{\frac{d}{2}}_{2,1})}\|n^{\pm}\|_{\widetilde{L}_t^{\infty}(\dot{B}^{\frac{d}{2}+1}_{2,1})}\lesssim \mathcal{X}(t)^2.
	\end{align}
	To bound the terms $T_j^{k}$ with $k=1,2,\cdots,6 $, we first consider $T_j^1$.
	Using \eqref{n12} and the commutator estimates in Lemma \ref{lemcommutator}, we have that
	\begin{align}\label{commutator1}
		&\sum_{j\in\mathbb{Z}}\left(2^{j(\frac{d}{2}+1)}\|(T_j^1,T_j^2,T_j^5,T_j^6)\|_{L_t^{1}(L^{2})}\right) \nonumber\\
        &\quad \lesssim \|(\theta,u^{+},u^-,\dive v)\|_{L^{1}_t (\dot{B}^{\frac{d}{2}+1}_{2,1})} \|(n^+,\theta,u^{+},u^-)\|_{\widetilde{L}^{\infty}_t 
(\dot{B}^{\frac{d}{2}+1}_{2,1})} \lesssim \mathcal{X}^2(t).
	\end{align}
	Similarly, we have
	\begin{align}
		\sum_{j\in\mathbb{Z}}\left(2^{j(\frac{d}{2}+1)}\|(T_j^3,T_j^4)\|_{L_t^{\infty}(L^2)}\right) \lesssim \|(n^{+},\theta,u^+,u^-)\|_{\widetilde{L}^{\infty}_t (\dot{B}^{\frac{d}{2}+1}_{2,1})}^2   \lesssim \mathcal{X}^2(t).\label{n100}
	\end{align}
	Collecting \eqref{n10}--\eqref{n100} and using Young's inequality, we have \eqref{3.47}.
\end{proof}

\subsection{The estimate of effective unknowns}\label{sub3.5}

In order to finish the proof, we shall establish a higher-order dissipation estimate of the two-phase effective velocity $v$ defined in \eqref{effective}. Note that by (\ref{equation_R_P_u_pm}), $v$ satisfies the 
following elliptic system
\begin{align}
	 \dive v=(-\Delta)^{-1}\dive(- \bar{\rho}^+ \partial_t u^++\bar{\rho}^-\partial_t u^-+\bar{\rho}^+ F_2^+-\bar{\rho}^- F_2^-).\label{v:eq}
\end{align}
Then, 
since the degree of $(-\Delta)^{-1}\dive$ is $-1$, we have 
\begin{align}\label{3.68}
	\|\dive v\|_{L^{1}_t(\dot{B}^{\frac{d}{2}+1}_{2,1})}&\lesssim \|\partial_t (u^{+},u^{-})\|_{L^{1}_t(\dot{B}^{\frac{d}{2}}_{2,1})}+\|(F_2^+,F_2^-)\|_{L^{1}_t(\dot{B}^{\frac{d}{2}}_{2,1})}.
\end{align}
Then, one needs to analyze the regularity of the time derivatives $u_t^{\pm}$. To this end, we introduce an effective viscous flux
$$
w^{\pm}=\beta_2^{\pm}\nabla \theta-\nu_1^{\pm}\Delta u^{\pm}-\nu_2^{\pm}\nabla\dive u^{\pm},
$$
which allows us to write $u_t= w^\pm$ up to some nonlinear terms. Compared with \cite{WYZ-MA}, this allows us to avoid the direct analysis of the initial data for $u_t^{\pm}$ requiring the 
compatibility condition.

\begin{lemma}\label{lemma3}
	Let $T>0$ be any given time, and $(n^+, n^-,u^+,u^-)$ be any strong solution to the Cauchy problem \eqref{equation_R_P_u_pm} for $t\in(0,T)$. Then it holds that
	\begin{equation}\label{3.70}
		\begin{aligned}
			\|\dive v\|_{L^1_t(\dot{B}^{\frac{d}{2}+1}_{2,1})}+\|(w^{+},w^{-},\partial_t u^+,\partial_t u^-)\|_{L_t^{1}(\dot{B}^{\frac{d}{2}}_{2,1})} \lesssim \mathcal{X}_0+ \mathcal{X}(t)^2+\mathcal{X}(t)^4,
		\end{aligned}
	\end{equation}
	where $\mathcal{X}(t)$ is defined in \eqref{Et}.
\end{lemma}

\begin{proof}
	Taking the time derivative of system $ (\ref{equation_R_P_u_pm})_3$ and $ (\ref{equation_R_P_u_pm})_4$,  we get the following Lam\'e system for $w^{\pm}$:
	\begin{equation} \label{equationforw}
		\begin{array}{ll}
			&\partial_t w^\pm - \nu_1^{\pm}\Delta w^\pm-\nu_2^\pm\nabla\dive w^\pm\\
            &\quad =\beta_2^\pm \nabla(-\beta_1^+\dive u^+-\beta_1^- \dive u^-+F_1)  -\nu_1^{\pm}\Delta F_2^\pm-\nu_2^{\pm}\nabla\dive F_2^\pm .
		\end{array}
	\end{equation}
	where $F_1$ and $F_2^\pm$ are given in \eqref{1.46}.  
Then, by making use of the maximal regularity estimate (see Lemma \ref{lame-system-estimate}) for \eqref{equationforw},   we have
	\begin{equation}
		\begin{aligned}\label{inequ-w+-}
		&\|w^{\pm}\| _{\widetilde{L}_t^{\infty}(\dot{B}^{\frac{d}{2}-2}_{2,1})}+	\|w^{\pm}\| _{L_t^{1}(\dot{B}^{\frac{d}{2}}_{2,1})}\\
        &\quad\lesssim \|\theta_0\| _{\dot{B}^{\frac{d}{2}-1}_{2,1}}+\|(u_0^{+},u_0^-)\| _{\dot{B}^{\frac{d}{2}}_{2,1}}+\|u^\pm\|_{L^1_t(\dot{B}^{\frac{d}{2}}_{2,1})}+\|F_1\| _{ L_t^{1}(\dot{B}^{\frac{d}{2}-1}_{2,1})}
			+\|F_2^\pm\| _{ L_t^{1}(\dot{B}^{\frac{d}{2}}_{2,1})}.
		\end{aligned}
	\end{equation}
With the estimates \eqref{low}, \eqref{F1ell},  \eqref{high1} and \eqref{F1high} at hand, one gets
\begin{align}
\|u^\pm\|_{L^1_t(\dot{B}^{\frac{d}{2}}_{2,1})}+\|F_1\| _{ L_t^{1}(\dot{B}^{\frac{d}{2}-1}_{2,1})}\lesssim \mathcal{E}_0+(1+\mathcal{E}(t)^2)\mathcal{E}(t)^2.\label{dggh}
\end{align}
	Next, we analyze the last nonlinear term on the right-hand side of \eqref{inequ-w+-}.
	Using \eqref{uv2} and \eqref{F1} yields
	\begin{align}\label{11p}
		&\|g_1^{\pm}(\theta)\Delta u^{\pm}\| _{ L_t^{1}(\dot{B}^{\frac{d}{2}}_{2,1})}+\|g_2^{\pm}(\theta)\nabla\dive u^{\pm}\| _{ L_t^{1}(\dot{B}^{\frac{d}{2}}_{2,1})}\nonumber\\
		&\quad\lesssim \|g_1^{\pm}(\theta)\|_{ L_t^{2}(\dot{B}^{\frac{d}{2}}_{2,1})}\|u^{\pm}\|_{ L_t^{2}(\dot{B}^{\frac{d}{2}+2}_{2,1})}+\|g_2^{\pm}(\theta)\|_{ 
L_t^{2}(\dot{B}^{\frac{d}{2}}_{2,1})}\|u^{\pm}\|_{ L_t^{2}(\dot{B}^{\frac{d}{2}+2}_{2,1})}\nonumber\\
		&\quad\lesssim \| \theta\|_{ L_t^{2}(\dot{B}^{\frac{d}{2}}_{2,1})}\|u^{+}\|_{ L_t^{2}(\dot{B}^{\frac{d}{2}+2}_{2,1})}\\
        &\quad \lesssim \| \theta\|_{ \widetilde{L}_t^{\infty}(\dot{B}^{\frac{d}{2}}_{2,1})}^{\frac{1}{2}} \| \theta\|_{ L_t^{1}(\dot{B}^{\frac{d}{2}}_{2,1})}^{\frac{1}{2}} \|u^{+}\|_{ L_t^{2}(\dot{B}^{\frac{d}{2}+2}_{2,1})} \lesssim \mathcal{X}(t)^2.
	\end{align}
	And similarly, we obtain
	\begin{align*}
		\|u^\pm\cdot\nabla &u^\pm\| _{ L_t^{1}(\dot{B}^{\frac{d}{2}}_{2,1})}+\|g_4^{\pm}(\theta)\nabla \theta\| _{L^1_t (\dot{B}_{2,1}^{\frac{d}{2} })}\nonumber\\
        &\lesssim \|u^\pm\|_{ \widetilde{L}_t^{\infty}(\dot{B}^{\frac{d}{2}}_{2,1})}\|u^\pm\|_{ 
        L_t^{1}(\dot{B}^{\frac{d}{2}+1}_{2,1})}+\|\nabla\theta\|_{\widetilde{L}_t^{\infty}(\dot{B}^{\frac{d}{2}}_{2,1})}  \|\theta\|_{L^1_t (\dot{B}_{2,1}^{\frac{d}{2} })}\lesssim 
        \mathcal{X}(t)^2.
	\end{align*}
	It is sufficient to deal with $ \frac{\nabla u^\pm \cdot \nabla \alpha^\pm}{R^\pm}$ since the other terms are essentially the same. Recalling \eqref{decom++}, we get
	\begin{align*}
		&\Big\|\tfrac{\nabla u^\pm  \cdot \nabla \alpha^\pm }{R^\pm }\Big\| _{ L_t^{1}(\dot{B}^{\frac{d}{2}}_{2,1})}\nonumber\\
		\lesssim& \Big\|\nabla u^\pm \cdot\nabla h^\pm(n^+,\theta)\cdot \tfrac{n^\pm}{n^\pm+1}\Big\| _{ L_t^{1}(\dot{B}^{\frac{d}{2}}_{2,1})}+\|\nabla u^\pm \nabla h^\pm(n^+,\theta) \| _{ 
L_t^{1}(\dot{B}^{\frac{d}{2}}_{2,1})}\nonumber\\
		\lesssim& (1+\|n^\pm \|_{ \widetilde{L}_t^{\infty}(\dot{B}^{\frac{d}{2}}_{2,1})})\|n^\pm \|_{ \widetilde{L}_t^{\infty}(\dot{B}^{\frac{d}{2}+1}_{2,1})}\|u^\pm \|_{ 
L_t^{1}(\dot{B}^{\frac{d}{2}+1}_{2,1})}\nonumber\\
		\lesssim  &\mathcal{X}(t)^2+\mathcal{X}(t)^3.
	\end{align*}
	Consequently, we have
	\begin{align}\label{inequforH1}
		\|F_2^\pm\| _{ L_t^{1}(\dot{B}^{\frac{d}{2}}_{2,1})}\lesssim  \mathcal{X}(t)^2+\mathcal{X}(t)^3.
	\end{align}
By \eqref{inequ-w+-}, \eqref{inequforH1} and \eqref{dggh}, it holds that
	\begin{align}\label{3.81}
		\|w^{\pm}\|_{\widetilde{L}_t^{\infty}(\dot{B}^{\frac{d}{2}-2}_{2,1})}+\|w^{\pm}\|_{L_t^{1}(\dot{B}^{\frac{d}{2}}_{2,1})} \lesssim \mathcal{X}_0+ 
\mathcal{X}(t)^2+\mathcal{X}(t)^4.
	\end{align}
	
	Finally, we recover the desired estimate of $\partial_t u^{\pm}$. Since
	\begin{align*}
		\partial_t u^\pm+w^\pm=F_2^\pm,
        \end{align*}
	we conclude from \eqref{inequforH1} and \eqref{3.81} that
	\begin{align}
		\|\partial_t u^\pm\|_{ L_t^{1}(\dot{B}^{\frac{d}{2}}_{2,1})}&\lesssim \|w^\pm\|_{ L_t^{1}(\dot{B}^{\frac{d}{2}}_{2,1})}+\|F_2^\pm\|_{ L_t^{1}(\dot{B}^{\frac{d}{2}}_{2,1})}\lesssim \mathcal{X}_0+ \mathcal{X}(t)^2+\mathcal{X}(t)^4.\label{3.84}
	\end{align}
	Then, we obtain \eqref{3.70} holds by combining \eqref{3.68}, \eqref{inequforH1}, \eqref{3.81} and \eqref{3.84}. This finishes the proof of Lemma \ref{lemma3}.
\end{proof}

\subsection{Global existence}
In this subsection, we construct a local Friedrichs approximation (cf. \cite{bahouri1}) and  extend the local approximate sequence to a global one by the uniform-in-time a priori estimates  
established in Subsections \ref{sub1}--\ref{sub2}. Then we show the convergence of the approximate sequence to the expected global solution to the Cauchy problem for the system  \eqref{system1}--\eqref{system5}. Define the Friedrichs projector $\dot{\mathbb{E}}_q$ by $\dot{\mathbb{E}}_q f:=\mathcal{F}^{-1}(\textbf{1}_{\mathcal{C}_q}\mathcal{F}f)$ for $ f\in L^2_q$, where $L^2_q$ is the  set of $L^2(\mathbb{R}^d)$ functions spectrally supported in the annulus $\mathcal{C}_q:=\{\xi\in\mathbb{R}^d|\frac{1}{q}\leq |\xi|\leq q\}$ endowed with the standard $L^2$ 
topology, and $\textbf{1}_{\mathcal{C}_q}$ is the characteristic function on $\mathcal{C}_q$.

We now consider the following approximate problems ($q\geq1$) for  \eqref{equation_R_P_u_pm}:
\begin{equation}\label{app1}
	\left\{
	\begin{array}{ll}
		\partial_t n^{+}+ \dot{\mathbb{E}}_q(u^+\cdot \nabla n^{+})+\dot{\mathbb{E}}_q ((1+n^{+})\dive u^{+})=0,\\
		\partial_t n^{-} +\dot{\mathbb{E}}_q(u^-\cdot \nabla n^{-})+\dot{\mathbb{E}}_q((1+n^{-})\dive u^{-})=0,\\
		\partial_t u^{+}+ \dot{\mathbb{E}}_q(u^{+}\cdot\nabla u^{+})+\dot{\mathbb{E}}_q(\tfrac{M_1}{\rho^+}\nabla n^{+}+\tfrac{M_2}{\rho^+}\nabla n^-) - \nu_1^+\Delta u^{+}-\nu_2^+ 
\nabla\dive u^{+}=\dot{\mathbb{E}}_qH^+,\\
		\partial_t u^{-} +\dot{\mathbb{E}}_q( u^{-}\cdot\nabla u^{-})+\dot{\mathbb{E}}_q(\tfrac{M_1}{\rho^-}\nabla n^++\tfrac{M_2}{\rho^-}\nabla n^-)- \nu_1^-\Delta 
u^{-}-\nu_2^-\nabla\dive u^{-}=\dot{\mathbb{E}}_q H^-,\\
		( n^{\pm}, u^{\pm})|_{t=0}=(\dot{\mathbb{E}}_q n_0^\pm, \dot{\mathbb{E}}_qu_0^\pm).
	\end{array}
	\right.
\end{equation}
Introducing the dissipative variable $\theta=P(n^+,n^-)-\bar{P}$, we can rewrite \eqref{app1} as
\begin{equation}\label{approximate-system}
	\left\{
	\begin{array}{ll}
		\partial_t n^{+} +\dot{\mathbb{E}}_q (u^{+} \cdot\nabla n^{+} ) +\dot{\mathbb{E}}_q \big((1+n^{+} )\dive u^{+}\big) =0,\\
		\partial_t \theta +\beta_1^+\dive \dot{\mathbb{E}}_q u^{+}+\beta_1^-\dive \dot{\mathbb{E}}_q u^{-}= F^q_1,\\
		\partial_t u^{+} +\beta_2^+\nabla\dot{\mathbb{E}}_q \theta- \nu_1^{+}\Delta \dot{\mathbb{E}}_q u^{+}-\nu_2^+\nabla\dive \dot{\mathbb{E}}_q u^{+}=\dot{\mathbb{E}}_q F_2^+,\\
		\partial_t u^{-} +\beta_2^-\nabla\dot{\mathbb{E}}_q \theta- \nu_1^{-}\Delta \dot{\mathbb{E}}_q u^{-}-\nu_2^-\nabla\dive  \dot{\mathbb{E}}_q u^{-}=\dot{\mathbb{E}}_q F_2^-,\\
		(n^{+},\theta,u^{+},u^{-})(x,0)=(\dot{\mathbb{E}}_q n^{+}_0,\dot{\mathbb{E}}_q\theta_0,\dot{\mathbb{E}}_q u^{+}_0,\dot{\mathbb{E}}_q u^{-}_0)(x).\ &
	\end{array}
	\right.
\end{equation}
where the nonlinear terms $\mathcal{H}^\pm$ and $\mathcal{F}^+$ are given by \eqref{1.43} and \eqref{1.46}, respectively, and
\begin{equation}
\begin{aligned}
F^q_1=-\mathcal{C} \rho^- \dot{\mathbb{E}}_q  \dive( (1+n^+) u^+)+\beta_1^+\dot{\mathbb{E}}_q  \dive u^+ -\mathcal{C} \rho^+ \dot{\mathbb{E}}_q \dive ( (1+n^-) u^-)+\beta_1^-\dot{\mathbb{E}}_q  \dive u^-.
\end{aligned}
\end{equation}
It is classical to show that $(\dot{\mathbb{E}}_q n^{+}_0,\dot{\mathbb{E}}_q\theta_0,\dot{\mathbb{E}}_q u^{+}_0,\dot{\mathbb{E}}_q u^{-}_0)$ satisfies \eqref{initialenergy} uniformly with 
respect to $q \geq  1$ and converges to  $( n^{+}_0, \theta_0,  u^{+}_0,  u^{-}_0)$ strongly in the sense \eqref{initialenergy}.  Since all the Sobolev norms are equivalent in 
\eqref{approximate-system} due to the Bernstein inequality, we can check that \eqref{app1} is a system of ordinary differential equations  in 
$L^2_q\times L^2_q\times L^2_q\times L^2_q $ and  locally Lipschitz with respect to the variable $(n^+,\theta, u^+, u^-)$ for every $q\geq 1$.  By virtue of the Cauchy-Lipschitz theorem 
(see \cite[Theorem 3.2]{bahouri1}),  there exists a maximal time $T^*_q>0$ such that the problem \eqref{app1} admits a unique solution $(n^{+,q},n^{-,q}, u^{+,q}, 
u^{-,q})\in C([0,T^*_q);L^2_q)$. Then, $(n^{+,q},\theta^{q}, u^{+,q}, 
u^{-,q})$ with $\theta^{q}=P(n^{+,q},n^{-,q})-\bar{P}$ solves \eqref{approximate-system} on $[0,T^*_q)\times \mathbb{R}^d$.

Define the maximal time
\begin{align}
	T_q:=\sup\{t\geq0~:~& \mathcal{X}^q(t)\leq 2C_0\mathcal{X}_{0}  \},
\end{align}
where $\mathcal{X}^q(t)$ is the exact same as \eqref{Et} but for  $(n^{\pm,q},\theta^q, u^{+,q}, u^{-,q})$, $\mathcal{X}_0$ is defined in \eqref{E0} independent of $q$, and $C_0$ is given by Proposition \ref{propapriori}. It is clear that $0<T_q\leq T^*_q.$

We claim $T_q= T^*_q$. Otherwise, we assume that $T_q< T^*_q$ holds. Since $(n^{+,q},n^{-,q}, u^{+,q}, u^{-,q})=\dot{\mathbb{E}}_q (n^{+,q},n^{-,q}, u^{+,q}, u^{-,q})$, the orthogonal
projector $\dot{\mathbb{E}}_q$ has no effect on the energy estimates for \eqref{app1}--\eqref{approximate-system} carried out in Lemmas \ref{lemma1}--\ref{lemma3}. Due to $\|(n^{+,q},\theta^q)\|_{L^{\infty}_t(L^{\infty})}\lesssim \mathcal{X}^q(t)$, one can let $\mathcal{X}_0\leq \var_0^*$ such that the condition \eqref{priori1} holds. Then, Proposition \ref{propapriori} ensures that
\begin{align}\label{inequ-for-E1}
	\mathcal{X}^q(t)\leq C_0\mathcal{X}_0 +C_0\left(1+\mathcal{X}^q(t)^2\right)\mathcal{X}^q(t)^2.
\end{align}
By \eqref{inequ-for-E1}, we take
$$\mathcal{X}_0\leq \varepsilon_0:=\Big\{\var_0^*,1,\frac{1}{6C_0}\Big\}
$$
to derive
\begin{align}\label{inequ-for-E2}
	\mathcal{X}^q(t)\leq\frac{3}{2} C_0\mathcal{X}_0, \quad 0<t<T_q.
\end{align}
By \eqref{inequ-for-E2} and the time continuity property, $T_q$ is not the maximal time such that $\mathcal{X}(t)\leq C_0\mathcal{X}_{0}$ holds.  This contradicts the definition of $T_q$.

If $T^*_q<\infty$, by \eqref{inequ-for-E2} and $T_q= T^*_q$, we can take $(n^{+,q},\theta^q, u^{+,q}, u^{-,q})(t)$ for $t$ sufficiently close to $T^*_q$ as the new initial data and
obtain the existence from $t$ to some $t+\eta^*>T^*_q$ with a sufficiently small constant $\eta^*>0$ by the Cauchy–Lipschitz theorem, which contradicts the definition of $T^*_q$. Therefore, we have $T^*_q=\infty$, and $\{(n^{+,q},\theta^q, u^{+,q}, u^{-,q})\}_{q\ge1}$  is a global-in-time approximate sequence.

Then, utilizing the uniform bound $\mathcal{X}^q(t)\leq C_0\mathcal{X}_{0} $ following a standard compactness argument (see \cite[Pages 458-459]{bahouri1}), one can show that as $q\rightarrow \infty$, $(n^{+,q},\theta^q, u^{+,q}, u^{-,q})$ converges strongly to a limit $(n^+,n^-,u^+,u^-)$ which is a global-in-time solution to the Cauchy problem 
\eqref{equation_R_P_u_pm} satisfying $\mathcal{X}(t)\lesssim \mathcal{X}_0$. Let $R^\pm:=n^\pm+1$. Due to \eqref{given} and the implicit function theorem, we can obtain the functions
$\alpha^\pm=\alpha^\pm(R^+,R^-)$, $\rho^\pm=\rho^\pm(R^+,R^-)$ and $P=P(R^+,R^-)$. It is not difficult to check that $(\alpha^\pm,\rho^\pm,u^\pm)$ is the desired global solution to the original Cauchy problem  \eqref{system1}--\eqref{system5} and fulfills the properties \eqref{r1} and \eqref{e1}. 
 To complete the proof of Theorem \ref{Thm1.1}, we finally prove the uniqueness issue below.

\subsection{Uniqueness}
 It suffices to consider the reformulated system \eqref{equation-for-new-term}. For any given time $T>0$, let $(n_1^+, n_1^-,u_1^+,u_1^-)$ and  $(n_2^+, n_2^-,u_2^+,u_2^-)$  be two solutions 
 to \eqref{equation-for-new-term} on $[0,T]\times\mathbb{R}^d$ subject to the same initial data $(n_0^+, n_0^-,u_0^+,u_0^-)$ and satisfying the regularities given by Theorem \ref{Thm1.1}.
The difference
$$
(\widetilde{n}^+,\widetilde{n}^-,\widetilde{u}^+,\widetilde{u}^-)=(n_1^+-n_2^+, n_1^--n_2^-,u_1^+-u_2^+,u_1^--u_2^-)
$$
solves
\begin{equation}\label{3.84e}
	\left\{
	\begin{array}{ll}
		\partial_t \widetilde{n}^++ u_1^+\cdot \nabla \widetilde{n}^+  +(1+n_1^+)\dive \widetilde{u}^+=\widetilde{S}_1,\\
		\partial_t \widetilde{n}^- +u_1^-\cdot \nabla \widetilde{n}^- +(1+n_1^-)\dive \widetilde{u}^-=\widetilde{S}_2,\\
		\partial_t \widetilde{u}^+ + u_1^+\cdot\nabla \widetilde{u}^++\tfrac{M_1^1}{\rho_1^+}\nabla \widetilde{n}^++\tfrac{M_2^1}{\rho_1^+}\nabla \widetilde{n}^- - \nu_1^+\Delta 
\widetilde{u}^+-\nu_2^+\nabla\dive \widetilde{u}^+=\widetilde{S}_3,\\
		\partial_t \widetilde{u}^- + u_1^-\cdot\nabla \widetilde{u}^-+\tfrac{M_1^1}{\rho_1^-}\nabla \widetilde{n}^++\tfrac{M_2^1}{\rho_1^-}\nabla \widetilde{n}^-- \nu_1^-\Delta 
\widetilde{u}^--\nu_2^-\nabla\dive \widetilde{u}^-=\widetilde{S}_4,
	\end{array}
	\right.
\end{equation}
where we denoted $\rho_1^\pm= \rho^\pm(n_1^+,n_1^-)$, $M_1^1=(\mathcal{C}\rho^-)(n_1^+,n_1^-)$, $M_2^1=(\mathcal{C}\rho^+)(n_1^+,n_1^-)$ and
\begin{equation}
	\left\{
	\begin{aligned}
		\widetilde{S}_1&=-\widetilde{u}^+\cdot\nabla n_2^+-\widetilde{n}^+\dive u_2^+,\\
		\widetilde{S}_2&=-\widetilde{u}^-\cdot\nabla n_2^--\widetilde{n}^-\dive u_2^-,\\
		\widetilde{S}_3&=-\widetilde{u}^+\cdot\nabla u_2^+-\left(\tfrac{M_1^1}{\rho_1^+}-\tfrac{M_1^2}{\rho_2^+}\right)\nabla n^+_2 
-\left(\tfrac{M_2^1}{\rho_1^+}-\tfrac{M_2^2}{\rho_2^+}\right)\nabla n^-_2\\
		&\quad+\left(\tfrac{\mu^+}{\rho_1^+}-\tfrac{\mu^+}{\rho_2^+}\right)\Delta u_2^++\left(\tfrac{\mu^++\lambda^+}{\rho_1^+}-\tfrac{\mu^++\lambda^+}{\rho_2^+}\right)\nabla\dive u_2^+\\
		&\quad+H^+(n_1^+,n_1^-,u_1^+,u_1^-)-H^+(n_2^+,n_2^-,u_2^+,u_2^-),\\
		\widetilde{S}_4&=-\widetilde{u}^-\cdot\nabla u_2^--\left(\tfrac{M_1^1}{\rho_1^-}-\tfrac{M_1^2}{\rho_2^-}\right)\nabla n^-_2 
-\left(\tfrac{M_2^1}{\rho_1^-}-\tfrac{M_2^2}{\rho_2^-}\right)\nabla n^-_2\\
		&\quad+\left(\tfrac{\mu^-}{\rho_1^-}-\tfrac{\mu^-}{\rho_2^-}\right)\Delta u_2^-+\left(\tfrac{\mu^-+\lambda^-}{\rho_1^-}-\tfrac{\mu^-+\lambda^-}{\rho_2^-}\right)\nabla\dive u_2^-\\
		&\quad +H^-(n_1^+,n_1^-,u_1^+,u_1^-)-H^-(n_2^+,n_2^-,u_2^+,u_2^-).
	\end{aligned}
	\right.
\end{equation}
Applying $\dot{\Delta}_j$ to \eqref{3.84} yields
\begin{equation}
	\left\{
	\begin{array}{ll}
		\partial_t \dot{\Delta}_j\widetilde{n}^++ u_1^+\cdot \nabla \dot{\Delta}_j\widetilde{n}^+  +(1+n_1^+)\dive 
\dot{\Delta}_j\widetilde{u}^+=\dot{\Delta}_j\widetilde{S}_1+\widetilde{T}^1_j,\\
		\partial_t \dot{\Delta}_j\widetilde{n}^- +u_1^-\cdot \nabla \dot{\Delta}_j\widetilde{n}^- +(1+n_1^-)\dive 
\dot{\Delta}_j\widetilde{u}^-=\dot{\Delta}_j\widetilde{S}_2+\widetilde{T}_j^2,\\
		\partial_t \dot{\Delta}_j\widetilde{u}^+ + u_1^+\cdot\nabla \dot{\Delta}_j\widetilde{u}^++\tfrac{M_1^1}{\rho_1^+}\nabla \dot{\Delta}_j\widetilde{n}^++\tfrac{M_2^1}{\rho_1^+}\nabla 
\dot{\Delta}_j\widetilde{n}^- \\
		\quad\quad\quad\quad\quad\quad - \nu_1^+\Delta \dot{\Delta}_j\widetilde{u}^+-\nu_2^+\nabla\dive \dot{\Delta}_j\widetilde{u}^+=\dot{\Delta}_j\widetilde{S}_3+\widetilde{T}^3_j,\\
		\partial_t \dot{\Delta}_j\widetilde{u}^- + u_1^-\cdot\nabla \dot{\Delta}_j\widetilde{u}^-+\tfrac{M_1^1}{\rho_1^-}\nabla \dot{\Delta}_j\widetilde{n}^++\tfrac{M_2^1}{\rho_1^-}\nabla 
\dot{\Delta}_j\widetilde{n}^-\\
		\quad\quad\quad\quad\quad\quad- \nu_1^-\Delta \dot{\Delta}_j\widetilde{u}^--\nu_2^-\nabla\dot{\Delta}_j\dive \widetilde{u}^-=\dot{\Delta}_j\widetilde{S}_4+\widetilde{T}_j^4,
	\end{array}
	\right.
\end{equation}
with  the commutator terms 
\begin{equation}
	\left\{
	\begin{aligned}
		\widetilde{T}^1_j&= [u_1^+,\dot{\Delta}_j]\nabla \widetilde{n}^++[n_1^+,\dot{\Delta}_j]\dive \widetilde{u}^+ ,\\
		\widetilde{T}^2_j&= [u_1^-,\dot{\Delta}_j]\nabla \widetilde{n}^-+[n_1^-,\dot{\Delta}_j]\dive \widetilde{u}^- ,\\
		\widetilde{T}^3_j&= [u_1^+,\dot{\Delta}_j]\nabla \widetilde{u}^+ +\Big[\tfrac{M_1^1}{\rho_1^+},\dot{\Delta}_j\Big]\nabla \widetilde{n}^+ 
+\Big[\tfrac{M_2^1}{\rho_1^+},\dot{\Delta}_j\Big]\nabla \widetilde{n}^-,\\
		\widetilde{T}^4_j &= [u_1^-,\dot{\Delta}_j]\nabla \widetilde{u}^- +\Big[\tfrac{M_1^1}{\rho_1^-},\dot{\Delta}_j\Big]\nabla \widetilde{n}^+ 
+\Big[\tfrac{M_2^1}{\rho_1^-},\dot{\Delta}_j\Big]\nabla \widetilde{n}^-.
	\end{aligned}
	\right.
\end{equation}
Through applying the weighted energy method as in \eqref{inequ-foru+u-n+n-}, we obtain
\begin{equation}\label{inequpm-npm}
	\begin{aligned}
		&\frac{d}{dt}\int_{\mathbb{R}^d} \frac{1}{2}\left( |\dot{\Delta}_j 
\widetilde{u}^{+}|^2+|\dot{\Delta}_j 
\widetilde{u}^{-}|^2+\tfrac{M_1^1}{(1+n^+_1)\rho_1^+}|\dot{\Delta}_j \widetilde{n}^+|^2+\tfrac{M_2^1}{(1+n^-_1)\rho_1^-}|\dot{\Delta}_j \widetilde{n}^-|^2\right)dx\\
&\quad\quad+\nu_1^{+}\int_{\mathbb{R}^d}|\nabla \dot{\Delta}_j\widetilde{u}^{+}|^2 dx +\nu_2^{+}\int_{\mathbb{R}^d}|\dive  \dot{\Delta}_j\widetilde{u}^+|^2 dx\nonumber\\
        &\quad\quad+\nu_1^{-}\int_{\mathbb{R}^d}|\nabla \dot{\Delta}_j\widetilde{u}^{-}|^2 dx +\nu_2^{-}\int_{\mathbb{R}^d}|\dive  \dot{\Delta}_j\widetilde{u}^-|^2 dx\nonumber\\
		&\quad\lesssim  \left\|\partial_t\left(\tfrac{M_1^1}{(1+n_1^+)\rho_1^+},\tfrac{M_2^1}{(1+n_1^+)\rho_1^+}\right),\dive \left( 
\tfrac{M_1^1u_1^+}{(1+n_1^+)\rho_1^+},\tfrac{M_2^1u_1^-}{(1+n_1^-)\rho_1^-}\right) \right\|_{L^{\infty}} \|\dot{\Delta}_j (\widetilde{n}^{+},\widetilde{n}^{-},\widetilde{u}^{+},\widetilde{u}^{-})\|_{L^2}^2\\
		&\quad\quad+\left\|\nabla\left(\tfrac{M_1^1}{\rho_1^+},\tfrac{M_2^1}{\rho_1^-},\tfrac{M_1^1}{\rho_1^-},\tfrac{M_2^1}{\rho_1^+}\right)\right\|_{L^{\infty}}\|\dot{\Delta}_j 
(\widetilde{n}^{+},\widetilde{n}^{-})\|_{L^2}\|\dot{\Delta}_j (\widetilde{u}^{+},\widetilde{u}^{-})\|_{L^2}\\
		&\quad\quad+\|\dive \dot{\Delta}_j (\widetilde{u}^{+},\widetilde{u}^{-})\|_{L^2}\|\dot{\Delta}_j 
(\widetilde{n}^{+},\widetilde{n}^{-})\|_{L^2}\\
		&\quad\quad+\left(\|(\widetilde{T}_j^1,\widetilde{T}_j^2,\widetilde{T}_j^3,\widetilde{T}_j^4)\|_{ L^{2}}+ \|\dot{\Delta}_j(\widetilde{S}_1, \widetilde{S}_2, 
\widetilde{S}_3,\widetilde{S}_4)\|_{L^2}\right)\|\dot{\Delta}_j (\widetilde{n}^{+},\widetilde{n}^{-},\widetilde{u}^{+},\widetilde{u}^{-})\|_{L^2},
	\end{aligned}
\end{equation}
Integrating \eqref{inequpm-npm} over $[0,t]$, taking the square root of both sides and summing the resulting estimate over $j\in Z$ with the weight $2^{j\frac{d}{2}}$, we get
\begin{equation*}
	\begin{aligned}
		&\|(\widetilde{n}^+,\widetilde{n}^-,\widetilde{u}^{+},\widetilde{u}^{-})\|_{\widetilde{L}^{\infty}_t(\dot{B}_{2,1}^{\frac{d}{2}})}+  \|(\widetilde{u}^{+},\widetilde{u}^{-})\|_{\widetilde{L}^{2}_t(\dot{B}_{2,1}^{\frac{d}{2}+1})} \\
		\lesssim&\left(\int_0^t \left(\|(\partial_t n_1^{+},\partial_t n_1^{-},\nabla n_1^+,\nabla n_1^-,\nabla u_1^{+},\nabla u_1^{-}) \|_{L^{\infty}}\| (\widetilde{n}^{+},\widetilde{n}^{-},\widetilde{u}^{+},\widetilde{u}^{-})\|_{ \dot{B}_{2,1}^{\frac{d}{2} } 
} \right.\right.+\|\dive  (\widetilde{u}^+,\widetilde{u}^{-})\|_{\dot{B}_{2,1}^{\frac{d}{2}}}\\
		&\quad+\sum_{j\in Z}2^{j\frac{d}{2}}\|(T_j^1,T_j^2,T_j^3,T_j^4)\|_{ L^{2}}
		+ \left.\left.\|(\widetilde{S}_1, \widetilde{S}_2, \widetilde{S}_3,\widetilde{S}_4)\|_{\dot{B}_{2,1}^{\frac{d}{2}}}\right)ds\right)^{\frac{1}{2}}\| (\widetilde{n}^{+},\widetilde{n}^{-},\widetilde{u}^{+},\widetilde{u}^{-})\|_{\widetilde{L}^{\infty}(\dot{B}_{2,1}^{\frac{d}{2} })}^{\frac{1}{2}},
	\end{aligned}
\end{equation*}
which, together with Young's inequality, implies
\begin{align}
		&\|(\widetilde{n}^+,\widetilde{n}^-,\widetilde{u}^{+},\widetilde{u}^{-})\|_{\widetilde{L}^{\infty}_t(\dot{B}_{2,1}^{\frac{d}{2}})}+  \|(\widetilde{u}^{+},\widetilde{u}^{-})\|_{\widetilde{L}^{2}_t(\dot{B}_{2,1}^{\frac{d}{2}+1})}  
\nonumber\\
		& \lesssim\int_0^t \left(\|(\partial_t n_1^{+},\partial_t n_1^{-},\nabla n_1^+,\nabla n_1^-,\nabla u_1^{+},\nabla u_1^{-}) \|_{L^{\infty}}\| (\widetilde{n}^+,\widetilde{n}^-,\widetilde{u}^{+},\widetilde{u}^{-})\|_{ \dot{B}_{2,1}^{\frac{d}{2} } } 
+\|(\widetilde{u}^{+},\widetilde{u}^{-})\|_{\dot{B}_{2,1}^{\frac{d}{2}+1}}\right.\nonumber \\
		&+ \sum_{j\in Z}2^{j\frac{d}{2}}\|(T_j^1,T_j^2,T_j^3,T_j^4)\|_{ L^{2}} \left.+\|(\widetilde{S}_1, \widetilde{S}_2, 
\widetilde{S}_3,\widetilde{S}_4)\|_{\dot{B}_{2,1}^{\frac{d}{2}}}\right)ds.\label{unrttt}
	\end{align}
Using the classical commutator estimates in Lemma \ref{lemcommutator}, we have
\begin{align}
	\sum_{j\in Z}2^{j\frac{d}{2}}\|(T_j^1,T_j^2,T_j^3,T_j^4)\|_{ L^{2}} \lesssim \|(n_1^+,u_1^{\pm})\|_{ \dot{B}_{2,1}^{\frac{d}{2}+1} }\|(\widetilde{n}^+,\widetilde{n}^-,\widetilde{u}^{+},\widetilde{u}^{-})\|_{ 
\dot{B}_{2,1}^{\frac{d}{2} } }.\label{unrttt1}
\end{align}
In addition, according to standard product laws and composite estimates in Lemmas \ref{Lemma5-5} and \ref{Lemma5-7}, the nonlinear terms can be handled as
\begin{align}
	\|(\widetilde{S}_1, \widetilde{S}_2, \widetilde{S}_3,\widetilde{S}_4)\|_{\dot{B}_{2,1}^{\frac{d}{2}}}\lesssim \|(n_1^+,n_1^-,u_1^+,u_1^-,,n_2^+,n_2^-,u_2^+,u_2^-)\|_{\dot{B}_{2,1}^{\frac{d}{2}+1}}\|(\widetilde{n}^+,\widetilde{n}^-,\widetilde{u}^{+},\widetilde{u}^{-})\|_{ \dot{B}_{2,1}^{\frac{d}{2} } }.\label{unrttt2}
\end{align}
Substituting \eqref{unrttt1} and \eqref{unrttt2} into \eqref{unrttt} and then employing Gr\"onwall's
inequality, we have
\begin{equation}
	\begin{aligned}
		&\|(\widetilde{n}^+,\widetilde{n}^-,\widetilde{u}^{+},\widetilde{u}^{-})\|_{\widetilde{L}^{\infty}_t(\dot{B}_{2,1}^{\frac{d}{2}})}+  
\|(\widetilde{u}^{+},\widetilde{u}^{-})\|_{\widetilde{L}^{2}_t(\dot{B}_{2,1}^{\frac{d}{2}+1})}\lesssim \sqrt{t}e^{Ct}\|(\widetilde{u}^{+},\widetilde{u}^{-})\|_{\widetilde{L}^{2}_t(\dot{B}_{2,1}^{\frac{d}{2}+1})}.
	\end{aligned}
\end{equation}
Taking $t\in[0,T_1]$ with a suitably small time $T_1\in(0,T]$, we deduce that
\begin{equation}
	\begin{aligned}
		&\|(\widetilde{n}^+,\widetilde{n}^-,\widetilde{u}^{+},\widetilde{u}^{-})\|_{\widetilde{L}^{\infty}_t(\dot{B}_{2,1}^{\frac{d}{2}})}+  \|(\widetilde{u}^{+},\widetilde{u}^{-})\|_{\widetilde{L}^{2}_t(\dot{B}_{2,1}^{\frac{d}{2}+1})}=0.
	\end{aligned}
\end{equation}
The above argument can be repeated on $[T_1,2T_1]$, $[2T_1,3T_1]$,$\cdots$, until the whole interval $[0,T]$ is exhausted. Therefore, the proof of uniqueness has been finished. \hfill 
$\Box$

\section{Time-decay rates (I): Proof of Theorem \ref{Thm1.3}}\label{sectiondecay1}

In this section, we show the optimal time decay rates of a strong solution to the Cauchy problem in the case that $\|(\theta_0,u_0^{+},u_0^{-})^\ell\|_{\dot{B}_{2,\infty}^{\sigma_0}}$ is only bounded.

\subsection{Weighted Lyapunov estimates}

Inspired by the works \cite{Guo-2012,Xin-2021} and the recent adaptation \cite{LiShou2023}, we have the time-weighted Lyapunov estimates of the solution $(R^+,\theta,u^{\pm})$ to the Cauchy 
problem \eqref{equation_R_P_u_pm}.

\begin{lemma}\label{lem-for-low-frequency}
	Let $d\geq 3$ and $\sigma_0\in[-\frac{d}{2},\frac{d}{2}-2)$, and let $(n^+, \theta,u^+,u^-)$ be the global classical solution to the Cauchy problem \eqref{equation_R_P_u_pm} given by
Theorem \ref{Thm1.1}. Then under the assumptions \eqref{a1} and \eqref{a2},  for any $A>\frac{1}{2}(\frac{d}{2}-\sigma_0)$, it holds that
	\begin{align}\label{weighteddecay}
		\|\tau^{A}\theta\|&_{\widetilde{L}^{\infty}_t(\dot{B}_{2,1}^{\frac{d}{2}-2}\cap 
\dot{B}_{2,1}^{\frac{d}{2}})}+\|\tau^{A}(u^{+},u^{-})\|_{\widetilde{L}^{\infty}_t(\dot{B}_{2,1}^{\frac{d}{2}-2}\cap \dot{B}_{2,1}^{\frac{d}{2}-1})}\nonumber\\
		&\quad+\|\tau^{A}\theta \|_{L^{1}_t(\dot{B}_{2,1}^{\frac{d}{2}})}+\|\tau^{A}(u^{+},u^{-})\|_{L^{1}_t (\dot{B}_{2,1}^{\frac{d}{2}}\cap\dot{B}^{\frac{d}{2}+1}_{2,1})}\nonumber \\
		&\lesssim t^{A-\frac{1}{2}(\frac{d}{2}-2-\sigma_0)}(\mathcal{E}_0+\|(\theta,u^{+},u^{-})^{\ell}\|_{L^{\infty}_{t}(\dot{B}^{\sigma_0}_{2,\infty})}),\quad\quad t>0.
	\end{align}
\end{lemma}

\begin{proof}
	For any $j\leq 0$, after multiplying the Lyapunov type inequality \eqref{Lyapunovlow} by $t^{A}$, we show
	\begin{align*}
		\frac{d}{dt} \Big(t^{A}\mathcal{E}_{1,j}(t)\Big)+ 2^{2j}t^{A}\mathcal{E}_{1,j}(t)\lesssim t^{A-1}\mathcal{E}_{1,j}(t)+ t^{A}\|\dot{\Delta}_j( F_1, F_2^+, 
F_2^-)\|_{L^2}\sqrt{\mathcal{E}_{1,j}(t)},\quad j\leq 0.
	\end{align*}
	This gives rise to
	\begin{align}
		t^{A} \|\dot{\Delta}_j(\theta, u^{+},u^{-})\|_{L^2}&+2^{2j}\int_0^t\tau^{A}\|\dot{\Delta}_j(\theta, u^{+},u^{-})\|_{L^2}d\tau\nonumber\\
		&\lesssim \int_0^t\tau^{A-1}\|\dot{\Delta}_j(\theta, u^{+},u^{-})\|_{L^2}d\tau+\int_0^t\tau^{A} \|\dot{\Delta}_j(F_1, F_2
        ^+, F_2^-)\|_{L^2}d\tau.
	\end{align}
	Then, multiplying by $2^{j(\frac{d}{2}-2)}$, taking the supremum on $[0,t]$, and then summing over $j\leq0$, we get
	\begin{align}\label{time-weigh-for-low}
		\|\tau^{A}(\theta, u^{+},u^{-})\|^{\ell}_{\widetilde{L}^{\infty}_t(\dot{B}_{2,1}^{\frac{d}{2}-2})}&+\|\tau^{A}(\theta, u^{+},u^{-})\|^{\ell}_{L^{1}_t 
(\dot{B}_{2,1}^{\frac{d}{2}})}\nonumber\\
		&\lesssim \int_0^t\tau^{A-1}\|(\theta, u^{+},u^{-})\|^{\ell}_{\dot{B}_{2,1}^{\frac{d}{2}-2}}d\tau+\|\tau^{A}( F_1, F_2^+, F_2^-)\|^{\ell}_{L^{1}_t(\dot{B}_{2,1}^{\frac{d}{2}-2})}.
	\end{align}
	To control the first term on the right-hand side of \eqref{time-weigh-for-low}, we employ the interpolation argument between weighted dissipation norm and the 
$\dot{B}^{\sigma_0}_{2,\infty}$ norm. Let $\eta_0\in(0,1)$ be given by
	\begin{align}
		\frac{d}{2}-2=(1-\eta_0)\sigma_0+\eta_0\frac{d}{2},
	\end{align}
	and $\var_0>0$ be a constant to be chosen later. We deduce from the real interpolation inequality \eqref{inter}  that
	\begin{align*}
		&\quad\int_0^t\tau^{A-1}\|(\theta, u^{+},u^{-})^{\ell}\|_{\dot{B}_{2,1}^{\frac{d}{2}-2}}d\tau\nonumber\\
        &\lesssim 
\int_0^t\tau^{A-1}\|(\theta,u^{+},u^{-})^{\ell}\|^{1-\eta_0}_{\dot{B}_{2,\infty}^{\sigma_0}}\|(\theta, u^{+},u^{-})^{\ell}\|^{\eta_0}_{\dot{B}_{2,1}^{\frac{d}{2}}}d\tau\nonumber\\
		&\lesssim \left(\int_0^t\tau^{A-\frac{1}{1-\eta_0}}d\tau\right)^{1-\eta_0}\|(\theta,u^{+},u^{-})\|^{1-\eta_0}_{L^{\infty}_t(\dot{B}_{2,\infty}^{\sigma_0})}\|\tau^{A}(\theta, 
u^{\pm})\|^{\eta_0}_{L^1_t(\dot{B}_{2,1}^{\frac{d}{2}})}\nonumber
		\\
		&\lesssim \left(t^{A-\frac{1}{2}(\frac{d}{2}-2-\sigma_0)}\|(\theta,u^{+},u^{-})^{\ell}\|_{\widetilde{L}^{\infty}_t(\dot{B}_{2,\infty}^{\sigma_0})}\right)^{1-\eta_0}\|\tau^{A}(\theta, 
u^{\pm})^{\ell}\|^{\eta_0}_{L^1_t(\dot{B}_{2,1}^{\frac{d}{2}+1})}\nonumber\\
		&\lesssim \var_0^{-1}t^{A-\frac{1}{2}(\frac{d}{2}-2-\sigma_0)}\|(\theta,u^{+},u^{-})^{\ell}\|_{L^{\infty}_t(\dot{B}_{2,\infty}^{\sigma_0})}+ \var_0 \|\tau^{A}(\theta,u^{+},u^{-})^{\ell}\|_{L^1_t(\dot{B}_{2,1}^{\frac{d}{2}})}.
	\end{align*}
	On the other hand, taking advantage of the dissipative structure of $(\theta, u^{+},u^{-})$ for high frequencies and the cut-off properties of frequencies, it is easy to verify that
	\begin{align*}
		\int_0^t\tau^{A-1}\|\theta^h \|_{\dot{B}_{2,1}^{\frac{d}{2}-2}}^{\ell}d\tau&= 
\int_0^t\tau^{A-1}\|\theta^h\|^{1-\eta_0}_{\dot{B}_{2,1}^{\frac{d}{2}-2}}\|\theta^h\|^{\eta_0}_{\dot{B}_{2,1}^{\frac{d}{2}-2}}d\tau\nonumber\\
		&\lesssim 
\left(\int_0^t\tau^{A-\frac{1}{1-\eta_0}}d\tau\right)^{1-\eta_0}\|\theta^h\|^{1-\eta_0}_{\widetilde{L}^{\infty}_t(\dot{B}_{2,1}^{\frac{d}{2}-2})}\|\tau^{A}\theta^h\|^{\eta_0}_{L^1_t(\dot{B}_{2,1}^{\frac{d}{2}-2})}\nonumber
		\\
		&\lesssim \var_0^{-1}t^{A-\frac{1}{2}(\frac{d}{2}-2-\sigma_0)}\|\theta\|_{\widetilde{L}^{\infty}_t(\dot{B}_{2,1}^{\frac{d}{2}})}^h+ 
\var_0\|\tau^{A}\theta^h\|_{L^1_t(\dot{B}_{2,1}^{\frac{d}{2}})}^h,
	\end{align*}
	and
	\begin{align*}
		\int_0^t\tau^{A-1}\|(u^{\pm})^h \|_{\dot{B}_{2,1}^{\frac{d}{2}-2}}^{\ell}d\tau&\lesssim 
\int_0^t\tau^{A-1}\|(u^{\pm})^h\|^{1-\eta_0}_{\dot{B}_{2,1}^{\frac{d}{2}-2}}\|(u^{\pm})^h\|^{\eta_0}_{\dot{B}_{2,1}^{\frac{d}{2}-2}}d\tau\nonumber\\
		&\lesssim \var_0^{-1}t^{A-\frac{1}{2}(\frac{d}{2}-2-\sigma_0)}\|u^{\pm}\|_{\widetilde{L}^{\infty}_t(\dot{B}_{2,1}^{\frac{d}{2}-1})}^h+ 
\var_0\|\tau^{A}u^{\pm}\|_{L^1_t(\dot{B}_{2,1}^{\frac{d}{2}+1})}^h.
	\end{align*}
	Hence, we have for any constant $\var_0>0$ that
	\begin{align}\label{lowinter}
		&\quad\int_0^t\tau^{A-1}\|(\theta, u^{+},u^{-})\|_{\dot{B}_{2,1}^{\frac{d}{2}-2}}^{\ell}d\tau \nonumber\\
        &\lesssim  \int_0^t\tau^{A-1}\|(\theta, u^{+},u^{-})^{\ell}\|_{\dot{B}_{2,1}^{\frac{d}{2}-2}}d\tau 
+\int_0^t\tau^{A-1}\|(\theta, u^{+},u^{-})^{h}\|_{\dot{B}_{2,1}^{\frac{d}{2}-2}}d\tau \nonumber \\
		&\lesssim \var_0^{-1}\Big(\|(\theta,u^{+},u^{-})^{\ell}\|_{\widetilde{L}^{\infty}_t(\dot{B}_{2,\infty}^{\sigma_0})}+\mathcal{X}(t) \Big)t^{A-\frac{1}{2}(\frac{d}{2}-2-\sigma_0)}\\
		&\quad+ \var_0(\|\tau^{A}\theta\|_{L^{1}_t(\dot{B}_{2,1}^{\frac{d}{2}})}+\|\tau^{A}(u^{+},u^{-})\|_{L^{1}_t 
(\dot{B}_{2,1}^{\frac{d}{2}}\cap\dot{B}^{\frac{d}{2}+1}_{2,1})}).\nonumber
	\end{align}
	Concerning the nonlinear terms on the right-hand side of \eqref{time-weigh-for-low}, we claim that
	\begin{align}
		\|\tau^{A}(F_1,F_2^+,F_2^-)\|_{L^{1}_t(\dot{B}_{2,1}^{\frac{d}{2}-1})}\lesssim 
\mathcal{X}(t)(\|\tau^{A}\theta\|_{L^{1}_t(\dot{B}_{2,1}^{\frac{d}{2}})}+\|\tau^{A}(u^{+},u^{-})\|_{L^{1}_t 
(\dot{B}_{2,1}^{\frac{d}{2}}\cap\dot{B}^{\frac{d}{2}+1}_{2,1})}).\label{F1F2F3t}
	\end{align}
    Different from Lemma \ref{lemma1}, we are able to analyze some nonlinear terms, e.g., 
$g_1^\pm(n^+,\theta)\dive  u^{\pm}$ and $\rho^- u^+\cdot\nabla n^+$ in $L^1_t(\dot{B}^{d/2-2}_{2,1})$ as the bounds of $n^+, \theta$ and $u^+$ have been established in Theorem 
\ref{Thm1.1}. Arguing similarly as for proving \eqref{F1}, one may analyze $F_1$ as
	\begin{align*}
		\|\tau^{A}F_1\|_{L^{1}_t(\dot{B}_{2,1}^{\frac{d}{2}-2})}&\lesssim \|\tau^{A}g_3^\pm(n^{+},\theta)\dive u^{\pm}\|_{L^1_t (\dot{B}_{2,1}^{\frac{d}{2}-2})}\\
		&\quad+\|\tau^{A}\rho^- u^+\cdot\nabla n^+\|_{L^1_t (\dot{B}_{2,1}^{\frac{d}{2}-2})}+\|\tau^{A}\rho^+ u^-\cdot\nabla n^-\|_{L^1_t (\dot{B}_{2,1}^{\frac{d}{2}-2})}\\
		&\quad+\|\tau^{A}(\mathcal{C}(n^++1,\theta+\bar{P})-\mathcal{C}( 1, \bar{P}))\rho^- u^+\cdot\nabla n^+\|_{L^1_t (\dot{B}_{2,1}^{\frac{d}{2}-2})}\\
		&\quad+\|\tau^{A}(\mathcal{C}(n^++1,\theta+\bar{P})-\mathcal{C}( 1, \bar{P}))\rho^+ u^-\cdot\nabla n^-\|_{L^1_t (\dot{B}_{2,1}^{\frac{d}{2}-2})}\\
		&\lesssim \|(n^{+},\theta)\|_{\widetilde{L}^{\infty}_t(\dot{B}_{2,1}^{\frac{d}{2}-1})}\| \tau^{A}{\rm{\dive }\,} (u^{+},u^{-})\|_{L^1_t(\dot{B}_{2,1}^{\frac{d}{2}-1 })}\\
		&\quad+(1+\|\theta\|_{\widetilde{L}^{\infty}_t(\dot{B}^{\frac{d}{2}}_{2,1})})\|(n^+,\theta)\|_{\widetilde{L}^{\infty}_t(\dot{B}^{\frac{d}{2}-1}_{2,1}\cap 
\dot{B}^{\frac{d}{2}}_{2,1})}\|\tau^{A} (u^{+},u^{-})\|_{L^1_t(\dot{B}^{\frac{d}{2}}_{2,1})} \\
		&\lesssim \mathcal{X}(t) \| \tau^{A}(u^{+},u^{-})\|_{L_t^{1}(\dot{B}^{\frac{d}{2}}_{2,1})}.
	\end{align*}
	The other terms regarding $F_{2}$ and $F_3$ can be handled similarly, so we omit the details. Then, putting \eqref{lowinter} and \eqref{F1F2F3t} into \eqref{time-weigh-for-low} yields
	\begin{align}\label{time-weigh-for-low1}
		\|\tau^{A}(\theta, u^{+},u^{-})\|&^{\ell}_{\widetilde{L}^{\infty}_t(\dot{B}_{2,1}^{\frac{d}{2}-2})}+\|\tau^{A}(\theta, u^{+},u^{-})\|^{\ell}_{L^{1}_t 
(\dot{B}_{2,1}^{\frac{d}{2}})}\nonumber\\
		&\lesssim   \var_0^{-1}\Big(\|(\theta,u^{+},u^{-})^{\ell}\|_{\widetilde{L}^{\infty}_t(\dot{B}_{2,\infty}^{\sigma_0})}+\mathcal{X}(t) \Big)t^{A-\frac{1}{2}(\frac{d}{2}-2-\sigma_0)}\\
		&\quad+(\mathcal{X}(t)+\var_0)\Big(\|\tau^{A}\theta\|_{L^{1}_t(\dot{B}_{2,1}^{\frac{d}{2}})}+\|\tau^{A}(u^{+},u^{-})\|_{L^{1}_t 
(\dot{B}_{2,1}^{\frac{d}{2}}\cap\dot{B}^{\frac{d}{2}+1}_{2,1})}\Big).\nonumber
	\end{align}

	We then turn to carry out the weighted estimates in high frequencies. For any $j\geq -1$, we show after multiplying the Lyapunov type inequality \eqref{Lyapunovhigh} by $t^{A}$ that
	\begin{align*}
		\frac{d}{dt} \Big(t^{A}\mathcal{E}_{2,j}(t)\Big)+  t^{A}\mathcal{E}_{2,j}(t)\lesssim t^{A-1}\mathcal{E}_{2,j}(t)+t^{A}\|\dot{\Delta}_j(\nabla F_1, F_2^+, 
F_2^-)\|_{L^2}\sqrt{\mathcal{E}_{2,j}(t)},\quad j\geq -1,
	\end{align*}
	which implies
	\begin{align*}
		\|\tau^{A}\dot{\Delta}_j&(\nabla\theta, u^{+},u^{-})\|_{L^{\infty}_{t}(L^2)}+  \|\tau^{A}\dot{\Delta}_j(\nabla\theta, u^{+},u^{-})\|_{L^{\infty}_{t}(L^2)}\\
		&\lesssim \int_0^t \tau^{A-1}\|\dot{\Delta}_j(\nabla\theta, u^{+},u^{-})\|_{L^2}d\tau+\|\tau^{A}\dot{\Delta}_j(\nabla F_1, F_2^+, F_2^-)\|_{L^1_t(L^2)},\quad j\geq -1.
	\end{align*}
	Thence, summing the above inequality over $j\geq -1$ with the weight $2^{j(\frac{d}{2}-1)}$, we have
	\begin{equation}\label{time-weigh-for-high}
		\begin{aligned}
			&\|\tau^{A}(\nabla\theta, u^{+},u^{-})\|_{\widetilde{L}^{\infty}_t(\dot{B}_{2,1}^{\frac{d}{2}-1})}^h+\|\tau^{A}(\nabla\theta, 
u^{\pm})\|_{L^{1}_t(\dot{B}_{2,1}^{\frac{d}{2}-1})}^h\\
			&\quad \lesssim \int_0^t\tau^{A-1}\|(\nabla\theta, u^{+},u^{-})\|^h_{\dot{B}_{2,1}^{\frac{d}{2}-1}}d\tau+\|\tau^{A}(\nabla F_1, F_2^+, 
F_2^-)\|^h_{L^{1}_t(\dot{B}_{2,1}^{\frac{d}{2}-1})}.
		\end{aligned}
	\end{equation}
	By similar arguments as used in the high-frequency part of \eqref{lowinter}, we conclude that, for any constant $\var_0>0$,
	\begin{equation}\label{ineq_for-u_pm_h}
		\begin{aligned}
			&\quad \int_0^t\tau^{A-1}\|(\nabla\theta, u^{+},u^{-})\|^h_{\dot{B}_{2,1}^{\frac{d}{2}-1}}d\tau   \lesssim \var_0 \mathcal{X}(t)t^{A-\frac{1}{2}(\frac{d}{2}-2-\sigma_0)}+ \var_0 
\|\tau^{A}(u^{+},u^{-})\|_{L^{1}_t (\dot{B}^{\frac{d}{2}+1}_{2,1})}.
		\end{aligned}
	\end{equation}
	Also, the nonlinearities on the right-hand side of \eqref{time-weigh-for-high} can be estimated by
    \begin{align}\label{inequalityh_for_F_1}	
\|\tau^{A}F_1\|^h_{L_t^{1}(\dot{B}^{\frac{d}{2}}_{2,1}))}\lesssim \mathcal{X}(t)(\|\tau^{A}\theta\|_{L^{1}_t(\dot{B}_{2,1}^{\frac{d}{2}})}+\|\tau^{A}u^{\pm}\|_{L^{1}_t 
(\dot{B}^{\frac{d}{2}+1}_{2,1})}),
	\end{align}
	and similarly,
	\begin{align}\label{inequalityh_for_F_3}
		\|\tau^{A}F_2^\pm\|^h_{L^1_t( \dot{B}_{2,1}^{\frac{d}{2}-1})}\lesssim 
\mathcal{X}(t)(\|\tau^{A}\theta\|_{L^{1}_t(\dot{B}_{2,1}^{\frac{d}{2}})}+\|\tau^{A}u^{\pm}\|_{L^{1}_t (\dot{B}^{\frac{d}{2}+1}_{2,1})}).
	\end{align}
	Therefore, it follows that
	\begin{equation}\label{tAhigh1}
		\begin{aligned}
			&\|\tau^{A}(\nabla\theta, u^{+},u^{-})\|_{\widetilde{L}^{\infty}_t(\dot{B}_{2,1}^{\frac{d}{2}-1})}^h+\|\tau^{A}(\nabla\theta, 
u^{\pm})\|_{L^{1}_t(\dot{B}_{2,1}^{\frac{d}{2}-1})}^h\\
			& \quad \lesssim \var_0^{-1} \mathcal{X}(t)t^{A-\frac{1}{2}(\frac{d}{2}-2-\sigma_0)}+ (\var_0+\mathcal{X}(t)) 
(\|\tau^{A}\theta\|_{L^{1}_t(\dot{B}_{2,1}^{\frac{d}{2}})}+\|\tau^{A}(u^{+},u^{-})\|_{L^{1}_t (\dot{B}^{\frac{d}{2}+1}_{2,1})}).
		\end{aligned}
	\end{equation}

	Furthermore, we are able to improve the regularity estimate of $u^\pm$. Multiplying $\eqref{equation_R_P_u_pm}_3$ and $\eqref{equation_R_P_u_pm}_4$ by $t^{A}$, we get
	\begin{align}
		\partial_t (t^{A}u^\pm) - \nu_1^{+}\Delta (t^{A} u^\pm)-\nu_2^+\nabla\dive  (t^{A}u^\pm)=\theta t^{A-1}u^\pm-\beta_2^\pm \nabla (t^{A}\theta)+t^{A}F_2^{\pm}\label{weightuaa}
	\end{align}
	with the initial data $t^{A}u^{\pm}|_{t=0}=0.$
	By virtue of Lemma \ref{lame-system-estimate}, we have
	\begin{align}
		\|\tau^{A}u^{\pm}\|^h_{L^{1}_t (\dot{B}_{2,1}^{\frac{d}{2}+1})}\lesssim \int_0^t\tau^{A-1}\| u^{\pm}\|^h_{ \dot{B}_{2,1}^{\frac{d}{2}-1}}d\tau 
+\|\tau^{A}\theta\|^h_{L^{1}_t(\dot{B}_{2,1}^{\frac{d}{2}-1})}+\|\tau^{A}F_2^\pm\|^h_{L^{1}_t(\dot{B}_{2,1}^{\frac{d}{2}-1})},
	\end{align}
	which, together with \eqref{ineq_for-u_pm_h},  \eqref{inequalityh_for_F_3} and \eqref{tAhigh1}, gives rise to
	\begin{align}
		\|\tau^{A}u^{\pm}&\|^h_{L^{1}_t (\dot{B}_{2,1}^{\frac{d}{2}+1})}\lesssim \var_0^{-1} \mathcal{X}(t)t^{A-\frac{1}{2}(\frac{d}{2}-2-\sigma_0)}+ \var_0 (\|\tau^A\theta\|_{L^1_t(\dot{B}^{\frac{d}{2}}_{2,1})}+\|\tau^A u\|_{L^1_t(\dot{B}^{\frac{d}{2}+1}_{2,1})}).\label{timeuhigher}
	\end{align}
	
Finally, combining \eqref{time-weigh-for-low1}, \eqref{tAhigh1} and \eqref{timeuhigher}, taking a suitable small constant $\var_0$ and recalling $\mathcal{X}(t)\lesssim \mathcal{X}_0<<1$, we conclude \eqref{weighteddecay} and complete the proof of Lemma \ref{lem-for-low-frequency}.
\end{proof}

\subsection{Uniform evolution of the low-frequency regularity}

In order to bound the term in \eqref{weighteddecay} and derive decay estimates, we need to establish the uniform $\dot{B}^{\sigma_0}_{2,\infty}$ regularity evolution of the solution for
low frequencies. This is stated in the following lemma.

\begin{lemma}\label{lemma_for_E_L_sigma_0}
	Let $(n^+, \theta,u^+,u^-)$ be the global solution to the Cauchy problem \eqref{equation_R_P_u_pm} by Theorem \ref{Thm1.1}.  Then, under the assumptions of Theorem \ref{Thm1.3}, the following inequality 
holds:
	\begin{align}\label{sigma_0-functional}
		\|(\theta, u^{+},u^{-})\|_{L^{\infty}_t( \dot{B}_{2,\infty}^{\sigma_0})}+\|(\theta, u^{+},u^{-})\|_{\widetilde{L}^{1}_t (\dot{B}_{2,\infty}^{\sigma_0+2})}\leq C\delta_0,
	\end{align}
	for $ \delta_0:=\|(\theta_0, u_0^{+},u_0^-)^{\ell}\|_{\dot{B}^{\sigma_0}_{2,\infty}}+ \mathcal{X}_0$ .
\end{lemma}

\begin{proof}
	Since
	\begin{equation}\label{highsigma0}
		\begin{aligned}
			&\|(\theta, u^{+},u^{-})\|_{L^{\infty}_t( \dot{B}_{2,\infty}^{\sigma_0})}^h+\|(\theta, u^{+},u^{-})\|_{\widetilde{L}^{1}_t (\dot{B}_{2,\infty}^{\sigma_0+2})}^h\\
			&\quad\lesssim \|(\theta, u^{+},u^{-})\|_{L^{\infty}_t( \dot{B}_{2,1}^{\frac{d}{2}+1})}^h+\|(\theta, u^{+},u^{-})\|_{L^{1}_t (\dot{B}_{2,1}^{\frac{d}{2}})}^h\lesssim \mathcal{X}_0,
		\end{aligned}
	\end{equation}
	it suffices to bound the low-frequency norms. Recall that the Lyapunov-type inequality \eqref{low-frequency1} holds. 
Multiplying \eqref{low-frequency1} by $2^{j\sigma_0}$ and taking the supremum on both $[0,t]$ and $j\leq 0$, we get
	\begin{align}
		&\|(\theta, u^{+},u^{-})\|^{\ell}_{L^{\infty}_t( \dot{B}_{2,\infty}^{\sigma_0})}+\|(\theta, u^{+},u^{-})\|^{\ell}_{\widetilde{L}^{1}_t (\dot{B}_{2,\infty}^{\sigma_0+2})}\nonumber\\
        &\quad\lesssim 
\|(\theta_0, u_0^{+},u_0^{-})\|^{\ell}_{\dot{B}_{2,\infty}^{\sigma_0}}+\|( F_1, F_2^+, F_2^-)\|^{\ell}_{\widetilde{L}^{1}_t(\dot{B}_{2,\infty}^{\sigma_0})}.\label{sigma1111}
	\end{align}
	We first analyze $\|F_1\|^{\ell}_{\widetilde{L}^{1}_t(\dot{B}_{2,\infty}^{\sigma_0})}$. It follows from \eqref{uv3} and \eqref{F1} that
	\begin{equation*}
		\begin{aligned}
			\|g_3^\pm(n^{+},\theta)\dive u^{\pm}\|^{\ell}_{\widetilde{L}^1_t (\dot{B}_{2,\infty}^{\sigma_0})}&\lesssim 
\|g_3^\pm(n^{+},\theta)\|_{\widetilde{L}^{\infty}_t(\dot{B}_{2,1}^{\frac{d}{2}-1})} \|\dive u^{\pm}\|_{\widetilde{L}^1_t( \dot{B}_{2,\infty}^{\sigma_0+1})} \\
			&\lesssim \| u^{\pm}\|_{\widetilde{L}^1_t(\dot{B}_{2,\infty}^{\sigma_0+2})} \|(n^{+},\theta)\|_{L^{\infty}_t(\dot{B}_{2,1}^{\frac{d}{2}-1})} .
		\end{aligned}
	\end{equation*}
	Along the same line of the above estimate, we have
	\begin{equation*}
		\begin{aligned}
			&\|\rho^- u^+\cdot\nabla n^+\|^{\ell}_{\widetilde{L}^1_t (\dot{B}_{2,\infty}^{\sigma_0})}+\|\rho^+ u^-\cdot\nabla n^-\|^{\ell}_{\widetilde{L}^1_t 
(\dot{B}_{2,\infty}^{\sigma_0})}\\
			&\quad \lesssim (1+\|(\rho^+-\bar{\rho}^+,\rho^--\bar{\rho}^-)\|_{\widetilde{L}^{\infty}_{t}(\dot{B}^{\frac{d}{2}}_{2,1})})\| (u^+\cdot\nabla n^+,u^-\cdot\nabla n^-)\|_{\widetilde{L}^1_t( 
\dot{B}_{2,\infty}^{\sigma_0})}\\
			&\lesssim (1+\|\theta\|_{\widetilde{L}^{\infty}_{t}(\dot{B}^{\frac{d}{2}}_{2,1})})\|(u^+,u^-)\|_{\widetilde{L}^1_t(\dot{B}_{2,\infty}^{\sigma_0+2})}\| (n^+,\theta)\|_{\widetilde{L}^{\infty}_t 
(\dot{B}_{2,1}^{\frac{d}{2}-1})},
		\end{aligned}
	\end{equation*}
	and
	\begin{equation*}
		\begin{aligned}
			&\|(\mathcal{C}(n^++1,\theta+\bar{P})-\mathcal{C}( 1, \bar{P}))\rho^- u^+\cdot\nabla n^+\|^{\ell}_{\widetilde{L}^1_t (\dot{B}_{2,\infty}^{\sigma_0})}\\
			&\quad+\|(\mathcal{C}(n^++1,\theta+\bar{P})-\mathcal{C}( 1, \bar{P}))\rho^+ u^-\cdot\nabla n^-\|^{\ell}_{\widetilde{L}^1_t (\dot{B}_{2,\infty}^{\sigma_0})}\\
			&\lesssim \|(n^{+},\theta)\|_{\widetilde{L}^{\infty}_t(\dot{B}_{2,1}^{\frac{d}{2}})}  
(1+\|\theta\|_{\widetilde{L}^{\infty}_{t}(\dot{B}^{\frac{d}{2}}_{2,1})})\|(u^+,u^-)\|_{\widetilde{L}^1_t (\dot{B}_{2,\infty}^{\sigma_0+2})}\| 
(n^+,\theta)\|_{\widetilde{L}^{\infty}_t(\dot{B}_{2,1}^{\frac{d}{2}-1})}.
		\end{aligned}
	\end{equation*}
	Hence, by \eqref{e1} we have
	\begin{align}
		\|F_1\|^{\ell}_{\widetilde{L}^1_t (\dot{B}^{\sigma_0}_{2,\infty})}\lesssim \mathcal{X}(t) \| u^{\pm}\|_{L^1_t(\dot{B}_{2,\infty}^{\sigma_0+2})} .\label{F1sigma1}
	\end{align}
	With regard to the terms in $\|F_2\|^{\ell}_{L^1_t \dot{B}^{\sigma_0}_{2,\infty}}$, we also deduce that
	\begin{align*}
		& \|u^{\pm}\cdot\nabla u^{\pm}\|^{\ell}_{\widetilde{L}^1_t (\dot{B}_{2,\infty}^{\sigma_0})}+\|g_4^{\pm}(\theta)\nabla \theta\|^{\ell}_{\widetilde{L}^1_t (\dot{B}_{2,\infty}^{\sigma_0})}\\
		&\quad\lesssim \|u^{\pm}\|_{\widetilde{L}^{\infty}_t(\dot{B}_{2,1}^{\frac{d}{2}-1})}\| \nabla u^{\pm}\|_{\widetilde{L}^1_t 
(\dot{B}_{2,\infty}^{\sigma_0+1})}+\|g_4^{\pm}(\theta)\|_{\widetilde{L}^{\infty}_t(\dot{B}_{2,1}^{\frac{d}{2}-1})} 
\|\nabla\theta\|_{\widetilde{L}^1_t(\dot{B}_{2,\infty}^{\sigma_0+1})}\nonumber\\
		&\quad\lesssim \mathcal{X}(t)\|(\theta,u^{\pm})\|_{\widetilde{L}^1_t (\dot{B}_{2,\infty}^{\sigma_0+2})}
	\end{align*}
	and
	\begin{align*}
		&\|g_1^{\pm}(\theta)\Delta u^{\pm}\|^{\ell}_{\widetilde{L}^1_t (\dot{B}_{2,\infty}^{\sigma_0})}+ \|g_2^{\pm}(\theta)\nabla\dive  u^{\pm}\|^{\ell}_{\widetilde{L}^1_t 
(\dot{B}_{2,\infty}^{\sigma_0})}\\
		&\quad\lesssim\|g_1^{\pm}(\theta)\|_{\widetilde{L}^{\infty}_t( \dot{B}_{2,1}^{\frac{d}{2}})}\|\Delta 
u^{\pm}\|_{\widetilde{L}^1_t(\dot{B}_{2,\infty}^{\sigma_0})}+\|g_2^{\pm}(\theta)\|_{\widetilde{L}^{\infty}_t( \dot{B}_{2,1}^{\frac{d}{2}})}\|\nabla\dive  
u^{\pm}\|_{\widetilde{L}^1_t(\dot{B}_{2,\infty}^{\sigma_0})}\nonumber\\
		&\quad\lesssim\mathcal{X}(t)\| u^{\pm}\|_{\widetilde{L}^1_t(\dot{B}_{2,\infty}^{\sigma_0+2})}.
	\end{align*}
	In addition, it is easy to verify that
	\begin{align}
		&\quad  \| \frac{\nabla u^\pm \cdot \nabla \alpha^\pm}{R^\pm}\|^{\ell}_{\widetilde{L}^1_t (\dot{B}_{2,\infty}^{\sigma_0})}\nonumber\\
		&\leq\|\nabla u^\pm\cdot\nabla h^\pm(n^+,\theta)\tfrac{n^\pm}{n^\pm+1}\|^{\ell}_{\widetilde{L}^1_t (\dot{B}_{2,\infty}^{\sigma_0})}+\|\nabla u^\pm\cdot\nabla h^\pm(n^+,\theta)\|^{\ell}_{\widetilde{L}^1_t 
(\dot{B}_{2,\infty}^{\sigma_0})}\nonumber\\
		&\lesssim 
(\|(n^+,\theta)\|_{\widetilde{L}^{\infty}_t(\dot{B}_{2,1}^{\frac{d}{2}})}^2+\|(n^+,\theta)\|_{\widetilde{L}^{\infty}_t(\dot{B}_{2,1}^{\frac{d}{2}})})\| 
u^\pm\|_{\widetilde{L}^1_t(\dot{B}_{2,\infty}^{\sigma_0+2})}\nonumber\\
		&\lesssim \mathcal{X}(t)\| u^\pm\|_{\widetilde{L}^1_t(\dot{B}_{2,\infty}^{\sigma_0+2})}\nonumber.
	\end{align}
	Thus, we have
	\begin{align}
		\|F_2^\pm\|^{\ell}_{\widetilde{L}^1_t (\dot{B}_{2,\infty}^{\sigma_0})}\lesssim \mathcal{X}(t)\| 
u^\pm\|_{\widetilde{L}^1_t(\dot{B}_{2,\infty}^{\sigma_0+2})}.\label{F3sigma1}
	\end{align}
	Substituting \eqref{F1sigma1} and \eqref{F3sigma1} into \eqref{sigma1111} and using \eqref{highsigma0}, we have
	\begin{equation*}
		\begin{aligned}
			&\|(\theta, u^{+},u^{-})\|_{L^{\infty}_t( \dot{B}_{2,\infty}^{\sigma_0})}+\|(\theta, u^{+},u^{-})\|_{\widetilde{L}^{1}_t (\dot{B}_{2,\infty}^{\sigma_0+2})}\\
            &\quad \lesssim \|(\theta_0, 
u_0^{+},u_0^{-})\|_{\dot{B}_{2,\infty}^{\sigma_0}}+\mathcal{X}(t)\| (\theta,u^+,u^-)\|_{\widetilde{L}^1_t(\dot{B}_{2,\infty}^{\sigma_0+2})}.
		\end{aligned}
	\end{equation*}
	Making use of $\mathcal{X}(t)\lesssim\mathcal{X}_0\ll1$ and $\|(\theta_0,u_0^{+},u_0^{-})\|_{\dot{B}_{2,\infty}^{\sigma_0}}+\mathcal{X}_0\sim \delta_0$, we prove \eqref{sigma_0-functional}. 
The proof of Lemma \ref{lemma_for_E_L_sigma_0} is finished.
	
\end{proof}

\subsection{The gain of decay rates} Let the assumptions of Theorem \ref{Thm1.1} hold and $(n^+, \theta,u^+,u^-)$ be the global solution to the Cauchy problem \eqref{equation_R_P_u_pm} 
given by Theorem \ref{Thm1.1}. For any $A>>1$, it follows by \eqref{weighteddecay} and \eqref{sigma_0-functional} that
\begin{align}\label{priorestimate-forEtheta}
	\|\tau^{A}\theta\|&_{\widetilde{L}^{\infty}_t(\dot{B}_{2,1}^{\frac{d}{2}-2}\cap 
\dot{B}_{2,1}^{\frac{d}{2}})}+\|\tau^{A}(u^{+},u^{-})\|_{\widetilde{L}^{\infty}_t(\dot{B}_{2,1}^{\frac{d}{2}-2}\cap \dot{B}_{2,1}^{\frac{d}{2}-1})}\nonumber\\
	&\quad+\|\tau^{A}\theta\|_{L^{1}_t(\dot{B}_{2,1}^{\frac{d}{2}})}+\|\tau^{A}(u^{+},u^{-})\|_{L^{1}_t (\dot{B}_{2,1}^{\frac{d}{2}}\cap\dot{B}^{\frac{d}{2}+1}_{2,1})}\nonumber \\
	&\lesssim \delta_0 t^{A-\frac{1}{2}(\frac{d}{2}-2-\sigma_0)}.
\end{align}
This implies that, for all $t>0$,
\begin{align}
	\|\theta(t)\|&_{\dot{B}_{2,1}^{\frac{d}{2}-2}\cap \dot{B}_{2,1}^{\frac{d}{2}}}+\|(u^{+},u^{-})(t)\|_{\dot{B}_{2,1}^{\frac{d}{2}-2}\cap \dot{B}_{2,1}^{\frac{d}{2}-1}}\lesssim \delta_0 
t^{-\frac{1}{2}(\frac{d}{2}-2-\sigma_0)}.\label{above111}
\end{align}
On the other hand, the left-hand side of \eqref{above111} is uniformly bounded by $\mathcal{E}_0\lesssim \delta_0$ due to \eqref{r1}. Thus, we have
\begin{align}\label{mmmmmmmc}
	\|\theta(t)\|&_{\dot{B}_{2,1}^{\frac{d}{2}-2}\cap \dot{B}_{2,1}^{\frac{d}{2}}}+\|(u^{+},u^{-})(t)\|_{\dot{B}_{2,1}^{\frac{d}{2}-2}\cap \dot{B}_{2,1}^{\frac{d}{2}-1}}\lesssim \delta_0 
(1+t)^{-\frac{1}{2}(\frac{d}{2}-2-\sigma_0)},\quad t>0,
\end{align}
which in particular indicates that
\begin{align}
	&\|\theta(t)\|_{\dot{B}^{\sigma}_{2,1}}\lesssim \delta_0 (1+t)^{-\frac{1}{2}(\frac{d}{2}-2-\sigma_0)},\quad \sigma\in[\frac{d}{2}-2,\frac{d}{2}],\label{4.27}\\
	&\|u^{\pm}(t)\|_{\dot{B}_{2,1}^{\sigma}}\lesssim   \delta_0 (1+t)^{-\frac{1}{2}(\frac{d}{2}-2-\sigma_0)},\quad \sigma\in[\frac{d}{2}-2,\frac{d}{2}-1].\label{4.28}
\end{align}
For any $\sigma\in(\sigma_0,\frac{d}{2}-2]$, it follows from \eqref{sigma_0-functional}, \eqref{mmmmmmmc}, and the interpolation inequality \eqref{inter} that
\begin{equation}
	\begin{aligned}
		\|(\theta,u^{\pm})(t)\|_{ \dot{B}_{2,1}^{\sigma }}&\lesssim\|(\theta,u^{\pm})(t)\|^{\theta}_{\dot{B}_{2,\infty}^{\sigma_0}} 
\|(\theta,u^{\pm})(t)\|^{1-\theta}_{\dot{B}_{2,1}^{\frac{d}{2}-1}}\\
		&\lesssim \delta_0(1+t)^{-\frac{1}{2}(\sigma-\sigma_0)},\quad \sigma\in\left(\sigma_0,\frac{d}{2}-2\right).
	\end{aligned}
\end{equation}

Finally, we establish the decay of higher-order regularity norms for $u^{\pm}$. By applying the maximal regularity estimate for \eqref{weightuaa}, we have
\begin{align}
	\|\tau^{A}u^{\pm}\|_{\widetilde{L}^{\infty}_t(\dot{B}_{2,1}^{\frac{d}{2}+1})}&\lesssim \| \tau^{A-1} u\|_{\widetilde{L}^{\infty}_t( \dot{B}_{2,1}^{\frac{d}{2}-1})}+\|\tau^{A}\theta\|_{\widetilde{L}^{\infty}_t( \dot{B}_{2,1}^{\frac{d}{2}-1})}+\|\tau^{A}F_2^\pm\|_{\widetilde{L}^{\infty}_t(\dot{B}_{2,1}^{\frac{d}{2}-1})}\nonumber.
\end{align}
Here we have
\begin{equation*}
	\begin{aligned}
		\| \tau^{A-1} u^\pm\|_{\widetilde{L}^{\infty}_t( \dot{B}_{2,1}^{\frac{d}{2}-1})}&\lesssim \Big(\|u^\pm\|_{\widetilde{L}^{\infty}_t( \dot{B}_{2,1}^{\frac{d}{2}-1})}\Big)^{\frac{A-1}{A}} 
\Big(\|\tau^{A}u^\pm\|_{\widetilde{L}^{\infty}_t( \dot{B}_{2,1}^{\frac{d}{2}-1})}\Big)^{\frac{1}{A}} \\
		&\lesssim \mathcal{X}_0^{\frac{A-1}{A}} \Big( (1+t)^{A-\frac{1}{2}(\frac{d}{2}-2-\sigma_0)}\delta_0\Big)^{\frac{1}{A}}\lesssim  
(1+t)^{A-\frac{1}{2}(\frac{d}{2}-2-\sigma_0)}\delta_0,
	\end{aligned}
\end{equation*}
and
\begin{align*}
	\|\tau^{A}\theta\|_{\widetilde{L}^{\infty}_t( \dot{B}_{2,1}^{\frac{d}{2}-1})}+\|\tau^{A}F_2^\pm\|_{\widetilde{L}^{\infty}_t(\dot{B}_{2,1}^{\frac{d}{2}-1})}\lesssim  
(1+t)^{A-\frac{1}{2}(\frac{d}{2}-2-\sigma_0)}\delta_0,
\end{align*}
due to \eqref{inequalityh_for_F_3} and \eqref{priorestimate-forEtheta}. This implies that
\begin{align}
	\| u^{\pm}(t)\|^h_{ \dot{B}_{2,1}^{\frac{d}{2}+1}}\lesssim \delta_0(1+t)^{-\frac{1}{2}(\frac{d}{2}-2-\sigma_0)}.
\end{align}
from which and \eqref{4.28} we arrive at
\begin{align}
	\| u^{\pm}(t)\|_{ \dot{B}_{2,1}^{\sigma}}\lesssim \| u^{\pm}(t)\|_{ \dot{B}_{2,1}^{\frac{d}{2}-2}}^{\ell}+\| u^{\pm}(t)\|_{ \dot{B}_{2,1}^{\frac{d}{2}+1}}^{h}\lesssim 
\delta_0(1+t)^{-\frac{1}{2}(\frac{d}{2}-2-\sigma_0)},\label{4300}
\end{align}
for all $\sigma\in(\frac{d}{2}-2,\frac{d}{2}+1]$. By \eqref{mmmmmmmc}--\eqref{4300}, we complete the proof of Theorem \ref{Thm1.3}.  \hfill $\Box$

\section{Time-decay rates (II): Proof of Theorem \ref{Thm1.4}}\label{sectiondecay2}

In this section, our aim is to improve the time decay estimate obtained in Section \ref{sectiondecay1}. This approach is developed by Danchin and Xu \cite{danchin5} pertaining to the 
compressible Navier-Stokes equations. More precisely, we will capture the maximal decay properties and establish optimal rates of the higher-order $\dot{B}^{\sigma}_{2,1}$-norms 
($d/2-2<\sigma<d/2$) of $(\theta, u^\pm)$. This improvement allows us to further study the large time convergence of the non-dissipative components $R^\pm$ toward their non-trivial 
equilibrium states.


Define the time-weighted energy functional
\begin{align}\label{time-energy-forZ}
	\mathcal{Z}(t)&=\sup_{\sigma\in[\sigma_0+\varepsilon,\frac{d}{2}-\var]}\|\langle\tau\rangle^{\frac{\sigma-\sigma_0}{2}}(\theta, u^{+},u^{-})\|_{L^{\infty}(\dot{B}_{2,1}^{\sigma})}^{\ell}\nonumber\\
    &\quad+ 
\|\langle \tau \rangle^{\alpha}\theta \|_{\widetilde{L}^{\infty}(\dot{B}_{2,1}^{\frac{d}{2}})}^{h}+\|\langle\tau\rangle^{\alpha}(u^{+},u^{-}) 
\|_{\widetilde{L}^{\infty}(\dot{B}_{2,1}^{\frac{d}{2}+1})}^{h}
\end{align}
with $\langle t\rangle:=(1+t^2)^{\frac{1}{2}}$ and $\alpha:= \frac{1}{2}(\frac{d}{2}-\sigma_0-\var)$.
In what follows, we need the following inequality repeatedly:
\begin{align}
	\int_0^t \langle t-\tau\rangle^{-\frac{1}{2}(\sigma-\sigma_0')}\langle\tau\rangle^{-\sigma^{*}}d\tau\lesssim\langle t\rangle^{-\frac{1}{2}(\sigma-\sigma_0')}, \mbox{ for } 0\leq 
\frac{1}{2}(\sigma-\sigma_0')\leq \sigma^{*},\ \sigma^{*}>1.\label{inteq}
\end{align}
\subsection{Bounds for the low frequencies}

We first address the low-frequency estimates as follows.
\begin{lemma}
	Let $(n^+, \theta,u^+,u^-)$ be the global classical solution to the Cauchy problem \eqref{equation_R_P_u_pm} given by Theorem \ref{Thm1.1}. Then under the assumptions of Theorem 
\ref{Thm1.4}, for any $\sigma_0'\in[\sigma_0,\frac{d}{2}-2)$, it holds that
	\begin{equation}\label{5.3}
		\begin{aligned}
			 \sup_{\sigma\in[\sigma_0'+\varepsilon,\frac{d}{2}-\var]}\|\langle\tau\rangle^{\frac{\sigma-\sigma_0'}{2}}(\theta, 
u^{+},u^{-})\|_{L^{\infty}(\dot{B}_{2,1}^{\sigma})}^{\ell}&\lesssim
			\|(\theta_0, u_0^{+},u_0^{-})\|_{\dot{B}^{\sigma_0}_{2,\infty}}^{\ell}\\
            &\quad\quad+ \|(n^+,\theta,u^{+},u^{-})\|_{\dot{B}^{\sigma_0+1}_{2,\infty}}\mathcal{Z}(t).
		\end{aligned}
	\end{equation}
\end{lemma}

\begin{proof}
	By applying Gr\"onwall's inequality to \eqref{Lyapunovlow}, we get
	\begin{equation}\nonumber
		\begin{aligned}
			&\|\dot{\Delta}_j (\theta, u^{+},u^{-})\|_{L^2}\lesssim e^{-c2^{2j}t}\|\dot{\Delta}_j (\theta_0, u^{+}_0,u^{-}_0)\|_{L^2}+\int_0^t 
e^{-c2^{2j}(t-\tau)}\|\dot{\Delta}_j(F_1,F_2^+,F_2^-)\|_{L^2}  d\tau,
		\end{aligned}
	\end{equation}
    which implies that, for any $\sigma>\sigma_0$,
    	\begin{equation}\label{4.6}
		\begin{aligned}
			\|(\theta, u^{+},u^{-})\|_{\dot{B}^{\sigma}_{2,1}}^{\ell}&\lesssim t^{-\frac{1}{2}(\sigma-\sigma_0)}\|(\theta_0, u^{+}_0,u^{-}_0)\|_{\dot{B}^{\sigma_0}_{2,\infty}}^{\ell}+\int_0^t (t-\tau)^{-\frac{1}{2}(\sigma-\sigma_0)}
\|(F_1,F_2^+,F_2^-)\|_{\dot{B}^{\sigma_0}_{2,\infty}}^{\ell}  d\tau,
		\end{aligned}
	\end{equation}
where we used the fact that
	\begin{align*}
		\sup_{t\geq0}\sum_{j\in\mathbb{Z}}  t^{\frac{\sigma'}{2}} 2^{j\sigma'} e^{-C' 2^{2j}t}<\infty\quad \text{for all}~\sigma',C'>0.
	\end{align*}
On the other hand, the low-frequency cut-off property indicates that
	\begin{equation}\label{4.8}
		\begin{aligned}
			\|(\theta, u^{+},u^{-})\|_{\dot{B}^{\sigma}_{2,1}}^{\ell}&\lesssim \|(\theta_0, u^{+}_0,u^{-}_0)\|_{\dot{B}^{\sigma}_{2,1}}^{\ell}+\int_0^t 
\|(F_1,F_2^+,F_2^-)\|_{\dot{B}^{\sigma}_{2,1}}^{\ell}  d\tau\\
			&\lesssim \|(\theta_0, u^{+}_0,u^{-}_0)\|_{\dot{B}^{\sigma_0}_{2,\infty}}^{\ell}+\int_0^t \|(F_1,F_2^+,F_2^-)\|_{\dot{B}^{\sigma_0}_{2,\infty}}^{\ell}  d\tau.
		\end{aligned}
	\end{equation}
	Combining \eqref{4.6} and \eqref{4.8}, we deduce that, for any $\sigma>\sigma_0$,
	\begin{equation}\label{5.8}
		\begin{aligned}
			\|(\theta, u^{+},u^{-})\|_{\dot{B}^{\sigma}_{2,1}}^{\ell}&\lesssim (1+t)^{-\frac{1}{2}(\sigma-\sigma_0)}\|(\theta_0, u_0^{+},u_0^{-})\|_{\dot{B}^{\sigma_0}_{2,\infty}}^{\ell}\\
			&\quad+\int_0^t(1+t-\tau)^{-\frac{1}{2}(\sigma-\sigma_0)} \|(F_1,F_2^+,F_2^-)\|_{\dot{B}^{\sigma_0}_{2,\infty}}^{\ell} d\tau.
		\end{aligned}
	\end{equation}
	In view of \eqref{uv3} and \eqref{F1}, we have
	\begin{equation}\label{5.9}
		\begin{aligned}
			\|F_1\|_{\dot{B}^{\sigma_0}_{2,\infty}}^{\ell}&\lesssim \|g_3^\pm(n^{+},\theta)\dive u^{\pm}\|_{\dot{B}^{\sigma_0}_{2,\infty}}^{\ell}\\
			&\quad+\|\rho^- u^+\cdot\nabla n^+\|_{\dot{B}^{\sigma_0}_{2,\infty}}^{\ell}+\|\rho^+ u^-\cdot\nabla n^-\|_{\dot{B}^{\sigma_0}_{2,\infty}}^{\ell}\\
			&\quad+\|(\mathcal{C}(n^++1,\theta+\bar{P})-\mathcal{C}( 1, \bar{P}))\rho^- u^+\cdot\nabla n^+\|_{\dot{B}^{\sigma_0}_{2,\infty}}^{\ell}\\
			&\quad+\|(\mathcal{C}(n^++1,\theta+\bar{P})-\mathcal{C}( 1, \bar{P}))\rho^+ u^-\cdot\nabla n^-\|_{\dot{B}^{\sigma_0}_{2,\infty}}^{\ell}\\
			&\lesssim (1+\|(n^+,\theta)\|_{\dot{B}^{\frac{d}{2}}_{2,1}})\|n^+\|_{\dot{B}^{\sigma_0+1}_{2,\infty}}\|(u^+,u^-)\|_{\dot{B}^{\frac{d}{2}}_{2,1}\cap \dot{B}^{\frac{d}{2}+1}_{2,1}}.
		\end{aligned}
	\end{equation}
	Similar computations show that
	\begin{equation}\label{5.10}
		\begin{aligned}
			&\|F_2^+\|_{\dot{B}^{\sigma_0}_{2,\infty}}^{\ell}+\|F_2^-\|_{\dot{B}^{\sigma_0}_{2,\infty}}^{\ell}\\
			&\quad\lesssim (1+\|(n^+,\theta)\|_{\dot{B}^{\frac{d}{2}}_{2,1}})\|(n^+,\theta,u)\|_{\dot{B}^{\sigma_0+1}_{2,\infty}} 
(\|\theta\|_{\dot{B}^{\frac{d}{2}}_{2,1}}+\|(u^+,u^-)\|_{\dot{B}^{\frac{d}{2}}_{2,1}\cap\dot{B}^{\frac{d}{2}+1}_{2,1}}).
		\end{aligned}
	\end{equation}
	According to the definition \eqref{time-energy-forZ} of $\mathcal{Z}(t)$, it follows that
	\begin{equation}\label{5.11}
		\begin{aligned}
			\|\theta\|_{\dot{B}^{\frac{d}{2}}_{2,1}}+\|(u^+,u^-)\|_{\dot{B}^{\frac{d}{2}}_{2,1}\cap \dot{B}^{\frac{d}{2}+1}_{2,1}}&\lesssim 
\|(\theta,u^{+},u^{-})\|_{\dot{B}^{\frac{d}{2}-\var}_{2,1}}^{\ell}+\|\theta\|_{\dot{B}^{\frac{d}{2}}_{2,1}}^h+\|(u^+,u^-)\|_{\dot{B}^{\frac{d}{2}+1}_{2,1}}^h\\
			&\lesssim \langle t\rangle^{-\alpha}\mathcal{Z}(t).
		\end{aligned}
	\end{equation}
	Since for suitably small $0<\var<<1$ and any $\sigma_0<\sigma\leq \frac{d}{2}-\var$,
	\begin{align}
		\alpha=\frac{1}{2}\left(\frac{d}{2}-\sigma_0-\var\right)>1,\quad 0\leq \frac{1}{2}(\sigma-\sigma_0)\leq \alpha,
	\end{align}
	we conclude from \eqref{r1}, \eqref{inteq} and \eqref{5.9}--\eqref{5.11} that
	\begin{equation}
		\begin{aligned}
			&\int_0^t(1+t-\tau)^{-\frac{1}{2}(\sigma-\sigma_0)} \|(F_1,F_2^+,F_2^-)\|_{\dot{B}^{\sigma_0}_{2,\infty}}^{\ell} d\tau\\
			&\quad \lesssim  \|(n^+,\theta,u^+,u^-)\|_{L^{\infty}_t(\dot{B}^{\sigma_0+1}_{2,\infty})} \mathcal{Z}(t)\int_0^t (1+t-\tau)^{-\frac{1}{2}(\sigma-\sigma_0)} (1+ 
\tau)^{-\alpha}\,d\tau\\
			&\quad \lesssim \|(n^+,\theta,u^+,u^-)\|_{\dot{B}^{\sigma_0+1}_{2,\infty}}\mathcal{Z}(t) (1+t)^{-\frac{1}{2}(\sigma-\sigma_0)}.
		\end{aligned}
	\end{equation}
	This, together with \eqref{5.8}, gives rise to \eqref{5.3}.
\end{proof}

\subsection{Bounds for the high frequencies}
\begin{lemma}\label{lemma5.2}
	Let $(n^+, \theta,u^+,u^-)$ be the global classical solution to the Cauchy problem \eqref{equation_R_P_u_pm} given by Theorem \ref{Thm1.1}. Then under the assumptions of Theorem 
\ref{Thm1.4}, it holds that, for $\alpha=\frac{1}{2}(\frac{d}{2}- \sigma_0 )$,
	\begin{align}\label{5.13}
		\|\langle \tau \rangle^{\alpha}\theta \|_{\widetilde{L}^{\infty}(\dot{B}_{2,1}^{\frac{d}{2}})}^{h}&+\|\langle 
\tau\rangle^{\alpha}(u^+,u^-)\|_{\widetilde{L}^{\infty}(\dot{B}_{2,1}^{\frac{d}{2}+1})}^{h}\lesssim
		\|\theta_0\|_{\dot{B}^{\frac{d}{2}}_{2,1}}^{h}+\| (u_0^{+},u_0^{-})\|_{\dot{B}^{\frac{d}{2}-1}_{2,1}}^{h}+\mathcal{X}(t)\mathcal{Z}(t).
	\end{align}
\end{lemma}

\begin{proof}
	
	By applying Gr\"onwall's inequality to \eqref{Lyapunovhigh}, we get the damping structure in high frequencies for $j\geq-1$:
	\begin{equation}
		\begin{aligned}
			&\|\dot{\Delta}_j (\nabla \theta, u^{+},u^{-})\|_{L^2}\lesssim e^{-c t}\|\dot{\Delta}_j (\nabla\theta_0, u^{+}_0,u^{-}_0)\|_{L^2}+\int_0^t e^{- c(t-\tau)}\|\dot{\Delta}_j(\nabla 
F_1,F_2^+,F_2^-)\|_{L^2}  d\tau,
		\end{aligned}
	\end{equation}
	where $c>0$ is a uniform constant. Therefore, we get
	\begin{equation}\label{equ4-1}
		\begin{aligned}
			&\|\langle\tau\rangle^{\alpha}(\nabla \theta, 
u^{+},u^{-})\|^h_{\widetilde{L}^{\infty}_t(\dot{B}_{2,1}^{\frac{d}{2}-1})}\\
&\quad\lesssim\sum_{j\geq-1}\sup_{\tau\in[0,t]}\langle\tau\rangle^{\alpha} e^{-c\tau} \|(\nabla \theta_0, u_0^{+},u_0^{-})\|_{ 
\dot{B}_{2,1}^{\frac{d}{2}-1}}^h\\
			&\quad\quad+\sum_{j\geq-1}\sup_{\tau\in[0,t]}\langle\tau\rangle^{\alpha}\int_0^{\tau}e^{-(\tau-w)}2^{j(\frac{d}{2}-1)}\|\dot{\Delta}_j(\nabla F_1,F_2^+,F_2^-)\|_{L^2}dw\\
			&\quad\lesssim \|(\nabla \theta_0, u_0^{+},u_0^{-})\|^h_{ \dot{B}_{2,1}^{\frac{d}{2}-1}}+\|\langle\tau\rangle^{\alpha}(\nabla 
F_1,F_2^+,F_2^-)\|_{\widetilde{L}^{\infty}_{t}(\dot{B}^{\frac{d}{2}-1}_{2,1})}^h,
		\end{aligned}
	\end{equation}
	where one has used
	$$
	\int_0^t e^{-c(t-\tau)} \langle \tau\rangle^{-\alpha}d\tau\lesssim \langle t\rangle^{-\alpha}.
	$$
	To control the nonlinear terms  $(\nabla F_1,F_2^+,F_2^-)$, remembering the uniform bounds in \eqref{r1} obtained in Theorem \ref{Thm1.1}, we have
	\begin{equation*}
		\begin{aligned}
			\|\langle\tau\rangle^{\alpha}  F_1\|_{ \widetilde{L}^{\infty}_t(\dot{B}_{2,1}^{\frac{d}{2}})}^h &\lesssim \|\langle\tau\rangle^{\alpha}  g_3^\pm(n^{+},\theta)\dive u^{\pm}\|_{ 
\widetilde{L}^{\infty}_t(\dot{B}_{2,1}^{\frac{d}{2}})}^h+\|\langle\tau\rangle^{\alpha}  \rho^{\mp}u^{\pm}\cdot\nabla n^{\pm}\|_{ \widetilde{L}^{\infty}_t(\dot{B}_{2,1}^{\frac{d}{2}})}^h\\
			&+\|\langle\tau\rangle^{\alpha}(\mathcal{C}(n^++1,\theta+\bar{P})-\mathcal{C}( 1, \bar{P}))\rho^+ u^-\cdot\nabla n^-\|_{ 
\widetilde{L}^{\infty}_t(\dot{B}_{2,1}^{\frac{d}{2}})}^h\\
			&+\|\langle\tau\rangle^{\alpha}(\mathcal{C}(n^++1,\theta+\bar{P})-\mathcal{C}( 1, \bar{P}))\rho^- u^+\cdot\nabla n^+\|_{ 
\widetilde{L}^{\infty}_t(\dot{B}_{2,1}^{\frac{d}{2}})}^h\\
			&\lesssim \|(n^+,\theta)\|_{ \widetilde{L}^{\infty}_t(\dot{B}_{2,1}^{\frac{d}{2}})}\|\langle\tau\rangle^{\alpha} u^{\pm}\|_{ 
\widetilde{L}^{\infty}_t(\dot{B}_{2,1}^{\frac{d}{2}})}\\
			&+(1+\|\theta\|_{ \widetilde{L}^{\infty}_t(\dot{B}_{2,1}^{\frac{d}{2}})})\|n^+\|_{ \widetilde{L}^{\infty}_t(\dot{B}_{2,1}^{\frac{d}{2}+1})}\|\langle\tau\rangle^{\alpha} 
u^{\pm}\|_{ \widetilde{L}^{\infty}_t(\dot{B}_{2,1}^{\frac{d}{2}+1})}\\
			&+\|(n^+,\theta)\|_{ \widetilde{L}^{\infty}_t(\dot{B}_{2,1}^{\frac{d}{2}})}(1+\|\theta\|_{ \widetilde{L}^{\infty}_t(\dot{B}_{2,1}^{\frac{d}{2}})})\|n^+\|_{ 
\widetilde{L}^{\infty}_t(\dot{B}_{2,1}^{\frac{d}{2}+1})}\|\langle\tau\rangle^{\alpha} u^{\pm}\|_{ \widetilde{L}^{\infty}_t(\dot{B}_{2,1}^{\frac{d}{2}+1})}\\
			&\lesssim \mathcal{X}(t)\|\langle\tau\rangle^{\alpha} (u^{+},u^{-})\|_{ \widetilde{L}^{\infty}_t(\dot{B}_{2,1}^{\frac{d}{2}}\cap \dot{B}_{2,1}^{\frac{d}{2}+1})}.
		\end{aligned}
	\end{equation*}
One observes that
 $$
    \|\langle\tau\rangle^{\alpha} u^{\pm}\|_{ 
\widetilde{L}^{\infty}_t(\dot{B}_{2,1}^{\frac{d}{2}})}^\ell\leq \|\langle\tau\rangle^{\alpha} u^{\pm}\|_{ 
\widetilde{L}^{\infty}_t(\dot{B}_{2,\infty}^{\frac{d}{2}-\var})}^\ell=\|\langle\tau\rangle^{\alpha} u^{\pm}\|_{ 
L^{\infty}_t(\dot{B}_{2,\infty}^{\frac{d}{2}-\var})}^\ell\leq \|\langle\tau\rangle^{\alpha} u^{\pm}\|_{ 
L^{\infty}_t(\dot{B}_{2,1}^{\frac{d}{2}})}^\ell.
$$
Thus, we have
\begin{align*}
		\|\langle\tau\rangle^{\alpha} u^{\pm}\|_{ \widetilde{L}^{\infty}_t(\dot{B}_{2,1}^{\frac{d}{2}})}
		&\lesssim \|\langle\tau\rangle^{\alpha} u^{\pm}\|_{  L^{\infty}_t(\dot{B}_{2,1}^{\frac{d}{2}-\var})}^{\ell}+\|\langle\tau\rangle^{\alpha} u^{\pm}\|^h_{ 
\widetilde{L}^{\infty}_t(\dot{B}_{2,1}^{\frac{d}{2}})}\lesssim\mathcal{Z}(t).
	\end{align*}
Consequently, it holds that
	\begin{align}
		\|\langle\tau\rangle^{\alpha}  F_1\|_{ \widetilde{L}^{\infty}_t(\dot{B}_{2,1}^{\frac{d}{2}})}^h\lesssim \mathcal{X}(t)\mathcal{Z}(t).\label{F1hightime}
	\end{align}
	Arguing similarly as above, one deduces that
	\begin{align}
		&\quad\|\langle\tau\rangle^{\alpha}  (F_2^+,F_2^-)\|_{ \widetilde{L}^{\infty}_t(\dot{B}_{2,1}^{\frac{d}{2}-1})}^h\nonumber\\
		&\lesssim 
(1+\|(n^+,\theta)\|_{\widetilde{L}^{\infty}_{t}(\dot{B}^{\frac{d}{2}}_{2,1})})(\|(n^+,\theta)\|_{\widetilde{L}^{\infty}_{t}(\dot{B}^{\frac{d}{2}}_{2,1})}+\|(u^{+},u^{-})\|_{\widetilde{L}^{\infty}_{t}(\dot{B}^{\frac{d}{2}-1}_{2,1})})\nonumber\\
		&\quad\times (\|\langle \tau\rangle^{\alpha}\theta\|_{\widetilde{L}^{\infty}_{t}(\dot{B}^{\frac{d}{2}}_{2,1})}+\|\langle 
\tau\rangle^{\alpha}(u^{+},u^{-})\|_{\widetilde{L}^{\infty}_{t}(\dot{B}^{\frac{d}{2}}_{2,1}\cap \dot{B}^{\frac{d}{2}+1}_{2,1})}\nonumber\\
		&\lesssim  \mathcal{X}(t)\mathcal{Z}(t)\label{F2F3high}.
	\end{align}
	Substituting \eqref{F1hightime}--\eqref{F2F3high} into \eqref{equ4-1}, we obtain
	\begin{equation}\label{equ4-5}
		\begin{aligned}
			\|\langle\tau\rangle^{\alpha}&(\nabla \theta,u^{+},u^{-})\|_{\widetilde{L}^{\infty}_t(\dot{B}_{2,1}^{\frac{d}{2}-1})}\lesssim \|(\nabla \theta_0, u_0^{+},u_0^{-})\|_{ 
\dot{B}_{2,1}^{\frac{d}{2}-1}}^h+\mathcal{X}(t)\mathcal{Z}(t).
		\end{aligned}
	\end{equation}
	
	Next we will establish the gain of regularity and  decay  for the high frequencies of $u^{\pm}$.  It follows from $\eqref{equation_R_P_u_pm}_3$--$\eqref{equation_R_P_u_pm}_4$  
that
	\begin{align}\label{equ-alpha-upm}
		\partial_t (\langle t\rangle^{\alpha}u^\pm) - \nu_1^{+}\Delta (\langle t\rangle^{\alpha} u^\pm)-\nu_2^+\nabla\dive  (\langle t\rangle^{\alpha}u^\pm)=\alpha \langle 
t\rangle^{\alpha-2}tu^\pm-\beta_2^\pm \nabla (\langle t\rangle^{\alpha}\theta)+\langle t\rangle^{\alpha}F_{2}^\pm.
	\end{align}
	Then, by employing Lemma \ref{lame-system-estimate} for \eqref{equ-alpha-upm} with $\alpha>1$,
	we have
	\begin{equation}\label{5.200}
		\begin{aligned}
			\|\langle \tau\rangle^{\alpha}u^{\pm} \|_{\widetilde{L}^{\infty}(\dot{B}_{2,1}^{\frac{d}{2}+1})}^{h}&\lesssim \|\langle \tau\rangle^{\alpha-1}u^{\pm} 
\|_{\widetilde{L}^{\infty}(\dot{B}_{2,1}^{\frac{d}{2}-1})}^{h}+\|\nabla (\langle \tau\rangle^{\alpha}\theta) \|_{\widetilde{L}^{\infty}(\dot{B}_{2,1}^{\frac{d}{2}-1})}^{h}\\
			&\quad+\|\langle \tau\rangle^{\alpha}(F_{2}^+,F_{2}^-)\|_{\widetilde{L}^{\infty}(\dot{B}_{2,1}^{\frac{d}{2}-1})}^{h}.
		\end{aligned}
	\end{equation}
	Owing to \eqref{equ4-5}, it is easy to see that
	\begin{equation}\label{5.20}
		\begin{aligned}
			\|\langle \tau\rangle^{\alpha-1}u^{\pm} \|_{\widetilde{L}^{\infty}(\dot{B}_{2,1}^{\frac{d}{2}-1})}^{h}\lesssim \|\langle\tau\rangle^{\alpha }u^{\pm} 
\|_{\widetilde{L}^{\infty}(\dot{B}_{2,1}^{\frac{d}{2}-1})}^{h}\lesssim \|(\nabla \theta_0, u_0^{+},u_0^{-})\|_{ \dot{B}_{2,1}^{\frac{d}{2}-1}}^h+\mathcal{X}(t)\mathcal{Z}(t),
		\end{aligned}
	\end{equation}
	and
	\begin{equation}\label{5.22}
		\begin{aligned}
			\|\nabla (\langle \tau\rangle^{\alpha}\theta)\|_{\widetilde{L}^{\infty}(\dot{B}_{2,1}^{\frac{d}{2}-1})}^{h}\lesssim \|\langle\tau\rangle^{\alpha }\theta 
\|_{\widetilde{L}^{\infty}(\dot{B}_{2,1}^{\frac{d}{2} })}^{h}\lesssim \|(\nabla \theta_0, u_0^{+},u_0^{-})\|_{ \dot{B}_{2,1}^{\frac{d}{2}-1}}^h+\mathcal{X}(t)\mathcal{Z}(t).
		\end{aligned}
	\end{equation}
	Using the same argument for the proof of \eqref{F1high}--\eqref{F2F3high}, we obtain
	\begin{equation}\label{5.23}
		\begin{aligned}
			\|\langle \tau\rangle^{\alpha}F_{2}^\pm\|_{\widetilde{L}^{\infty}(\dot{B}_{2,1}^{\frac{d}{2}-1})}^{h}\lesssim \mathcal{X}(t)\mathcal{Z}(t).
		\end{aligned}
	\end{equation}
	It holds from \eqref{5.20}--\eqref{5.22} that
	\begin{equation}\label{equ4-6}
		\begin{aligned}
			\|\langle \tau\rangle^{\alpha}u^{\pm} \|_{\widetilde{L}^{\infty}(\dot{B}_{2,1}^{\frac{d}{2}+1})}^{h}\lesssim  \|\theta_0\|_{\dot{B}^{\frac{d}{2}}_{2,1}}^{h}+\| 
u_0^{\pm}\|_{\dot{B}^{\frac{d}{2}-1}_{2,1}}^{h}+\mathcal{X}(t)\mathcal{Z}(t).
		\end{aligned}
	\end{equation}
	Combining \eqref{equ4-5} with \eqref{equ4-6}, we end up with \eqref{5.13} and complete the proof of Lemma \ref{lemma5.2}.	
\end{proof}

\subsection{Proof of Theorem 1.4}
From the definition of $\mathcal{Z}(t)$, one knows that, for any $\sigma_0<\sigma\leq \frac{d}{2}-\var$,
\begin{equation}
	\begin{aligned}
		\|(\theta,u^{+},u^{-})(t)\|_{\dot{B}^{\sigma}_{2,1}}&\lesssim 
\|(\theta,u^{+},u^{-})(t)\|_{\dot{B}^{\sigma}_{2,1}}^{\ell}+\|\theta(t)\|_{\dot{B}^{\frac{d}{2}}_{2,1}}^h+\|u^\pm(t)\|_{\dot{B}^{\frac{d}{2}+1}_{2,1}}\\
		&\lesssim \Big(\langle t\rangle^{-\frac{1}{2}(\sigma-\sigma')}+\langle t\rangle^{-\alpha}\Big)\mathcal{Z}(t)\lesssim \langle t\rangle^{-\frac{1}{2}(\sigma-\sigma_0)}\mathcal{Z}(t).
	\end{aligned}
\end{equation}
In addition, it also follows that
\begin{align}
	\|\theta(t)\|_{\dot{B}^{\frac{d}{2}}_{2,1}}&\lesssim \|\theta(t)\|_{\dot{B}^{\frac{d}{2}-\var}_{2,1}}^{\ell}+\|\theta(t)\|_{\dot{B}^{\frac{d}{2}}_{2,1}}^h\lesssim \langle 
t\rangle^{-\frac{1}{2}(\frac{d}{2}-\sigma_0-\var)}\mathcal{Z}(t),\label{5.26}\\
	\|u^\pm(t)\|_{\dot{B}^{\frac{d}{2}}_{2,1}\cap \dot{B}^{\frac{d}{2}+1}_{2,1}}&\lesssim 
\|u^\pm(t)\|_{\dot{B}^{\frac{d}{2}-\var}_{2,1}}^{\ell}+\|u^\pm(t)\|_{\dot{B}^{\frac{d}{2}}_{2,1}}^h\lesssim \langle 
t\rangle^{-\frac{1}{2}(\frac{d}{2}-\sigma_0-\var)}\mathcal{Z}(t).\label{5.27}
\end{align}
Therefore, in order to justify the desired decay estimates, it suffices to bound the functional $\mathcal{Z}(t)$. To this end, collecting the weighted estimates \eqref{5.3} and 
\eqref{5.13} and recalling the definition \eqref{time-energy-forZ}, we obtain
\begin{equation}
	\begin{aligned}
		\mathcal{Z}(t)\lesssim \delta_0 +(\|(n^+,\theta,u^\pm)\|_{\dot{B}^{\sigma_0+1}_{2,\infty}}+\mathcal{X}(t))\mathcal{Z}(t).
	\end{aligned}\label{masggg}
\end{equation}
Due to \eqref{a1} and \eqref{r1}, one knows that $\mathcal{X}(t)\lesssim \mathcal{X}_0<<1$.  To derive the $\widetilde{L}^{\infty}_t(\dot{B}_{2,\infty}^{\sigma_0+1})$ bound of $n^\pm$, we have to require the smallness of $\delta_1$. Indeed, it follows from \eqref{inequ-for-n2+}, \eqref{sigma_0-functional} and \eqref{uv3} that
\begin{equation}
	\begin{aligned}
		\|n^{+}\|_{L^{\infty}_t(\dot{B}_{2,\infty}^{\sigma_0+1})}&\lesssim 
\|n_0^{+}\|_{\dot{B}_{2,\infty}^{\sigma_0+1}}+\|u^{+}\|_{L^1_t(\dot{B}_{2,\infty}^{\sigma_0+2})}+\|n^{\pm}  u^{+}\|_{L^1_t(\dot{B}_{2,\infty}^{\sigma_0+2})}\nonumber\\
		&\lesssim\|n_0^{+}\|_{\dot{B}_{2,\infty}^{\sigma_0+1}}+\|u^{+}\|_{L^1_t(\dot{B}_{2,\infty}^{\sigma_0+2})}+\| n^{+} \|_{\widetilde{L}^{\infty}_t(\dot{B}_{2,1}^{\frac{d}{2}})}\| 
u^{+}\|_{L^1_t(\dot{B}_{2,\infty}^{\sigma_0+2})}\nonumber\\
		&\lesssim (1+\mathcal{X}_0) \delta_1\lesssim \delta_1<<1.
	\end{aligned}
\end{equation}
On the other hand, according to \eqref{r1}, \eqref{sigma_0-functional} and the interpolation \eqref{inter}, one discovers that
\begin{align*}
	\|(\theta, u^+,u^-)\|_{L^{\infty}_{t}(\dot{B}^{\sigma_0+1}_{2,\infty})}&\lesssim \|(\theta, u^+,u^-)\|_{L^{\infty}_{t}(\dot{B}^{\sigma_0}_{2,\infty})}+\|(\theta, 
u^+,u^-)\|_{L^{\infty}_{t}(\dot{B}^{\frac{d}{2}+1}_{2,1})}\lesssim \delta_1<<1.
\end{align*}
Thus, we prove
\begin{align}
	\|(n^+,\theta,u^+,u^-)\|_{\dot{B}^{\sigma_0+1}_{2,\infty}}<<1.\label{5.29}
\end{align}
By \eqref{masggg} and \eqref{5.29}, the uniform bound of $\mathcal{Z}(t)$ is obtained:
\begin{align}
	\mathcal{Z}(t)\lesssim \delta_1.\label{zbound}
\end{align}
In light of \eqref{5.26}, \eqref{5.27} and \eqref{zbound}, \eqref{decay4} and \eqref{decay5} are proved.

Finally, we investigate the large-time behavior of the non-dissipation components $\alpha^\pm \rho^\pm=R^{\pm}=n^{\pm}+1$. We look at the equations of $\alpha^\pm \rho^\pm$:
$$
\partial_t (\alpha^{\pm}\rho^{\pm}) +\dive (\alpha^{\pm}\rho^{\pm}u^{\pm})=0.
$$
Integrating this in time yields
\begin{align}
	\alpha^{\pm}\rho^{\pm}=\alpha_0^{\pm}\rho_0^{\pm}-\int_0^{t }\dive (\alpha^{\pm}\rho^{\pm}u^{\pm})\,dt.\label{5.310}
\end{align}
Subtracting \eqref{5.310} from \eqref{barrho}, we observe
\begin{align}
	\alpha^{\pm}\rho^{\pm}-R_{\infty}^\pm=-\int_{t}^{\infty }\dive (\alpha^{\pm}\rho^{\pm}u^{\pm})\,dt.\label{5.32}
\end{align}
The integration term in \eqref{5.32} decays in time. In fact, making use of \eqref{r1}, \eqref{decay5} and \eqref{uv1}, we obtain
\begin{equation}\label{5.36}
	\begin{aligned}
		&\quad\|\dive (\alpha^{\pm}\rho^{\pm}u^{\pm})\|_{\dot{B}^{\frac{d}{2}}_{2,1}}\\
		&\lesssim \|u^{\pm}\|_{\dot{B}^{\frac{d}{2}+1}_{2,1}}+\|(\alpha^{\pm}\rho^{\pm}-1)u^{\pm}\|_{\dot{B}^{\frac{d}{2}+1}_{2,1}}\\
		&\lesssim 
\|u^{\pm}\|_{\dot{B}^{\frac{d}{2}+1}_{2,1}}+\|\alpha^{\pm}\rho^{\pm}-1\|_{\dot{B}^{\frac{d}{2}}_{2,1}}\|u^{\pm}\|_{\dot{B}^{\frac{d}{2}+1}_{2,1}}+\|\alpha^{\pm}\rho^{\pm}-1\|_{\dot{B}^{\frac{d}{2}+1}_{2,1}}\|u^{\pm}\|_{\dot{B}^{\frac{d}{2}}_{2,1}}\\
		&\lesssim \delta_1 \langle t\rangle^{-\frac{1}{2}(\frac{d}{2}-\sigma_0-\var)}.
	\end{aligned}
\end{equation}
The combination of \eqref{5.32} and \eqref{5.36} gives rise to
\begin{equation}
	\begin{aligned}
		\|(\alpha^{\pm}\rho^{\pm}-R_{\infty}^\pm)(t)\|_{\dot{B}^{\frac{d}{2}}_{2,1}}&\lesssim \int_t^{\infty}\|\dive 
(\alpha^{\pm}\rho^{\pm}u^{\pm})(\tau)\|_{\dot{B}^{\frac{d}{2}}_{2,1}}\,d\tau\\
		&\lesssim \delta_1\int_t^{\infty} \langle \tau\rangle^{-\frac{1}{2}(\frac{d}{2}-\sigma_0-\var)}d\tau\\
		&\lesssim \delta_1 \langle t\rangle^{-\frac{1}{2}(\frac{d}{2}-2-\sigma_0-\var)},
	\end{aligned}
\end{equation}
which gives \eqref{decay6}. Therefore, we have completed the proof of Theorem \ref{Thm1.4}.  \hfill $\Box$

\appendix
\section{Technical lemmas}

We recall some basic properties of Besov spaces and product estimates which will be used repeatedly in this paper. Remark that all the properties remain true for the Chemin--Lerner type 
spaces whose time exponent has to behave according to the H${\rm{\ddot{o}}}$lder inequality for the time variable.

The first lemma is the Bernstein inequalities, which in particular imply that $\dot{\Delta}_{j}u$ is smooth for every $u$ in any Besov spaces so that we can take direct calculations on 
linear equations after applying the operator $\dot{\Delta}_{j}$.
\begin{lemma}
	Let $0<r<R$, $1\leq p\leq q\leq \infty$ and $k\in \mathbb{N}$. For any $u\in L^p$ and $\lambda>0$, it holds
	\begin{equation}\notag
		\left\{
		\begin{aligned}
			&{\rm{Supp}}~ \mathcal{F}(u) \subset \{\xi\in\mathbb{R}^{d}~| ~|\xi|\leq \lambda R\}\Rightarrow \|D^{k}u\|_{L^q}\lesssim\lambda^{k+d(\frac{1}{p}-\frac{1}{q})}\|u\|_{L^p},\\
			&{\rm{Supp}}~ \mathcal{F}(u) \subset \{\xi\in\mathbb{R}^{d}~|~ \lambda r\leq |\xi|\leq \lambda R\}\Rightarrow \|D^{k}u\|_{L^{p}}\sim\lambda^{k}\|u\|_{L^{p}}.
		\end{aligned}
		\right.
	\end{equation}
\end{lemma}

Due to the Bernstein inequalities, the Besov spaces have the following properties.
\begin{lemma}
	The following properties hold{\rm:}
	\begin{itemize}
		\item{} For $s\in\mathbb{R}$, $1\leq p_{1}\leq p_{2}\leq \infty$ and $1\leq r_{1}\leq r_{2}\leq \infty$, it holds
		\begin{equation}\notag
			\begin{aligned}
				\dot{B}^{s}_{p_{1},r_{1}}\hookrightarrow \dot{B}^{s-d(\frac{1}{p_{1}}-\frac{1}{p_{2}})}_{p_{2},r_{2}}.
			\end{aligned}
		\end{equation}
		\item{} For $1\leq p\leq q\leq\infty$, we have the following chain of continuous embedding:
		\begin{equation}\nonumber
			\begin{aligned}
				\dot{B}^{0}_{p,1}\hookrightarrow L^{p}\hookrightarrow \dot{B}^{0}_{p,\infty}\hookrightarrow \dot{B}^{\sigma}_{q,\infty},\quad \sigma=-d(\frac{1}{p}-\frac{1}{q})<0.
			\end{aligned}
		\end{equation}
		\item{} If $p<\infty$, then $\dot{B}^{d/p}_{p,1}$ is continuously embedded in the set of continuous functions decaying to 0 at infinity;
		\item{} The following real interpolation property is satisfied for $1\leq p\leq\infty$, $s_{1}<s_{2}$ and $\theta\in(0,1)$:
		\begin{equation}
			\begin{aligned}
				&\|u\|_{\dot{B}^{\theta s_{1}+(1-\theta)s_{2}}_{p,1}}\lesssim \frac{1}{\theta(1-\theta)(s_{2}-s_{1})}\|u\|_{\dot{B}^{ 
s_{1}}_{p,\infty}}^{\theta}\|u\|_{\dot{B}^{s_{2}}_{p,\infty}}^{1-\theta},\label{inter}
			\end{aligned}
		\end{equation}
		which in particular implies for any $\varepsilon>0$ that
		\begin{equation}\nonumber
			\begin{aligned}
				H^{s+\varepsilon}\hookrightarrow \dot{B}^{s}_{2,1}\hookrightarrow \dot{H}^{s}.
			\end{aligned}
		\end{equation}
		\item{}
		Let $\Lambda^{\sigma}$ be defined by $\Lambda^{\sigma}=(-\Delta)^{\frac{\sigma}{2}}u:=\mathcal{F}^{-}\big{(} |\xi|^{\sigma}\mathcal{F}(u) \big{)}$ for $\sigma\in \mathbb{R}$ and 
$u\in\dot{S}^{'}_{h}$, then $\Lambda^{\sigma}$ is an isomorphism from $\dot{B}^{s}_{p,r}$ to $\dot{B}^{s-\sigma}_{p,r}$;
		\item{} Let $1\leq p_{1},p_{2},r_{1},r_{2}\leq \infty$, $s_{1}\in\mathbb{R}$ and $s_{2}\in\mathbb{R}$ satisfy
		$$
		s_{2}<\frac{d}{p_{2}}\quad\text{\text{or}}\quad s_{2}=\frac{d}{p_{2}}~\text{and}~r_{2}=1.
		$$
		The space $\dot{B}^{s_{1}}_{p_{1},r_{1}}\cap \dot{B}^{s_{2}}_{p_{2},r_{2}}$ endowed with the norm $\|\cdot 
\|_{\dot{B}^{s_{1}}_{p_{1},r_{1}}}+\|\cdot\|_{\dot{B}^{s_{2}}_{p_{2},r_{2}}}$ is a Banach space and has the weak compact and Fatou properties$:$ If $u_{n}$ is a uniformly bounded 
sequence of $\dot{B}^{s_{1}}_{p_{1},r_{1}}\cap \dot{B}^{s_{2}}_{p_{2},r_{2}}$, then an element $u$ of $\dot{B}^{s_{1}}_{p_{1},r_{1}}\cap \dot{B}^{s_{2}}_{p_{2},r_{2}}$ and a subsequence 
$u_{n_{k}}$ exist such that $u_{n_{k}}\rightarrow u $ in $\mathcal{S}'$ and
		\begin{equation}\nonumber
			\begin{aligned}
	\|u\|_{\dot{B}^{s_{1}}_{p_{1},r_{1}}\cap \dot{B}^{s_{2}}_{p_{2},r_{2}}}\lesssim \liminf_{n_{k}\rightarrow \infty} \|u_{n_{k}}\|_{\dot{B}^{s_{1}}_{p_{1},r_{1}}\cap 
\dot{B}^{s_{2}}_{p_{2},r_{2}}}.
			\end{aligned}
		\end{equation}
	\end{itemize}
\end{lemma}

The following Morse-type product estimates in Besov spaces play a fundamental role in the analysis of nonlinear terms:
\begin{lemma}\label{Lemma5-5}
	The following statements hold:
	\begin{itemize}
		\item{} Let $s>0$ and $1\leq p,r\leq \infty$. Then $\dot{B}^{s}_{p,r}\cap L^{\infty}$ is a algebra and
		\begin{equation}\label{uv1}
			\|uv\|_{\dot{B}^{s}_{p,r}}\lesssim \|u\|_{L^{\infty}}\|v\|_{\dot{B}^{s}_{p,r}}+ \|v\|_{L^{\infty}}\|u\|_{\dot{B}^{s}_{p,r}}.
		\end{equation}
		\item{}
		Let $s_{1}$, $s_{2}$ and $p$ satisfy $2\leq p\leq \infty$, $-\frac{d}{p}<s_{1},s_{2}\leq \frac{d}{p}$ and $s_{1}+s_{2}>0$. Then we have
		\begin{equation}\label{uv2}
			\|uv\|_{\dot{B}^{s_{1}+s_{2}-\frac{d}{p}}_{p,1}}\lesssim \|u\|_{\dot{B}^{s_{1}}_{p,1}}\|v\|_{\dot{B}^{s_{2}}_{p,1}}.
		\end{equation}
		\item{} Assume that $s_{1}$, $s_{2}$ and $p$ satisfy $2\leq p\leq \infty$, $-\frac{d}{p}<s_{1}\leq \frac{d}{p}$, $-\frac{d}{p}\leq s_{2}<\frac{d}{p}$ and $s_{1}+s_{2}\geq0$. Then 
it holds
		\begin{equation}\label{uv3}
			\|uv\|_{\dot{B}^{s_{1}+s_{2}-\frac{d}{p}}_{p,\infty}}\lesssim \|u\|_{\dot{B}^{s_{1}}_{p,1}}\|v\|_{\dot{B}^{s_{2}}_{p,\infty}}.
		\end{equation}
	\end{itemize}
\end{lemma}

The following commutator estimates will be used to control some nonlinearities in high frequencies.
\begin{lemma}\label{lemcommutator}
	Let $p\in [1,\infty )$ and $-\frac{d}{p}-1\leq s\leq \frac{d}{p}+1$. Then it holds that
	\begin{align}
		\sum_{j\in \mathbb{Z}}2^{js}\|[v,\dot{\Delta}_j]\partial_i u\|_{L^p}\lesssim\|\nabla v\|_{\dot{B}^{\frac{d}{p}}_{p,1}}\|u\|_{\dot{B}^{s}_{p,1}}, \ i=1,2,\cdots,d,
	\end{align}
	for the commutator $[A,B]=AB-BA.$
\end{lemma}

We state the following result concerning the continuity for composition functions.
\begin{lemma}\label{Lemma5-7}
	Let $G:I\to \mathbb{R}$ be a  smooth function satisfying $G(0)=0.$ For any $1\leq p\leq \infty,$ $s>0$, and $g\in \dot{B}^{s}_{2,1}\cap L^{\infty}$, it holds that $G(g)\in 
\dot{B}^{s}_{p,r}\cap L^{\infty}$ and
	\begin{align}
		\|G(g)\|_{\dot{B}^{s}_{p,r}}\leq C_g\|g\|_{\dot{B}^{s}_{p,r}},\label{F0}
	\end{align}
	where the constant $C_g>0$ depends only on $\|g\|_{L^{\infty}}$, $G'$, $s, $and $d$.
	
	In addition, if $g_{1}, g_{2}\in \dot{B}^{s}_{p,1}\cap L^{\infty}$, then it holds
	\begin{align}
		&\|G(g_{1})-G(g_{2})\|_{\dot{B}^{s}_{p,1}}\leq C_{g_{1},g_{2}}(1+\|(g_{1},g_{2})\|_{\dot{B}^{\frac{d}{2}}_{2,1}})\|g_{1}-g_{2}\|_{\dot{B}^{s}_{p,1}},\quad ~s\in 
(-\frac{d}{2},\frac{d}{2}],\label{F2}
	\end{align}
	where the constant $C_{g_{1},g_{2}}>0$ depends only on $\|(g_{1},g_{2})\|_{L^{\infty}}$, $G'$, $s$, $p$ and $d$.
\end{lemma}

We recall the following lemma about the continuity of the composition of multicomponent functions (see \cite{runst1}[Pages 387-388]). 

\begin{lemma}\label{lemma64}
	Let $m\in \mathbb{N}$, $1\leq p,r\leq \infty$, $s>0$, and $G\in C^{\infty}(\mathbb{R}^{m})$ satisfy $G(0,...,0)=0$. Then for any $f_{i}\in\dot{B}_{p,r}^{s}\cap L^{\infty}$ 
$(i=1,...,m)$, there exists a constant $C_{f}>0$ depending on $\sum_{i=1}^{m}\|f_{i}\|_{L^{\infty}}$, $F$, $s$, $m$ and $d$ such that
	\begin{equation}
		\begin{aligned}
			\|G(f_{1},...,f_{m})\|_{\dot{B}^{s}_{p,r} }\leq C_{f}\sum_{i=1}^{m}\|f_{i}\|_{\dot{B}^{s}_{p,r} }.\label{F1}
		\end{aligned}
	\end{equation}
	Furthermore, for any $f^{1}_{i}, f^{2}_{i}\in\dot{B}^{s}\cap \dot{B}^{\frac{d}{2}}$, we have
	\begin{equation}\label{F3}
		\begin{aligned}
			&\|G(f^{1}_{1},...,f^{1}_{m})-G(f^{1}_{1},...,f^{1}_{m})\|_{\dot{B}^{s}_{p,r} }\leq C^{*}_{f}\big(1+\sum_{i=1}^{m}\|f_{i}\|_{\dot{B}^{s}_{p,r}\cap\dot{B}^{\frac{d}{2}}_{p,r}} 
\big)\sum_{i=1}^{m}\|f_{i}^{1}-f_{i}^{2}\|_{\dot{B}^{s}_{p,r}\cap\dot{B}^{\frac{d}{p}}_{p,r}}.
		\end{aligned}
	\end{equation}
	Here $C^{*}_{f}>0$ depends on $\sum_{i=1}^{m}\|(f^{1}_{i},f^{2}_{i})\|_{L^{\infty}}$, $F$, $s$, $m$ and $d$.
\end{lemma}

Finally, we also have the maximal regularity estimates for the Lam\'e system.
\begin{lemma}\label{lame-system-estimate}
	Let $T>0,$ $\mu>0$, $2\mu+\lambda>0$, $s\in \mathbb{R},$ $1\leq p,r\leq \infty,$ and $1\leq \varrho_2\leq\varrho_1\leq\infty$.  Assume that $u_0\in \dot{B}^s_{p,r}$ and $f\in 
\widetilde{L}^{\varrho_2}(0,T;\dot{B}^{s-2+\frac{2}{\varrho_2}}_{p,r})$ hold. If $u$ is a solution of
	$$\begin{cases}
		\partial_t u-\mu\Delta u-(\mu+\lambda)\nabla\dive u=f, \quad &x\in\mathbb{R}^d,\quad t\in(0,T), \\
		u(x,0)=u_0(x),\quad &x\in\mathbb{R}^d,
	\end{cases}
	$$
	then $u$ satisfies
	\begin{align*}
		\min\{\mu,2\mu+\lambda\}^{\frac{1}{\varrho_1}}\|u\|_{\widetilde{L}_T^{\varrho_1}(\dot{B}^{s 
+\frac{2}{\varrho_1}}_{p,r})}\lesssim\|u_0\|_{\dot{B}^{s}_{p,r}}+\min\{\mu,2\mu+\lambda\}^{\frac{1}{\varrho_2}-1}\|f\|_{\widetilde{L}_T^{\varrho_2}(\dot{B}^{s-2 +\frac{2}{\varrho_2}}_{p,r})}.
	\end{align*}
\end{lemma}

\section*{Acknowledgments}
Ling-Yun Shou is supported by National Natural Science Foundation of China $\#$12301275.  Lei Yao's research is
partially supported by National Natural Science Foundation of China
$\#$12171390,  and the Fundamental Research Funds for the Central Universities under Grant: G2023KY05102.  Yinghui Zhang' research is partially supported by National Natural Science 
Foundation of China
$\#$12271114, Guangxi Natural Science Foundation $\#$2024GXNSFDA010071, $\#$2019JJG110003, Science and Technology Project of Guangxi $\#$GuikeAD21220114, the Innovation Project of Guangxi 
Graduate Education $\#$JGY2023061, Center for Applied Mathematics of Guangxi (Guangxi Normal University) and the Key Laboratory of Mathematical Model and Application (Guangxi Normal 
University), Education Department of Guangxi Zhuang Autonomous Region.
\smallskip

{\bf A conflict of interest statement} On behalf of all authors, the corresponding author states that there is no conflict of interest.
\smallskip

{\bf Data availability statement} Data sharing is not applicable to this article as no datasets were generated or
analyzed during the current study.

\end{document}